\documentclass[reqno]{amsart}
\usepackage{amssymb}
\usepackage{amsmath}
\usepackage{amsthm}
\usepackage[usenames]{color}
\usepackage{graphicx}
\usepackage{cite}
\usepackage{bbm}

\usepackage[colorlinks,linkcolor=blue,anchorcolor=red,citecolor=blue]{hyperref}
\usepackage{marginnote}

\usepackage{color}

\newtheorem{lemma}{Lemma}[section]
\newtheorem{theorem}{Theorem}[section]

\newtheorem{remark}{Remark}[section]
\newtheorem{corollary}{Corollary}[section]

\numberwithin{equation}{section}

\newcommand{\dis}{\displaystyle}

\newcommand{\rmre}{{\rm Re}}

\newcommand{\R}{\mathbb{R}}

\renewcommand{\S}{\mathbb{S}}
\newcommand{\T}{\mathbb{T}}

\newcommand{\CB}{\mathcal{B}}
\newcommand{\CD}{\mathcal{D}}
\newcommand{\CE}{\mathcal{E}}

\newcommand{\CK}{\mathcal{K}}
\newcommand{\CL}{\mathcal{L}}

\newcommand{\ep}{\epsilon}

\newcommand{\na}{\nabla}

\newcommand{\al}{\alpha}

\newcommand{\ga}{\gamma}

\newcommand{\la}{\lambda}
\newcommand{\de}{\delta}
\newcommand{\si}{\sigma}
\newcommand{\pa}{\partial}
\newcommand{\ka}{\kappa}

\newcommand{\Ga}{\Gamma}

\newcommand{\vertiii}[1]{{\left\vert\kern-0.25ex\left\vert\kern-0.25ex\left\vert #1
    \right\vert\kern-0.25ex\right\vert\kern-0.25ex\right\vert}}

\makeatletter
\@namedef{subjclassname@2020}{%
	\textup{2020} Mathematics Subject Classification}
\makeatother

\begin{document}
\title[Non-cutoff Boltzmann equation in the whole space]{Time-velocity decay  of solutions to the non-cutoff Boltzmann equation in the whole space}

\author[C. Cao]{Chuqi Cao}
\address[CQC]{Department of Applied Mathematics, The Hong Kong Polytechnic University, Hong Kong, P.R.~China} 
\email{chuqicao@gmail.com}

\author[R.-J. Duan]{Renjun Duan}
\address[RJD]{Department of Mathematics, The Chinese University of Hong Kong,
	Shatin, Hong Kong, P.R.~China}
\email{rjduan@math.cuhk.edu.hk}

\author[Z.-G. Li]{Zongguang Li}
\address[ZGL]{Department of Mathematics, The Chinese University of Hong Kong,
Shatin, Hong Kong, P.R.~China}
\email{zgli@math.cuhk.edu.hk}

\begin{abstract}
In this paper, we consider the perturbed solutions with polynomial tail in large velocities for the non-cutoff Boltzmann equation near global Maxwellians in the whole space. The global in time existence is proved in the weighted Sobolev spaces and the almost optimal time decay is obtained in Fourier transform based low-regularity spaces. The result shows a time-velocity decay structure of solutions that can be decomposed into two parts. One part allows the slow polynomial tail in large velocities, carries the initial data and enjoys the exponential or arbitrarily large polynomial time decay. The other part, with zero initial data, is dominated by the non-negative definite symmetric dissipation and has the exponential velocity decay but only the slow polynomial time decay.
\end{abstract}

	\date{\today}
	
	

	\maketitle
	\thispagestyle{empty}
	
		\tableofcontents
\section{Introduction}
We consider the following Cauchy problem on the spatially inhomogeneous non-cutoff Boltzmann equation in the whole space:
\begin{eqnarray}\label{BE}
	&\dis \pa_tF+v\cdot \na_x F=Q(F,F),   \quad &\dis F(0,x,v)=F_0(x,v),
\end{eqnarray}
where the unknown $F(t,x,v)\geq0$ stands for the density distribution function of rarefied gas particles with velocity $v\in \R^3$ at position $x\in \R^3$ and time $t> 0$, and initial data $F_0(x,v)\geq 0$ is given. The bilinear Boltzmann collision operator $Q(\cdot,\cdot)$ which acts only on velocity variable is defined by
\begin{align}\notag 
	Q(G,F)(v)=&\int_{\R^3}\int_{\S^2}B(v-u,\si)\left[ G(u')F(v')-G(u)F(v)\right]\,d\si du, 
\end{align}
	where the post-collision velocities $(v',u')$ denote
	\begin{align*}
		v'=\frac{v+u}{2}+\frac{|v-u|}{2}\sigma,\ \
		u'=\frac{v+u}{2}-\frac{|v-u|}{2}\sigma,\quad \sigma\in \mathbb{S}^2.
	\end{align*}
Moreover, we assume that the non-negative Boltzmann collision kernel $B(v-u,\si)$ takes the form:
	\begin{equation*}
			B(v-u,\sigma)=\vert v-u\vert^\gamma b(\cos\theta),\  -3<\gamma\leq1,
	\end{equation*}
	where
	\begin{equation*}
		\cos \theta=\frac{v-u}{\vert v-u\vert}\cdot \sigma, \  0<\theta \leq\pi/2,
	\end{equation*}
	and 
	\begin{equation}
		\sin \theta b(\cos\theta)\sim \theta^{-1-2s}\ \text{as }\theta\to 0,
		\  0<s<1.
		\label{def.s}
	\end{equation}
	
Define the global Maxwellian $\mu$ to be
$$
\mu=\mu(v):=(2\pi)^{-3/2}e^{-|v|^2/2 }.
$$
 Under the perturbation near the global Maxwellian, we look for solutions in the form of
\begin{align}\label{pert}
	F=\mu+g,
\end{align}
for the unknown function $g=g(t,x,v)$. Substituting \eqref{pert} into \eqref{BE}, we can rewrite the Cauchy problem on the Boltzmann equation in terms of $g$ as
\begin{eqnarray}\label{rbe}
	&\pa_t g+v\cdot \nabla_x g =\mathcal{L} g+Q(g,g),   \quad &\dis g(0,x,v)=g_0(x,v):=F_0(x,v)-\mu(v),
\end{eqnarray}
where the linearized collision operator $\mathcal{L}$ is given by
\begin{equation*}
	\mathcal{L}g:= Q(\mu,g)+Q(g,\mu).
\end{equation*}	

Next we introduce some notations which will be frequently used later. In this paper, the velocity weight is denoted by 
$$
\langle v\rangle:=\sqrt{1+|v|^2}.
$$
With this notation, given a function $f=f(v)$, we define the weighted Sobolev norm by
$$
\Vert f\Vert_{H^\ell_{v,k}}^2:=\int_{\mathbb{R}^3} \vert\langle v\rangle^k \langle \nabla\rangle^\ell f(v)\vert^2 dv.
$$
Particularly, if $\ell=0$, we then write $\Vert f\Vert_{L^2_{v,k}}:=\Vert f\Vert_{H^0_{v,k}}$. In case of the symmetric dissipation that is to be specified later, it is important to introduce the dissipation norm as in \cite{AMUXY-2011-CMP} that
\begin{align*}
	\Vert f \Vert_{H^{s*}_{v,k}}^2&=\iiint_{\mathbb{R}^3\times \mathbb{R}^3\times \mathbb{S}^2} B(v-u,\sigma) \mu(u) (\langle v'\rangle^k f(v')-\langle v\rangle^kf(v))^2 dvdu d\sigma \\
	&\quad+\iiint_{\mathbb{R}^3\times \mathbb{R}^3\times \mathbb{S}^2} B(v-u,\sigma) \langle u\rangle^{2k}f(u)^2 (\mu(v')^{1/2}-\mu(v)^{1/2})^2 dvdu d\sigma,
\end{align*}
where $0<s<1$ is given as in \eqref{def.s} and we write $\Vert f \Vert_{H^{s*}_v}=\Vert f \Vert_{H^{s*}_{v,0}}$ for $k=0$.
Given a function $f=f(x,v)$, we also denote
\begin{align*}
	\|f\|^2_{H^r_xH^\ell_{v,k}}:=\int_{\R^3}\|\langle\nabla_x\rangle^r f(x,\cdot)\|^2_{H^\ell_{v,k}}dx,
\end{align*}
and write $\|f\|_{L^2_{x,v}}:=\|f\|_{H^0_xH^0_{v,0}}$.

We shall prove the global existence in weighted Sobolev spaces. The corresponding norms are defined by
\begin{align}
	\|f\|^2_{X_k}&:=\|f\|^2_{L^2_xL^2_{v,k}}+\|\nabla^2_xf\|^2_{L^2_xL^2_{v,k-8}},\label{Xk}\\
	\|f\|^2_{X^*_k}&:=\|f\|^2_{L^2_xH^s_{v,k+\ga/2}}+\|\nabla^2_xf\|^2_{L^2_xH^s_{v,k-8+\ga/2}},\label{X*k}\\
	\|f\|^2_{\CE}&:=\|f\|^2_{L^2_{x,v}}+\|\nabla^2_xf\|^2_{L^2_{x,v}},\label{E}\\
	\|f\|^2_{\CD}&:=\|(\mathbf{I}-\mathbf{P})f\|^2_{L^2_xH^{s*}_v}+\|\nabla_x\mathbf{P}f\|^2_{H^1_xL^2_v}+\|\nabla^2_x(\mathbf{I}-\mathbf{P})f\|^2_{L^2_xH^{s*}_v},\label{D}
\end{align}
 where the macroscopic projection $\mathbf{P}$ is given by
\begin{align*}
	\mathbf{P}f(x,v)=[a^f(x,v)+b^f(x,v)\cdot v + c^f(x,v)(\vert v\vert^2-3)]\mu^{1/2}(v),
\end{align*}
with
\begin{align*}
	a^f&=a^f(x)=\int_{\R^3}f(x,v)\mu^{1/2}(v)dv,\\
	b^f&=b^f(x)=\int_{\R^3}f(x,v)v\mu^{1/2}(v)dv,\\
	c^f&=c^f(x)=\int_{\R^3}f(x,v)\frac{1}{6}(|v|^2-3)\mu^{1/2}(v)dv.
\end{align*}
To study the time decay of solutions, we shall establish estimates in spaces which rely on the Fourier transform $\hat{f}(t,\xi,v)=\mathcal{F}_x [f(t,\cdot,v)](\xi)$ where $\xi\in \R^3$ is the Fourier variable. Let $0<T\leq \infty$, then we also further define the following norms:
\begin{align*}
	\Vert \hat{f}\Vert_{L^p_\xi L^\infty_T L^2_{v,k}} &:=\big(\int_{\R^3} \sup_{0\le t\le T} \|\hat{f}(t,\xi,\cdot)\|^p_{L^2_{v,k}} d\xi\big)^{1/p},\\ 
	\Vert \hat{f}\Vert_{L^p_\xi L^2_T H^{s*}_{v,k}} &:=\Big(\int_{\R^3}  \big(\int^T_0   \Vert \hat{f}(t,\xi,\cdot)\Vert_{H^{s*}_{v,k}}^2  dt\big)^{p/2} d\xi\Big)^{1/p},\\
	\Vert \hat{f}\Vert_{L^p_\xi L^2_T H^\ell_{v,k}} &:=\Big(\int_{\R^3}  \big(\int^T_0   \Vert \hat{f}(t,\xi,\cdot)\Vert_{H^\ell_{v,k}}^2  dt\big)^{p/2} d\xi\Big)^{1/p},
\end{align*}
for $1\leq p<\infty$, and
\begin{align*}
	\Vert \hat{f}\Vert_{L^\infty_\xi L^\infty_T L^2_{v,k}} &:= \sup_{\xi\in\R^3}\sup_{0\le t\le T} \|\hat{f}(t,\xi,\cdot)\|_{L^2_{v,k}} ,\\
	\Vert \hat{f}\Vert_{L^\infty_\xi L^2_T H^{s*}_{v,k}}& :=\sup_{\xi\in\R^3}\Big(\int^T_0   \Vert \hat{f}(t,\xi,\cdot)\Vert_{H^{s*}_{v,k}}^2  dt\Big)^{1/2},\\
	\Vert \hat{f}\Vert_{L^\infty_\xi L^2_T H^\ell_{v,k}} &:=\sup_{\xi\in\R^3}\Big(\int^T_0   \Vert \hat{f}(t,\xi,\cdot)\Vert_{H^\ell_{v,k}}^2  dt\Big)^{1/2} .
\end{align*}

The main results of this paper are stated below.  First, we are concerned with the global existence of the perturbation $g$ in the space $X_k$.
	\begin{theorem}\label{GE}
		Let $-3<\ga\leq1$, $0<s<1$, $\gamma+2s >-1$ and $k \geq25$. There are $\ep_0>0$ and $C>0$ such that  if it holds that $F_0(x,v)=\mu(v)+g_0(x,v)\geq 0$ with $g_0\in X_k$ satisfying
		\begin{align}\notag
			\|g_0\|_{X_{k}}\leq \ep_0,
		\end{align}
		then the Cauchy problem on the Boltzmann equation \eqref{BE} or \eqref{rbe} admits a unique global solution $F(t,x,v)=\mu(v)+g(t,x,v)\geq 0$ with $g\in L^\infty (0,\infty;X_k)$ satisfying the estimate
		\begin{align}\label{Global}
			\|g(t)\|_{X_k}\leq C\|g_0\|_{X_{k}},
		\end{align}
		for all $t\geq0$.
		Moreover, $g$ can be decomposed into 
\begin{equation}
\label{decomp}
g(t,x,v)=g_1(t,x,v)+\sqrt{\mu}g_2(t,x,v),
\end{equation}	
where $g_1(t,x,v)$ and $g_2(t,x,v)$ with $g_1(0,x,v)=g_0(x,v)$ and $g_2(0,x,v)\equiv 0$ enjoy the following more precise estimates.
If $0\leq\ga\leq1$, then there are constants $\la>0$ and $C>0$ such that
		\begin{align}\label{GEH}
			&\sup_{0\leq s \leq t}\{\|e^{\la s}g_1(s)\|^2_{{X_k}}+\|g_2(s)\|^2_{\CE}\}+\int_0^t\{\|e^{\la s}g_1(s)\|^2_{{X^*_k}}+\|g_2(s)\|^2_{\CD}\}ds\leq C\|g_0\|^2_{{X_k}},	
		\end{align}
		for all $t\geq 0$.
	If $-3<\ga<0$, then for any $1<\rho <\frac{k-22}{|\ga|}$, there is a constant $C>0$ such that
	\begin{align}\label{GES}
		&\sup_{0\leq s \leq t}\{\|g_1(s)\|^2_{{X_k}}+\|(1+s)^\rho g_1(s)\|^2_{{X_{k+\rho\ga}}}+\|g_2(s)\|^2_{\CE}\}\notag\\&\qquad+\int_0^t\{\|g_1(s)\|^2_{{X^*_k}}+\|(1+s)^\rho g_1(s)\|^2_{{X^*_{k+\rho\ga}}}+\|g_2(s)\|^2_{\CD}\}ds\notag\\
		\leq& C\|g_0\|^2_{{X_k}},	
	\end{align}
	for all $t\geq 0$.
	\end{theorem}
	
Notice that in the decomposition \eqref{decomp} the first part $g_1(t,x,v)$ carries the full initial data $g_0(x,v)$ and admits the polynomial decay in large velocity and exponential or arbitrarily large polynomial decay in large time, while the second part $\sqrt{\mu}g_2(t,x,v)$ with zero initial data has the exponential decay in large velocity in the weighted sense.

Furthermore, the long time behavior of the second part $\sqrt{\mu}g_2(t,x,v)$, in particular, the time-decay of $g_2(t,x,v)$, is indeed dominated by the interplay between the transport operator and the degenerate non-negative definite symmetric operator in case of the whole space, which gives only the polynomial rate in large time similar to the one for the heat equation. Motivated by \cite{DSY}, we adopt a Fourier transform based approach to treat the slow time-decay of $g_2(t,x,v)$ or equivalently the original function $g(t,x,v)$ in the low-regularity function space $L^1_\xi  L^2_{v,k}$. To obtain an explicit rate in large time, we additionally require $\widehat{g_0}$ for initial data $g_0(x,v)$ to belong to the space $L^p_\xi L^2_{v,k}$ with $3/2<p\le \infty$. Precisely, the main result is stated as follows.     
	
\begin{theorem}\label{decay}
	Let $-3<\ga\leq1$, $0<s<1$, $\gamma+2s >-1$,  $k>22$, $3/2<p\le \infty$ and 
	\begin{align*}
		\sigma=3\Big(1-\frac{1}{p}\Big) -2\ep_1
	\end{align*}
	with $\ep_1>0$ arbitrarily small.
	
	For $0\leq\ga\leq1$, there are $\ep_0>0$ and $C>0$ such that if
	\begin{align}\notag
		\|\widehat{g_0}\|_{L^p_\xi L^2_{v,k}}+\|g_0\|_{X_{k+10}}\leq \ep_0,
	\end{align}
	then the solution to the Cauchy problem on the Boltzmann equation \eqref{BE} or \eqref{rbe} obtained in Theorem \ref{GE} satisfies
	\begin{align}\label{decayh}
		\Vert  \hat{g}(t)\Vert_{L^1_\xi  L^2_{v,k}}\leq C(1+t)^{-\sigma/2}(	\|\widehat{g_0}\|_{L^p_\xi L^2_{v,k}}+\|g_0\|_{X_{k+10}}), 
	\end{align}
	for all $t\geq 0$.

For $-3<\ga<0$, there is $\ep_0>0$ and $C>0$ such that if
\begin{align}\notag
	\|\widehat{g_0}\|_{L^p_\xi L^2_{v,k-\ga}}+\|g_0\|_{X_{k+14}}\leq \ep_0,
\end{align}
	then the solution to the Cauchy problem on the Boltzmann equation \eqref{BE} or \eqref{rbe} obtained in Theorem \ref{GE} satisfies
\begin{align}\label{decays}
	\Vert  \hat{g}(t)\Vert_{L^1_\xi L^2_{v,k}}\leq C(1+t)^{-\sigma/2}(\|\widehat{g_0}\|_{L^p_\xi L^2_{v,k-\ga}}+\|g_0\|_{X_{k+14}}), 
\end{align}
for all $t\geq 0$.
\end{theorem}

\begin{remark}
For the very soft case $-1<\gamma+2s <0, -3/2+s <\gamma<0$, a similar result was obtained in \cite{CG}. 
\end{remark}

For those solutions constructed in Theorem \ref{GE} that allow the velocity perturbation to be only polynomial instead, we say that they are solutions with polynomial tail in large velocity. The study of this kind of solutions in kinetic theory could trace back to Caflisch \cite{Caf} that first proposed the decomposition of solutions of the form $g=\sqrt{M_1}g_1+\sqrt{M_2}g_2$ with two Maxwellians of distinct temperatures and deduced the corresponding $L^2\cap L^\infty$ estimates in velocity variable for both $g_1$ and $g_2$. In fact, one of two exponential weights can be reduced to be only polynomial, cf.~\cite{DL-arma,DL-cmaa, DL-hard}. Furthermore, the velocity-pointwise estimates on the Boltzmann collision term weighted by the polynomial velocity functions were first considered by Arkeryd-Esposito-Pulvirenti \cite[Proposition 3.1]{AEP-87} and the technique has led to many applications \cite{ELM-94, ELM-95} for the fluid dynamic limit problem on the Boltzmann equation.

In the case of homogeneous Boltzmann equation with hard potentials in the torus, Mouhot \cite{Mo} established the  spectral gap-like estimates  to construct solutions in some enlarged function spaces corresponding to the slow velocity decay. Later, Gualdani-Mischler-Mouhot \cite{GMM} gave a systematic study for estimates on the resolvents and semigroups of non-symmetric operators which can be applied to the linearized Boltzmann equation in the torus. With such general theory, Mischler-Mouhot \cite{MiMo} also obtained the solution with polynomial tail for kinetic Fokker-Planck equation. There have been a lot of works on slow decay solutions after those aforementioned results. Interested readers may refer to Carrapatoso-Mischler \cite{CM} for the nonlinear Landau equation with Coulomb potentials in the torus, Briant-Guo \cite{BG} and Briant \cite{Br} for the cutoff Boltzmann equation in general bounded domains, and the references therein. 

In the non-cutoff Boltzmann with hard potentials case, Alonso-Morimoto-Sun-Yang first got the solutions in Sobolev spaces in \cite{AMSY}  and used the Di Giorgi argument to further deduce the well-posedness in $L^2\cap L^\infty$ space in \cite{AMSY-bd}. See also several recent works by the first author of this paper together with his collaborators including \cite{CHJ,Cao-22} for the non-cutoff Boltzmann equation with soft potentials and \cite{Cao} for Boltzmann equation with soft potentials under the cutoff assumption. 

Among the literature mentioned above, the space variable $x$ always lies in the torus or bounded domain. In case of the whole space $\R^3$, due to the unbounded property of the spatial domain, for instance, the Poincar\'e inequality fails, it seems not direct to adopt the same approach for studying the problem. Motivated by \cite{Caf} as well as \cite{DL-arma}, the second and third authors of this paper together with Liu \cite{DLL} first proved the well-posedness of polynomial tail solutions for the Boltzmann equation with cut-off assumption in the whole space. We should also emphasize that the enlarged functional spaces are useful when we consider the kinetic shear flow problem such that one can control the polynomial growth caused by the shear force, see \cite{DL-arma,DL-cmaa, DL-hard} as mentioned before as well as \cite{DLY} and many references therein.

We also recall that many classical results in the symmetrical linearized Boltzmann operator case have been widely obtained, where the solution to the equation is set in the form of $F=\mu+\sqrt{\mu}f$ such that the linearized operator is self-adjoint in $L^2_v$ without any velocity weight function. In fact, for the general mathematical theory of Boltzmann equation, one may refer to \cite{Cer, CIP, Gl} and references therein. In the specific setting under the perturbation near global Maxwellian, Ukai \cite{Uk} first obtained the global existence, uniqueness and large time behavior of solutions for hard potentials with angular cutoff. For the extensive literature, we only list some of the classical ones, which also provide powerful tools for our current work, see \cite{Guo-soft, SS, SG} for soft potentials, \cite{GuoY} for general bounded domains, and \cite{AMUXY11,AMUXY-2012-JFA,GS} for non-cutoff case.

In the current work, the main difficulties are in three aspects. One is that the linearized operator is not self-adjoint on $L^2_v$. Secondly, the spatial domain we consider is the whole space which is unbounded. The last aspect in contrast to \cite{DLL} is caused by the non-cutoff potentials so that we need to carefully select function spaces for existence and time-decay. If one looks at the cut-off case in the torus, then the semigroup $e^{-\CB t}$ generated by the linearized  operator $\CB:=-v\cdot \nabla_x-\CL$ enjoys the exponential decay
for hard potentials and subexponential decay for soft potentials in $L^2_{x,v}(\langle v\rangle^k)$ such that the nonlinear term can be bounded in terms of the Duhamel's formula. Due to the non-symmetric operator and the lack of spectral gap, it is even harder to obtain the polynomial time decay in case of the unbounded domain $\R^3$. Hence, we introduce a structure to decompose the equation to be a coupling system as in \eqref{g1} and \eqref{g2}. As mentioned before, such structure was used in Caflisch's work \cite{Caf} though the velocity perturbation is still exponential. 
We formulate the exact system for the whole space that explains how the slow time-decay in the symmetric case affects the solution, causing a difference with the ones in torus or bounded domain. For the coupling system, we prove that the first unknown $g_1$, which carries the full initial data $g_0$, enjoys polynomial tail in large velocity and fast time decay (either exponential for hard potentials or arbitrarily large polynomial for soft potentials), and the second unknown $g_2$, which satisfies zero initial condition, is dominated by the degenerate symmetric linearized operator and hence has slower time decay due to the interplay with the transport term in $\R^3$. The coupling mechanism makes the time decay of the full solution $g=g_1+\sqrt\mu g_2$ be controlled by the slower one, which corresponds to the symmetric component. This fast vs slow (or exponential vs polynomial) time-velocity decay property was partially used to study the cut-off case in \cite{DLL} and we shall extend it even for the non-cutoff case in this paper. In particular, the faster time decay of $g_1$ can be derived in the existence result and the almost optimal time decay for $g_2$ can be obtained under an extra condition.

Another major difference compared to the cut-off case in \cite{DLL} is caused by the coupling of $g_1$ and $g_2$ in the way that
\begin{align*}
	g_2(t)=\int_0^tU(t-s)\CK_b g_1(s)\,ds,
\end{align*}
where $U$ is the operator for the symmetric linearized problem. Hence we see that if one wants to obtain the $L^2$ or $L^\infty$ decay, one needs some $L^1$ control on $g_1$, since $U$ behaves as a heat diffusive operator. But, the equation for $g_1$ contains the nonlinear collision operator which is difficult to do the $L^1$ estimate for the non-cutoff kernel. To overcome this, a Fourier transform based space $L^1_\xi\cap L^p_\xi$ is applied to deduce the time decay. This space was used to study the low-regularity solution of the non-cutoff Boltzmann equation in \cite{DSY} motivated by the Wiener algebra applied in \cite{DLSY} to study the torus situation. Moreover, since the estimate \eqref{fgh} in the non-cutoff case is not optimal, that is, there is an extra $2s$ in the norm on the right hand side which cannot be bounded by the energy or dissipation functional. This causes that the existence and uniqueness cannot be established in such low-regularity function space in this paper. Instead, we use the Sobolev space with the interpolation technique to take care of this difficulty.

\section{Preliminaries}
In this subsection, we consider the Boltzmann collision operator $Q(f, g)$. 
\begin{lemma}[\cite{H}, Theorem 1.1]\label{upper bound for the Boltzmann operator} Suppose $-3<\ga\leq1$, $0<s<1$, $\gamma+2s >-1$. 
	Let $w_1, w_2 \in \R$, $a, b \in[0, 2s]$ with $w_1+w_2 =\gamma+2s$  and $a+b =2s$. Then there exists a constant $C$, for any functions $f, g, h$ we have \\
	(1) if $\gamma + 2s >0$, then 
	\[
	|(Q(g, h), f)| \le C (\Vert g \Vert_{L^1_{v,\gamma+2s +(-w_1)^++(-w_2)^+}}  +\Vert g \Vert_{L^2_v} ) \Vert h \Vert_{H^a_{v,w_1}}   \Vert f \Vert_{H^b_{v,w_2}} ,
	\]
	(2) if $\gamma + 2s = 0$, then 
	\[
	|(Q(g, h), f)|\le C (\Vert g \Vert_{L^1_{v,w_3}} + \Vert g \Vert_{L^2_v}) \Vert h \Vert_{H^a_{v,w_1}}\Vert f \Vert_{H^b_{v,w_2}} ,
	\]
	where $w_3 = \max\{\delta,(-w_1)^+ +(-w_2)^+ \}$, with $\delta>0$ sufficiently small.\\
	(3) if $-1< \gamma + 2s < 0$, then 
	\[
	|(Q(g, h), f)|\le C  (\Vert g \Vert_{L^1_{v,w_4}} + \Vert g \Vert_{L^2_{v,-(\gamma+2s)}}) \Vert h \Vert_{H^a_{v,w_1}}   \Vert f \Vert_{H^b_{v,w_2}},  
	\]
	where $w_4= \max\{-(\gamma+2s), \gamma+2s +(-w_1)^+ +(-w_2)^+ \}$.
\end{lemma}

\begin{lemma} [\cite{CHJ}, Lemma 2.3]
	Suppose that $-3 <\gamma \le 1$. For any $k \ge 14$, and functions $g, h$, we have
	\begin{align*}
		\quad\,|(Q (h, \mu), g \langle v \rangle^{2k}) |
		&\le   \Vert b(\cos\theta) \sin^{k-\frac {3+\gamma} 2} \frac \theta 2 \Vert_{L^1_\theta}    \Vert h \Vert_{L^2_{v,k+\gamma/2}}\Vert g \Vert_{L^2_{v,k+\gamma/2}} \\
		&\qquad\qquad+ C_k \Vert h \Vert_{L^2_{v,k+\gamma/2-1/2}}\Vert g \Vert_{L^2_{v,k+\gamma/2-1/2}}
		\\
		&\le   \Vert b(\cos\theta) \sin^{k- 2} \frac \theta 2 \Vert_{L^1_\theta}    \Vert h \Vert_{L^2_{v,k+\gamma/2}}\Vert g \Vert_{L^2_{v,k+\gamma/2}} \\
		&\qquad\qquad  + C_k \Vert h \Vert_{L^2_{v,k+\gamma/2-1/2}}\Vert g \Vert_{L^2_{v,k+\gamma/2-1/2}},
	\end{align*}
	for some constant $C_k>0$.  
\end{lemma}

\begin{lemma}[\cite{CHJ}, Theorem 3.1] 
	Suppose that $-3<\gamma\le 1$, $0<s<1$, $\gamma+2s>-1$, $k\ge 14$ and $G= \mu +g \ge 0$. If there exist $A_1,A_2>0$ such that 
	\begin{align*}
		G \ge 0,\quad  \Vert G \Vert_{L^1_v} \ge A_1, \quad \Vert G \Vert_{L^1_{v,2}} +\Vert G \Vert_{L \log L} \le A_2,
	\end{align*}
	where
	$\Vert F \Vert_{L \log L}=\int_{\R^3}|F(v)|\log(1+|F(v)|)dv$, then there exist some constants $\gamma_1, C_k>0$, such that
	\begin{align}\label{Gff}
		(Q(G, f), f \langle v \rangle^{2k} )
		&\le   - \frac {1} {8} \Vert  b(\cos \theta) \sin^2 \frac \theta 2\Vert_{L^1_\theta}\Vert f \Vert_{L^2_{v,k+\gamma/2}}^2  - \gamma_1 \Vert f \Vert_{H^s_{v,k+\gamma/2}}^2 + C_k  \Vert f \Vert_{L^2_v}^2
		\notag\\
		&\quad+C_k\Vert f \Vert_{L^2_{v,14}} \Vert g \Vert_{H^s_{v, k+\gamma/2 }}\Vert f \Vert_{H^s_{v, k + \gamma/2}} +C_k\Vert g \Vert_{L^2_{v,14} } \Vert f \Vert_{H^s_{v, k + \gamma/2}}^2.
	\end{align}
\end{lemma}

Gathering the two lemmas above, we have
\begin{corollary}
	Suppose that $-3<\gamma\le 1, s \in (0, 1), \gamma+2s>-1$, $k\ge 14$ and $G= \mu +g \ge 0$. If there exist $A_1,A_2>0$ such that 
	\begin{align*}
		G \ge 0,\quad  \Vert G \Vert_{L^1_v} \ge A_1, \quad \Vert G \Vert_{L^1_{v,2}} +\Vert G \Vert_{L \log L} \le A_2,
	\end{align*}
	then there exist some constants $\de, C_k>0$, such that
	\begin{align}\label{cor1}
		(L f, f \langle v \rangle^{2k})   +   (Q(g, f), f \langle v \rangle^{2k} )  = &   (Q(\mu+g, f), f \langle v \rangle^{2k} ) + (Q(f, \mu), f \langle v \rangle^{2k} ) 
		\notag\\
		\le &  - \de \Vert f \Vert_{H^s_{v,k+\gamma/2}}^2 + C_k  \Vert f \Vert_{L^2_v}^2+C_k\Vert g \Vert_{L^2_{v,14} } \Vert f \Vert_{H^s_{v, k + \gamma/2}}^2
		\notag\\
		&\qquad +C_k\Vert f \Vert_{L^2_{v,14}} \Vert g \Vert_{H^s_{v, k+\gamma/2 }}\Vert f \Vert_{H^s_{v, k + \gamma/2}}.
	\end{align}
\end{corollary}

The following estimates are about a commutator on the collision operator $Q$ with weight $\langle  v\rangle^k$.

\begin{lemma}[\cite{CHJ}, Lemma 2.4]\label{upper bound for the difference of the Boltzmann operator}
	Suppose $\gamma\in(-3,1],\gamma +2s>-1$ and $k\ge  14$, $g, f, h$ are smooth. Then we have
	\begin{align}
	|(\langle v\rangle^kQ(g, f)-Q(g,\langle v\rangle ^k  f),h\langle v\rangle^k)|&\le C_k\Vert f\Vert_{L^2_{v,14}}\Vert g\Vert_{H^s_{v,k+\gamma/2}}\Vert h\Vert_{H^s_{v,k+\gamma/2}}\notag\\
	&\quad+C_k\Vert g\Vert_{L^2_{v,14}}  \Vert f\Vert_{H^s_{v,k+\gamma/2}}\Vert h\Vert_{H^s_{v,k+\gamma/2}}.\notag
	\end{align}
\end{lemma}

\begin{lemma}
	Suppose $\gamma \in (-3, 1], s \in (0, 1), \gamma+2s >-1$. For any functions $f, g, h$ and $k \ge 14$, we have
	\begin{align}\label{fgh}
		(Q(f, g), h \langle v \rangle^{2k} )\leq &C_k \Vert f \Vert_{L^2_{v,14}} \min\{\|g\|_{H^s_{v,k+\gamma/2}}\|h\|_{H^s_{v,k+\gamma/2+2s}},\,\Vert g \Vert_{H^s_{v,k+\gamma/2+2s }}  \Vert h \Vert_{H^s_{v,k+\gamma/2}}\}\notag\\
		&+  \Vert g \Vert_{L^2_{v,14}}  \Vert f \Vert_{H^s_{v,k+\gamma/2}}  \Vert h \Vert_{H^s_{v,k+\gamma/2}} .
	\end{align}	
\end{lemma}
\begin{proof}
	It is straightforward to see
	$$
	(Q(f, g), h \langle v \rangle^{2k} )\leq |(Q(g,\langle v\rangle ^k  f),h\langle v\rangle^k)|+|(\langle v\rangle^kQ(g, f)-Q(g,\langle v\rangle ^k  f),h\langle v\rangle^k)|.
	$$
	One has from Lemma \ref{upper bound for the Boltzmann operator} that
	$$
	|(Q(g,\langle v\rangle ^k  f),h\langle v\rangle^k)|\leq C_k \Vert f \Vert_{L^2_{v,14}} \min\{\|g\|_{H^s_{v,k+\gamma/2}}\|h\|_{H^s_{v,k+\gamma/2+2s}},\,\Vert g \Vert_{H^s_{v,k+\gamma/2+2s}}  \Vert h \Vert_{H^s_{v,k+\gamma/2}}\},
	$$
	which, together with Lemma \ref{upper bound for the difference of the Boltzmann operator}, yields \eqref{fgh}.
\end{proof}

\section{Global existence}
In this section, we prove Theorem \ref{GE}. First resolve the problem \eqref{rbe} into a coupling system of $g_1=g_1(t,x,v)$ and $g_2=g_2(t,x,v)$ where $g_1$ and $g_2$ satisfy
	\begin{align}
		\pa_t g_1+v\cdot\nabla_x g_1 =&\CL_D g_1+Q(g_1,g_1)+Q(\sqrt{\mu}g_2,g_1)+Q(g_1,\sqrt{\mu}g_2), \label{g1} \\
		\pa_t g_2 +v\cdot\na_x g_2 =&Lg_2+\CL_B g_1+\Ga(g_2,g_2),  \label{g2}
	\end{align}
	with
	\begin{equation}\notag
		g_1(0, x, v)=g_{0}(x, v)=F_0(x,v)-\mu(v),\quad g_2(0,x,v)=0.
    \end{equation}
The linear and nonlinear operators $L$ and $\Gamma$ above are respectively defined by
\begin{equation}
\label{definition L Gamma}
L f  := \frac 1 {\sqrt{\mu} }Q ( \mu ,  \sqrt{\mu} f ) + \frac 1 {\sqrt{\mu} }Q (  \sqrt{\mu} f, \mu ) ,\quad\Gamma( g, f) :=  \frac 1 {\sqrt{\mu} }Q (\sqrt{\mu} g,  \sqrt{\mu} f ),
\end{equation}
we note that $L$ and $\CL$ are different operators. 
The linear operators $\CL_B$ and $\CL_D$ above are respectively defined by
\begin{align}\label{defLB}
	\CL_Bg_1(t,x,v):=\mu^{-1/2}(v)A\chi_M(v) g_1(t,x,v),
\end{align}
and
\begin{align}\label{defLD}
	\CL_D g_1(t,x,v):=(\CL-A\chi_M(v)) g_1(t,x,v),
\end{align}
where $0\le \chi_M \le 1$ is a smooth cutoff function such that for any $M>0$, $\chi_M(v)=1$ if $|v|\leq M$ and $\chi_M(v)=0$ if $|v|\ge2M$. Here the constants $A$ and $M$ are chosen in Lemma \ref{Q x nonlinear estimate 1} such that \eqref{estLD} holds. By setting $g=g_1+\sqrt{\mu}g_2$, it is direct to see  that $g$ is the solution to \eqref{rbe}.
In order to prove the global existence, we first estimate $g_1$ using \eqref{g1}. The following lemmas are important for bounding the nonlinear terms.
\begin{lemma}\label{leLD}Suppose that $-3<\gamma\le 1$, $0<s<1$, $\gamma+2s>-1$ and $G= \mu +g \ge 0$. Then there is a constant $\de>0$ such that if there exist $A_1$, $A_2>0$ satisfying
	\begin{align*}
		G \ge 0,\quad  \Vert G \Vert_{L^1_v} \ge A_1, \quad \Vert G \Vert_{L^1_{v,2}} +\Vert G \Vert_{L \log L} \le A_2,
	\end{align*}
	then for $k\geq14$, there are $A$ and $M$ for the operator $\CL_D$, and a constant $C_k$, such that
	\begin{align}\label{leld}
		(\CL_D f, f \langle v \rangle^{2k})   +   (Q(g, f), f \langle v \rangle^{2k} )  \le&   - \de \Vert f \Vert_{H^s_{v,k+\gamma/2}}^2
		+C_k\Vert f \Vert_{L^2_{v,14}} \Vert g \Vert_{H^s_{v,k+\gamma/2 }}\Vert f \Vert_{H^s_{v,k + \gamma/2}}\notag\\
		& +C_k\Vert g \Vert_{L^2_{v,14} } \Vert f \Vert_{H^s_{v,k + \gamma/2}}^2.
	\end{align}
\end{lemma}
\begin{proof}
We have from the definition of $\CL_D$ in \eqref{defLD} that
\begin{align*}
	&(\CL_D f, f \langle v \rangle^{2k})   +   (Q(g, f), f \langle v \rangle^{2k} )
	\\
	=&(\CL f, f \langle v \rangle^{2k})   +   (Q(g, f), f \langle v \rangle^{2k} )-(A\chi_M(v)f,f),
	\end{align*}	
which, combining with \eqref{cor1}, yields
\begin{align}\label{leLD1}&(\CL_D f, f \langle v \rangle^{2k})   +   (Q(g, f), f \langle v \rangle^{2k} )
	\notag\\\le &  - \de \Vert f \Vert_{H^s_{v,k+\gamma/2}}^2 + C_k  \Vert f \Vert_{L^2_v}^2-CA\|\chi_M(v)f\|_{L^2_v}^2
+C_k\Vert f \Vert_{L^2_{v,14}} \Vert g \Vert_{H^s_{v,k+\gamma/2 }}\Vert f \Vert_{H^s_{v,k + \gamma/2}}\notag\\
&\qquad +C_k\Vert g \Vert_{L^2_{v,14} } \Vert f \Vert_{H^s_{v,k + \gamma/2}}^2. 
\end{align}
	A direct calculation shows that
\begin{align*}
	& \frac{\de}{2} \Vert f \Vert_{H^s_{v,k+\gamma/2}}^2-C_k\Vert f \Vert_{L^2_v}^2+CA\|\chi_M(v)f\|_{L^2_v}^2\\
		\geq& \frac{\de}{2}\int_{\R^3}\int_{\R^3}\big(\langle v\rangle^{2k+\ga}(1-\chi_M(v))+\chi_M(v)\big)|f(t,x,v)|^2dv-C_k\Vert f \Vert_{L^2_v}^2+CA\|\chi_M(v)f\|_{L^2_v}^2\\
	=& \int_{\R^3}\int_{\R^3}\big(\frac{\de}{2}-C_k\langle v\rangle^{-2k-\ga}\big)\langle v\rangle^{2k+\ga}(1-\chi_M(v))|f(t,x,v)|^2dv\\
	&\quad+\int_{\R^3}\int_{\R^3}\big(\frac{\de}{2}-C_k+CA\big)\chi_M(v)|f(t,x,v)|^2dv.
\end{align*}
For suitably large $M$ and $A$, we obtain $(\frac{\de}{2}-C_k\langle v\rangle^{-2k-\ga})(1-\chi_M(v))>0$ and $\frac{\de}{2}-C_k+CA>0$, which leads to
\begin{align}\label{leLD2}
	& \frac{\de}{2} \Vert f \Vert_{H^s_{v,k+\gamma/2}}^2-C_k\Vert f \Vert_{L^2_v}^2+CA\|\chi_M(v)f\|_{L^2_v}^2\geq0.
\end{align}
Therefore, by \eqref{leLD1} and \eqref{leLD2}, one gets
\begin{align*}&(\CL_D f, f \langle v \rangle^{2k})   +   (Q(g, f), f \langle v \rangle^{2k} )
	\\\le &  - \frac{\de}{2} \Vert f \Vert_{H^s_{v,k+\gamma/2}}^2 
	+C_k\Vert f \Vert_{L^2_{v,14}} \Vert g \Vert_{H^s_{v,k+\gamma/2 }}\Vert f \Vert_{H^s_{v,k + \gamma/2}}+C_k\Vert g \Vert_{L^2_{v,14} } \Vert f \Vert_{H^s_{v,k + \gamma/2}}^2. 
\end{align*}
Renaming $\frac{\de}{2}$ to be $\de$ since it is an arbitrarily small constant which is independent of $k$, then \eqref{leld} holds. The proof of Lemma \ref{leLD} is complete.
\end{proof}

\begin{lemma}\label{Q x nonlinear estimate 1}
	Let $\alpha$ be any multi-index such that $|\alpha|=0, 2$ and $G=\mu+g+\sqrt{\mu}h$ for $g=g(t,x,v)$ and $h=h(t,x,v)$. Then there is a constant $\de>0$ such that if there exist $A_1,A_2>0$ satisfying
	\begin{align*}
		G \ge 0,\quad  \Vert G \Vert_{L^1} \ge A_1, \quad \Vert G \Vert_{L^1_{v,2}} +\Vert G \Vert_{L \log L} \le A_2,
	\end{align*}
	then for $k\geq22$, there are $A$ and $M$ for the operator $\CL_D$ and a constant $C_k$, such that
	\begin{align}\label{estLD}
		&\int_{\R^3}\int_{\R^3}   \CL_D  \partial^\alpha_x f +   Q(g,  \partial^\alpha_x f)    + Q(\sqrt{\mu}h,  \partial^\alpha_x f) \langle v \rangle^{2k-8|\alpha |}   \partial^\alpha_x f   dxdv
		\notag\\
		\le&   - \de \Vert f \Vert_{X^*_k}^2
		+C_k\Vert f \Vert_{X_k} \Vert g \Vert_{X^*_k}\Vert f \Vert_{X^*_k} +C_k\Vert g \Vert_{X_k} \Vert f \Vert_{X^*_k}^2\notag\\&+ C_k  \Vert f \Vert_{X_k}   \Vert h \Vert_{\CD}  \Vert f \Vert_{X^*_k}  +C_k\Vert h \Vert_{\CE} \Vert f \Vert_{X^*_k}^2.
	\end{align}
\end{lemma}

\begin{proof}
	For the case $\alpha =0$, applying \eqref{leld} and $L^\infty-L^2-L^2$ H\"older's inequality, one gets
	\begin{align}\label{Ldxv1}
		&(\CL_D f, f \langle v \rangle^{2k})   +   (Q(\sqrt{\mu }h+ g, f), f \langle v \rangle^{2k} )
		\notag\\
		\le & - \de \Vert f \Vert_{L^2_x H^s_{v,k+\gamma/2}}^2
		+C_k\Vert f \Vert_{L^\infty_x L^2_{v,14}} \Vert g \Vert_{L^2_x   H^s_{v,k+\gamma/2 }}\Vert f \Vert_{L^2_x H^s_{v,k + \gamma/2}} 
		\notag\\
		&+C_k\Vert g \Vert_{L^\infty_x L^2_{v,14} } \Vert f \Vert_{L^2_xH^s_{v,k + \gamma/2}}^2+C_k\Vert f \Vert_{L^\infty_x L^2_{v,14}} \Vert \sqrt{\mu} h \Vert_{L^2_x   H^s_{v,k+\gamma/2 }}\Vert f \Vert_{L^2_x H^s_{v,k + \gamma/2}} \notag\\
		&+C_k\Vert \sqrt{\mu} h  \Vert_{L^\infty_x L^2_{v,14} } \Vert f \Vert_{L^2_xH^s_{v,k + \gamma/2}}^2.
	\end{align}
Using Sobolev imbedding, it holds that  
\begin{align}\label{sobolev1}
	\Vert f \Vert_{L^\infty_xL^2_{v,10}} \le C \Vert f \Vert_{H^2_xL^2_{v,10}} \le C\min\{ \Vert f \Vert_{H^2_xL^2_{v,10}}, \Vert f \Vert_{H^2_xH^s_{v,10}} \} \le C \min \{ \Vert f \Vert_{X_k}, \Vert f \Vert_{X^*_k} \}.
\end{align}
It is direct to see
\begin{align}\label{controlmuh}
	\Vert \sqrt{\mu}  h \Vert_{L^2_xH^s_{v,k+\gamma/2}} &\le C \Vert h \Vert_{L^2_xH^{s}_{v,-5}}    \le C \Vert   Ph  \Vert_{L^2_xL^{2}_v}    +    C \Vert   (\mathbf{I}-\mathbf{P})  h  \Vert_{L^2_x H^{s}_{v,-5}}    \notag\\&\le C  \Vert h \Vert_{H^2_xL^2_v}     +    C\Vert   (\mathbf{I}-\mathbf{P})  h  \Vert_{H^2_x H^{s*}_v}   \notag\\
	&\le C \Vert h \Vert_{\CE} +  C\Vert h \Vert_{\CD}.
\end{align}
Substituting \eqref{sobolev1} and \eqref{controlmuh} into \eqref{Ldxv1}, we obtain
	\begin{align}\label{LD0order}
		&(\CL_D f, f \langle v \rangle^{2k})   +   (Q(\sqrt{\mu }h+ g, f), f \langle v \rangle^{2k} )
		\notag\\
		\le &       - \de \Vert f \Vert_{X^*_k}^2
		+C_k\Vert f \Vert_{X_k} \Vert g \Vert_{X^*_k}\Vert f \Vert_{X^*_k} +C_k\Vert g \Vert_{X_k} \Vert f \Vert_{X^*_k}^2 
		\notag\\
		&+C_k \min \{ \Vert f \Vert_{X_k}, \Vert f \Vert_{X^*_k} \}   (\Vert h \Vert_{\CE} +  \Vert h \Vert_{\CD}  )\Vert f \Vert_{X^*_k} +C_k\Vert h \Vert_{\CE} \Vert f \Vert_{X^*_k}^2 
		\notag\\
		\le & - \de\Vert f \Vert_{X^*_k}^2
		+C_k\Vert f \Vert_{X_k} \Vert g \Vert_{X^*_k}\Vert f \Vert_{X^*_k} +C_k\Vert g \Vert_{X_k} \Vert f \Vert_{X^*_k}^2 \notag\\
		&+ C_k  \Vert f \Vert_{X_k}   \Vert h \Vert_{\CD}  \Vert f \Vert_{X^*_k}  +C_k\Vert h \Vert_{\CE} \Vert f \Vert_{X^*_k}^2  .
	\end{align}
	For the case $\alpha = 2$, similarly it holds
	\begin{align*}
		\Vert f \Vert_{H^2_xL^2_{v,10}} \le C\min\{ \Vert f \Vert_{H^2_xL^2_{v,10}}, \Vert f \Vert_{H^2_xH^s_{v,10}} \} \le C \min \{ \Vert f \Vert_{X_k}, \Vert f \Vert_{X^*_k} \},
	\end{align*}
	and
	\begin{align*}
		\Vert \sqrt{\mu}  h \Vert_{H^2_xH^s_{v,k+\gamma/2}} \le C  \Vert h \Vert_{H^2_xL^2_v}     +   C \Vert   (\mathbf{I}-\mathbf{P})  h  \Vert_{H^2_x H^{s*}_v}\leq C   \Vert h \Vert_{\CE} +  C\Vert h \Vert_{\CD}.
	\end{align*}
	We have
	\begin{align}\label{LD2order}
		&(\CL_D  \partial^\alpha_x f,  \partial^\alpha_x f \langle v \rangle^{2k-8|\alpha|})   +   (Q(\sqrt{\mu }h + g,  \partial^\alpha_x f),  \partial^\alpha_x f \langle v \rangle^{2k-8|\alpha| } ) 
		\notag\\
		\le & - \de \Vert f \Vert_{L^2_x H^s_{v,k+\gamma/2-8}}^2
		+C_k\Vert  \partial^\alpha_x f \Vert_{L^2_x L^2_{v,14}} \Vert g \Vert_{L^\infty_x   H^s_{v,k+\gamma/2-8}}\Vert  \partial^\alpha_x f \Vert_{L^2_x H^s_{v,k+\gamma/2-8}}\notag\\
		& +C_k\Vert g \Vert_{L^\infty_x L^2_{v,14} } \Vert  \partial^\alpha_x f \Vert_{L^2_xH^s_{v,k+\gamma/2-8}}^2+C_k\Vert \sqrt{\mu }h \Vert_{L^\infty_x L^2_{v,14} } \Vert  \partial^\alpha_x f \Vert_{L^2_xH^s_{v,k+\gamma/2-8}}^2\notag\\
		& 
		+C_k\Vert  \partial^\alpha_x f \Vert_{L^2_x L^2_{v,14}} \Vert \sqrt{\mu }h \Vert_{L^\infty_x   H^s_{v,k+\gamma/2-8}}\Vert  \partial^\alpha_x f \Vert_{L^2_x H^s_{v,k+\gamma/2-8}}
		\notag\\
		\le &       - \de \Vert f \Vert_{X^*_k}^2
		+C_k\Vert f \Vert_{X_k} \Vert g \Vert_{X^*_k}\Vert f \Vert_{X^*_k} +C_k\Vert g \Vert_{X_k} \Vert f \Vert_{X^*_k}^2\notag\\
		&+ C_k  \Vert f \Vert_{X_k}   \Vert h \Vert_{\CD}  \Vert f \Vert_{X^*_k}  +C_k\Vert h \Vert_{\CE} \Vert f \Vert_{X^*_k}^2.
	\end{align}
	Combining \eqref{LD0order} and \eqref{LD2order}, one gets \eqref{estLD}. The proof of Lemma \ref{Q x nonlinear estimate 1} is complete.
\end{proof}

\begin{lemma}
	Let $k\ge  22$ and $\alpha$ be any multi-index such that $|\alpha|=0, 2$. Then there exists a constant $C_k$ such that
	\begin{align}\label{pagfh}
		\left|\int_{\R^3}\int_{\R^3}   \partial^\alpha_x  Q(g, \sqrt{\mu} f)   \langle v \rangle^{2k-8|\alpha |}\partial^\alpha_x h   dxdv\right|  \le C_k&    \Vert f \Vert_{\CE}  \Vert  g\Vert_{X^*_k} \Vert  h\Vert_{X^*_k}   +    \Vert g \Vert_{X_k}  \Vert f \Vert_{\CD}     \Vert  h\Vert_{X^*_k} 
	\end{align}
\end{lemma}
\begin{proof}
	Using the facts that
	\begin{align*}
		\Vert g \Vert_{H^2_xL^2_{v,10}} \le C \min\{ \Vert g \Vert_{H^2_xL^2_{v,10}}, \Vert g \Vert_{H^2_xH^s_{v,10}} \} \le C\min \{ \Vert g \Vert_{X_k}, \Vert g \Vert_{X^*_k} \},
	\end{align*}
	and
	\begin{align*}
		\Vert \sqrt{\mu}  f \Vert_{H^2_xH^{s}_{v,2l+\gamma/2+2s}}  \le C  \Vert f \Vert_{H^2_x}     +    C\Vert   (\mathbf{I}-\mathbf{P})  f  \Vert_{H^2_x H^{s,*}_v}  \le C  \Vert f \Vert_{\CE} +  C\Vert f \Vert_{\CD},
	\end{align*}
	by Lemma \ref{upper bound for the Boltzmann operator}, we have
	\begin{align*}
		(Q(g, \sqrt{\mu}f), \langle  v \rangle^{2k} h)_{H^2_xL^2_v} \le&C_k   \Vert g \Vert_{H^2_xL^2_{v,5}}     \Vert  \sqrt{\mu}f \Vert_{H^2_xL^2_{v,2k + \gamma+2s}}     \Vert h \Vert_{H^2_xH^{s}_v}  
		\\
		\le& C_k\min \{ \Vert g \Vert_{X_k}, \Vert g \Vert_{X^*_k} \}   ( \Vert f \Vert_{\CE} +  \Vert f \Vert_{\CD}) \Vert  h\Vert_{X^*_k}  \\
		\le&C_k  \Vert f \Vert_{\CE}  \Vert  g\Vert_{X^*_k} \Vert  h\Vert_{X^*_k}   +    C_k\Vert g \Vert_{X_k}  \Vert f \Vert_{\CD}     \Vert  h\Vert_{X^*_k}. 
	\end{align*}
	The lemma is thus proved. 
\end{proof}

\begin{lemma}\label{Q x nonlinear estimate 3}
	Let $k\ge  22$ and $\alpha$ be any multi-index such that $|\alpha|=2$ and $F=\mu+f$. Then there exists a constant $C_k$ such that
	\begin{align}\label{2order}
		&\left|\int_{\R^3}\int_{\R^3}(\partial^\alpha_x  Q(F, g)-Q(F,\partial^\alpha_x g))   \langle v \rangle^{2k-8|\alpha |}\partial^\alpha_x h   dxdv\right|\notag\\
		\leq& C_k(\Vert f \Vert_{X_k} \Vert g \Vert_{X^*_k}+\Vert g \Vert_{X_k} \Vert f \Vert_{X^*_k})\Vert h \Vert_{X^*_k}.
	\end{align}
\end{lemma}
\begin{proof}
	Before the proof of this lemma, we remark that the special weight $\langle v \rangle^{k-4|\alpha|}\partial^\alpha_x$ which is reduced with derivatives is in order to overcome the growth of weight comes from collision operator $Q$, and the reduced coefficient is not essential; in fact, we only need that it is bigger than $4s$. Here, we choose the reduced coefficient as $4$.
	By the Leibniz rule for the bilinear operator and since $\partial^\alpha_x \mu =0$, $|\alpha| \ge 1$, the commutator satisfies
	\[
	\partial^{\alpha}_xQ(F, g)-Q(F,\partial^\alpha_x g)=\sum_{|\alpha_1|\neq  0}C_{\alpha_1,\alpha_2}Q(\partial^{\alpha_1}_x  F ,\partial^{\alpha_2}_x g)=\sum_{|\alpha _1 |   \neq0}C_{\alpha_1,\alpha_2}Q(\partial ^{\alpha_1}_x  f ,\partial^{\alpha_2}_x g),
	\]
	where $\alpha_2 =\alpha-\alpha_1$. For each $(\alpha_1,\alpha_2)    \neq(0,2)$, we have that
	\begin{align*}
		&\quad\left|\int_{\R^3}Q(\partial^{\alpha_1}_x   f,\partial^{\alpha_2}_x   g)   \langle v \rangle^{2k-8|\alpha|  }\partial ^\alpha_x    h  dv\right|	\\
		&\leq \left|\int_{\R^3}Q(\partial^{\alpha_1}_x f,\langle v \rangle^{k-4|\alpha|}\partial^{\alpha_2}_x  g)\langle v \rangle^{k-4|\alpha|       }\partial^\alpha_x h d v  \right|\\
		&\qquad
	+\left|\int_{\R^3}(\langle v \rangle^{k-4|\alpha  |}Q(\partial^{\alpha_1}_x  f,\partial ^{\alpha _2}_x  g)-Q(\partial ^{\alpha _1}_x   f, \langle v \rangle^{k-4|\alpha|  }\partial ^{\alpha_2}_x  g))\langle v \rangle^{k-4|\alpha |}\partial^\alpha_x   h   dv\right|\\
	&:=T_1+T_2.
	\end{align*}
	By Lemma \ref{upper bound for the Boltzmann operator} we have
	\[
	T_1\le C_k\|\partial ^{\alpha _1}_x f\|_{L^2_{v,5}}\|\langle v \rangle^{k-4|\alpha |}\partial^{\alpha_2}_x g\|_{H^s_{v,\gamma/2+2s}}\|\langle v \rangle^{k-4|\alpha|  }\partial^\alpha_x h\|_{H^s_{v,\gamma/2}}.
	\]
	For the term $\int_{\R^3}T_1dx$, we consider two cases: $|\alpha_1|=|\alpha_2|=1$ and $\alpha_2=0$. First, if $|\alpha_1|=|\alpha_2|=1$,  due to Sobolev embeddings, we have
	\[
	H^1(\R^3)\hookrightarrow L^6(\R^3),\quad H^{1/2}(\R^3)\hookrightarrow L^3(\R^3),
	\]
	which implies
	\begin{align*}
		&\int_{\R^3}\|\partial^{\alpha_1}_x f\|^2_{L^2_5}\|\langle v \rangle^{k-4|\alpha|}\partial^{\alpha_2}_x g\|^2_{H^s_{v,\gamma/2+2s}}dx\notag\\
		\le&C\left(\int_{\R^3}\|\langle v \rangle^5\partial^{\alpha_1}_x  f\|^6_{L^2_{ v}}dx\right)^{\frac 1 3}\left(\int_{\T^3}\|\langle v \rangle ^{k-8+2s}\partial ^{\alpha_2}_x   g\|^3_{H^s_{v,\gamma/2}}dx\right)^{\frac 23}
		\\
		\le&C\|\langle v \rangle ^5\langle \nabla_x\rangle^2f\|^2_{L^2_{x, v}  }   \int_{\R^3}\|\langle v \rangle^{k-4\times\frac 32}\langle \nabla_x\rangle^{\frac 3   2}g\|^2_{H^s_{v,\gamma/2}}dx,
	\end{align*}
	which implies that
\begin{align}\label{T1}
	\int_{\R^3}  T_1dx\le C_k\|\langle v \rangle^5\langle \nabla_x\rangle^2f\|_{L^2_{x, v}}   & \left(\sum_{|\alpha|=0, 2}\|\langle v \rangle^{k-4|\alpha|}\partial^\alpha_x g\|_{L^2_xH^s_{v,\gamma/2}}\right)\notag\\
	&\qquad\times\left(\sum_{|\alpha|=0, 2}\|\langle v \rangle^{k-4|\alpha|}\partial^\alpha_x h\|_{L^2_xH^s_{v,\gamma/2}}\right).
\end{align}
	The bound for $T_1$ with $\alpha_2=0$ also follows from the Sobolev embedding $H^{3/2+\delta}(\R^3)\hookrightarrow L^\infty(\R^3)$ with any $\delta>0$. In this case, we have
	\begin{align*}
	\sup_{\R^3}\|\langle v \rangle^{k-4|\alpha|}g\|^2_{H^s_{v,\gamma/2+2s}}&\le \sup_{\R^3}\|\langle      v\rangle^{k-8+2s}g\|^2_{H^s_{v,\gamma/2}}\notag\\
	&\le   \int_{\R^3}\|    \langle v \rangle^{k-4(3/2+\delta)+2s-2+4\delta}\langle D_x\rangle^{3/2+\delta}g\|^2_{H^s_{v,\gamma/2}}dx.
	\end{align*}
	Choosing $\delta=(1-s)/2$, the $T_1$ term is estimated. Then we deal with $T_2$ term. Due to Lemma \ref{upper bound for the difference of the Boltzmann operator} and the same argument before, we can easily derive that
	\begin{align}\label{T2}
		\int_{\R^3}T_2dx\le& C_k\|\langle v \rangle ^{14}\langle \nabla_x\rangle^2f\|_{L^2_{x, v}}\left(\sum_{\alpha}\|\langle v \rangle^{k-4|\alpha|   }\partial^\alpha_x   g\|_{L^2_xH^s_{v,\gamma/2}}\right)\notag\\
		&\qquad\qquad\qquad\qquad\qquad\qquad\times\left(\sum_{\alpha}\|\langle v \rangle^{k-4|\alpha|}\partial^\alpha_x h\|_{L^2_xH^s_{v,\gamma/2}}\right)
		\notag\\
		&+C_k\|\langle v \rangle^{14}\langle \nabla_x\rangle^2g\|_{L^2_{x, v}}\left(\sum_{\alpha}\|\langle v \rangle^{k-4|\alpha|}\partial^\alpha_x   f\|_{L^2_xH^s_{v,\gamma/2}}\right)\notag\\
		&\qquad\qquad\qquad\qquad\qquad\qquad\times\left(\sum_{\alpha}\|\langle v \rangle ^{k-4|\alpha|}\partial^\alpha_x   h\|_{L^2_xH^s_{v,\gamma/2}}\right).
	\end{align}
	We combine \eqref{T1} and \eqref{T2} to get
	\begin{align}
		&\left|\int_{\R^3}\int_{\R^3}(\partial^\alpha_x  Q(F, g)-Q(F,\partial^\alpha_x g))   \langle v \rangle^{2k-8|\alpha |}\partial^\alpha_x h   dxdv\right|
		\notag\\
		\le& C_k\|    \langle v \rangle^{14}   \langle \nabla_x\rangle^2f\|_{L^2_{x, v}}\left(\sum_{\alpha}\|   \langle v \rangle^{k-4|\alpha|}\partial^\alpha_x g\|_{L^2_xH^s_{v,\gamma/2}}\right)\left(\sum_{\alpha}\|   \langle v \rangle^{k-4|\alpha|}\partial^\alpha_x h\|_{L^2_xH^s_{v,\gamma/2}}\right)
		\notag\\
		&+C_k\|   \langle v \rangle^{14}\langle \nabla_x\rangle^2g\|_{L^2_{x, v}}\left(\sum_{\alpha}\|   \langle v \rangle^{k-4|\alpha| }\partial^\alpha_x   f\|_{L^2_xH^s_{v,\gamma/2}}\right)\left(\sum_{\alpha}\|    \langle v \rangle^{k - 4|\alpha|   }\partial^\alpha_x   h  \|_{L^2_xH^s_{v,\gamma/2}}\right),\notag
	\end{align}
	which implies the desired results. It ends the proof of this lemma.
\end{proof}

With the preparations above, we now focus on equation \eqref{g1}.
\begin{lemma}
	Let $-3<\ga\leq1$, $0<s<1$, $\ga+2s>-1$ and $k\ge22$. There exists a constant $C_k$ such that it holds
	\begin{align}\label{estg1}
			&\|g_1(t)\|^2_{{X_k}}+\int_0^t\|g_1(s)\|^2_{{X^*_k}}ds\notag\\
		\leq &\|g_0\|^2_{{X_k}}+C_k\big(\sup_{0\leq s \leq t}\|g_1(s)\|_{{X_k}}\int_0^t\|g_1(s)\|^2_{{X^*_k}}ds\notag\\&\qquad\qquad\qquad+\sup_{0\leq s \leq t}\|g_1(s)\|_{{X_k}}\int_0^t\|g_2(s)\|^2_{{\CD}}ds+\sup_{0\leq s \leq t}\|g_2(s)\|_{{\CE}}\int_0^t\|g_1(s)\|^2_{{X^*_k}}ds\big),
	\end{align}
for $0\leq t< \infty$.
\end{lemma}
\begin{proof}
We apply $\pa^\al_x$ with $|\al|=0,2$ on both sides of \eqref{g1} to get
\begin{align*}
	\pa_t \pa^\al_xg_1+v\cdot\nabla_x \pa^\al_xg_1 =&\pa^\al_x\CL_D g_1+\pa^\al_xQ(g_1,g_1)+\pa^\al_xQ(\sqrt{\mu}g_2,g_1)+\pa^\al_xQ(g_1,\sqrt{\mu}g_2).
\end{align*}
Multiplying the above equation with $\langle v\rangle^{2k-8|\al|} \pa^\al_xg_1$, and taking integration over $\R^3_x\times\R^3_v$, it follows from \eqref{estLD}, \eqref{pagfh} and \eqref{2order} that
\begin{align}\label{estpag1}
	&\frac{1}{2}\frac{d}{dt}\|\pa^\al_xg_1(t)\|^2_{L^2_xL^2_{v,k-4|\al|}}+\de\|\pa^\al_xg_1(t)\|^2_{L^2_xH^s_{v,k-4|\al|+\ga/2}}\notag\\\leq&
	C_k\big(\|g_1(t)\|_{{X_k}}\|g_1(t)\|^2_{{X^*_k}}+ \|g_1(t)\|_{{X_k}}\|g_2(t)\|_{\CD}\|g_1(t)\|_{{X^*_k}}+\|g_2(t)\|_{\CE}\|g_1(t)\|^2_{{X^*_k}}\big).
\end{align}
By taking summation of the resulting equation over $|\al|=0,2$, one has
\begin{align*}
	\frac{d}{dt}\|g_1(t)\|^2_{{X_k}}+\|g_1(t)\|^2_{{X^*_k}}\leq& C_k\big(\|g_1(t)\|_{{X_k}}\|g_1(t)\|^2_{{X^*_k}}+\|g_1(t)\|_{{X_k}}\|g_2(t)\|_{\CD}\|g_1(t)\|_{{X^*_k}}\notag\\
	&\qquad\qquad+ \|g_2(t)\|_{\CE}\|g_1(t)\|^2_{{X^*_k}}\big),
 \end{align*}
which yields \eqref{estg1}.
\end{proof}
Next we prove the time decay estimates $g_1$ in case of hard and soft potentials respectively.
\begin{lemma}
	Let $-3<\ga\le1$, $0<s<1$, $\ga+2s>-1$ and $k\geq22$. 
	
	For $0\leq\ga\leq1$, there exists a constant $\la>0$ such that
	\begin{align}\label{estg1hard}
	&\|e^{\la t}g_1(t)\|^2_{{X_k}}+\int_0^t\|e^{\la s}g_1(s)\|^2_{{X^*_k}}ds\notag\\
\leq& \|g_0\|^2_{{X_k}}+C_k\big(\sup_{0\leq s \leq t}\|e^{\la s}g_1(s)\|_{{X_k}}\notag\\
&\qquad\qquad\qquad+\sup_{0\leq s \leq t}\|g_1(s)\|_{{X_k}}+\sup_{0\leq s \leq t}\|g_2(s)\|_{{\CE}}\big)\int_0^t\|e^{\la s}g_1(s)\|^2_{{X^*_k}}ds\notag\\&+C_k\sup_{0\leq s \leq t}\|e^{\la s}g_1(s)\|_{{X_k}}\int_0^t\|g_1(s)\|^2_{{X^*_k}}ds+C_k\sup_{0\leq s \leq t}\|e^{\la s}g_1(s)\|_{{X_k}}\int_0^t\|g_2(s)\|^2_{{\CD}}ds,
	\end{align}
for $0\le t<\infty$. 

For $-3<\ga<0$ and $\rho>1$, it holds that
\begin{align}\label{estg1soft}
	&\|(1+t)^\rho g_1(t)\|^2_{{X_k}}+\int_0^t\|(1+s)^\rho g_1(s)\|^2_{{X^*_k}}ds\notag\\
	\leq &\|g_0\|^2_{{X_k}}+C_k\int_0^t\| g_1(s)\|^2_{{X^*_{k-\rho\ga}}}ds+C_k\sup_{0\leq s \leq t}\|(1+s)^\rho g_1(s)\|_{{X_k}}\int_0^t\|g_2(s)\|^2_{{\CD}}ds\notag\\&+C_k\big(\sup_{0\leq s \leq t}\|(1+s)^\rho g_1(s)\|_{{X_k}}+\sup_{0\leq s \leq t}\|g_2(s)\|_{{\CE}}\big)\int_0^t\|(1+s)^\rho g_1(s)\|^2_{{X^*_k}}ds,
\end{align}
for $0\le t<\infty$.
	\end{lemma}
\begin{proof}
We first consider $0\leq\ga\leq1$. From \eqref{g1}, for a small constant $\la$ which will be chosen later, we can write the equation of $e^{\la t}g_1$ as follows:
\begin{align*}
	\pa_t e^{\la t}g_1+v\cdot\nabla_x e^{\la t}g_1 -\la e^{\la t}g_1=& \CL_D e^{\la t}g_1+Q(g_1,e^{\la t}g_1)\notag\\&\qquad+Q(\sqrt{\mu}g_2,e^{\la t}g_1)+Q(e^{\la t}g_1,\sqrt{\mu}g_2).
\end{align*}
Similar calculation as in \eqref{estpag1} shows that for $|\al|=0,2$,
\begin{align*}
	&\frac{1}{2}\frac{d}{dt}\|\pa^\al_xe^{\la t}g_1(t)\|^2_{L^2_xL^2_{v,k-4|\al|}}+\de\|\pa^\al_xe^{\la t}g_1(t)\|^2_{L^2_xH^s_{v,k-4|\al|+\ga/2}}\notag\\
	&\qquad-\la\|\pa^\al_xe^{\la t}g_1(t)\|^2_{L^2_xL^2_{v,k-4|\al|}}\notag\\
	\leq&
	C_k\big(\|e^{\la t}g_1(t)\|_{{X_k}}\|e^{\la t}g_1(t)\|_{{X^*_k}}\|g_1(t)\|_{{X^*_k}}+\|g_1(t)\|_{{X_k}}\|e^{\la t}g_1(t)\|^2_{{X^*_k}}\notag\\
	&\qquad\qquad+ \|e^{\la t}g_1(t)\|_{{X_k}}\|g_2(t)\|_{\CD}\|e^{\la t}g_1(t)\|_{{X^*_k}}+\|g_2(t)\|_{\CE}\|e^{\la t}g_1(t)\|^2_{{X^*_k}}\big).
\end{align*}
We choose $\la$ so small that
\begin{align}
	&\frac{1}{2}\frac{d}{dt}\|\pa^\al_xe^{\la t}g_1(t)\|^2_{L^2_xL^2_{v,k-4|\al|}}+\frac{\de}{2}\|\pa^\al_xe^{\la t}g_1(t)\|^2_{L^2_xH^s_{v,k-4|\al|+\ga/2}}\notag\\\leq&
	C_k\big(\|e^{\la t}g_1(t)\|_{{X_k}}\|e^{\la t}g_1(t)\|_{{X^*_k}}\|g_1(t)\|_{{X^*_k}}+\|g_1(t)\|_{{X_k}}\|e^{\la t}g_1(t)\|^2_{{X^*_k}}\notag\\
	&\qquad\qquad+ \|e^{\la t}g_1(t)\|_{{X_k}}\|g_2(t)\|_{\CD}\|e^{\la t}g_1(t)\|_{{X^*_k}}+\|g_2(t)\|_{\CE}\|e^{\la t}g_1(t)\|^2_{{X^*_k}}\big).\notag
\end{align}
Then by taking summation over $|\al|=0,2$, one gets
\begin{align}
	&\frac{d}{dt}\|e^{\la t}g_1(t)\|^2_{{X_k}}+\|e^{\la t}g_1(t)\|^2_{{X^*_k}}\notag\\\leq&
	C_k\big(\|e^{\la t}g_1(t)\|_{{X_k}}\|e^{\la t}g_1(t)\|_{{X^*_k}}\|g_1(t)\|_{{X^*_k}}+\|g_1(t)\|_{{X_k}}\|e^{\la t}g_1(t)\|^2_{{X^*_k}}\notag\\&\qquad\qquad
	+ \|e^{\la t}g_1(t)\|_{{X_k}}\|g_2(t)\|_{\CD}\|e^{\la t}g_1(t)\|_{{X^*_k}}+\|g_2(t)\|_{\CE}\|e^{\la t}g_1(t)\|^2_{{X^*_k}}\big).\notag	
\end{align}
Taking time integration and using Cauchy-Schwarz inequality, \eqref{estg1hard} holds.

We now turn to the case $-3<\ga<0$. It is straightforward to get
\begin{align*}
	&\pa_t (1+t)^\rho g_1+v\cdot\nabla_x (1+t)^\rho g_1 -\rho(1+t)^{\rho-1}g_1\notag\\= &\CL_D (1+t)^\rho g_1+ Q( g_1,(1+t)^\rho g_1)+Q(\sqrt{\mu}g_2,(1+t)^\rho g_1)+Q((1+t)^\rho g_1,\sqrt{\mu}g_2).
\end{align*}
Similar calculation as in \eqref{estpag1} shows that for $|\al|=0,2$ and $\rho>1$,
\begin{align}\label{timeg1}
	&\frac{1}{2}\frac{d}{dt}\|\pa^\al_x(1+t)^\rho g_1(t)\|^2_{L^2_xL^2_{v,k-4|\al|}}+\de\|\pa^\al_x(1+t)^\rho g_1(t)\|^2_{L^2_xH^s_{v,k-4|\al|+\ga/2}}\notag\\
	&\qquad-\rho\|\pa^\al_x(1+t)^{\rho-1/2}g_1(t)\|^2_{L^2_xL^2_{v,k-4|\al|}}\notag\\
	\leq&
	C_k\big(\|(1+t)^\rho g_1(t)\|_{{X_k}}\|(1+t)^\rho g_1(t)\|_{{X^*_k}}\| g_1(t)\|_{{X^*_k}}+\| g_1(t)\|_{{X_k}}\|(1+t)^\rho g_1(t)\|^2_{{X^*_k}} \notag\\&\quad+\|(1+t)^\rho g_1(t)\|_{{X_k}}\|g_2(t)\|_{\CD}\|(1+t)^\rho g_1(t)\|_{{X^*_k}}+\|g_2(t)\|_{\CE}\|(1+t)^\rho g_1(t)\|^2_{{X^*_k}}\big).
\end{align}
We should treat the term $$\de\|\pa^\al_x(1+t)^\rho g_1(t)\|^2_{L^2_xH^s_{v,k-4|\al|+\ga/2}}-\rho\|\pa^\al_x(1+t)^{\rho-1/2}g_1(t)\|^2_{L^2_xL^2_{v,k-4|\al|}}$$ carefully. For the case that $1+t\geq \frac{1}{\ka(1+|v|)^{\ga}}$, one has
\begin{align}\label{larget}
\rho\|\pa^\al_x(1+t)^{\rho-1/2}g_1(t)\|^2_{L^2_xL^2_{v,k-4|\al|}}\leq \ka\rho\|\pa^\al_x(1+t)^{\rho}g_1(t)\|^2_{L^2_xL^2_{v,k-4|\al|+\ga/2}}.
\end{align}
Then we choose $\ka$ so small that
\begin{align}\label{tlarge}
&\de\|\pa^\al_x(1+t)^\rho g_1(t)\|^2_{L^2_xH^s_{v,k-4|\al|+\ga/2}}-\ka\rho\|\pa^\al_x(1+t)^{\rho}g_1(t)\|^2_{L^2_xL^2_{v,k-4|\al|+\ga/2}}\notag\\\leq& \frac{\de}{2}\|\pa^\al_x(1+t)^\rho g_1(t)\|^2_{L^2_xH^s_{v,k-4|\al|+\ga/2}}.
\end{align}
For $1+t\leq \frac{1}{\ka(1+|v|)^{\ga}}$, we have
\begin{align}\label{tsmall}
	\rho\|\pa^\al_x(1+t)^{\rho-1/2}g_1(t)\|^2_{L^2_xL^2_{v,k-4|\al|}}\leq C\|\pa^\al_x(1+|v|)^{-\rho\ga+\ga/2}g_1(t)\|^2_{L^2_xL^2_{v,k-4|\al|}}.
\end{align}
Collecting \eqref{timeg1}, \eqref{tlarge} and \eqref{tsmall}, we obtain \eqref{estg1soft}. 
\end{proof}

To prove the existence of $g$, we also need the stability of $g_2$. Recall \eqref{E} and \eqref{D}.
 
\begin{lemma}
	Let $0<s<1$, $-3<\ga\leq1$ and $\ga+2s>-1$. For any $\rho>1$, it holds that
	\begin{align}\label{estg2}
			&\|g_2(t)\|^2_{\CE}+\int_0^t\|g_2(s)\|^2_{\CD}ds\notag\\
		\leq& C_k\big(\sup_{0\leq s \leq t}\|g_2(s)\|_{\CE}\int_0^t\|g_2(s)\|^2_{\CD}ds+\sup_{0\leq s \leq t}\|g_2(s)\|_{\CE}\sup_{0\leq s \leq t}\|(1+s)^\rho g_1(s)\|_{X_8}\big),
	\end{align}
	for $0\leq t\leq \infty$.
\end{lemma}
\begin{proof}
Applying $\pa^\al$ and multiplying $\pa^\al_xg_2$ to \eqref{g2}, taking inner product over $\R^3_x\times\R^3_v$ and from arguments in \cite{AMUXY-2011-CMP,AMUXY-2012-JFA, GS}, we have
\begin{align}\label{g2g2}
	&\frac{1}{2}\frac{d}{dt}\|g_2(t)\|^2_{\CE}+\de\|g_2(t)\|^2_{\CD}\notag\\\leq& C\|g_2(t)\|_{\CE}\|g_2(t)\|^2_{\CD}+C|(\CL_B g_1(t),g_2(t))|+C\sum_{|\al|=2}|(\CL_B \pa^\al_xg_1(t),\pa^\al_xg_2(t))|.
\end{align}
Using the definitions of $\CL_B$ and $X_k$ norm in \eqref{defLB} and \eqref{Xk}, one has
\begin{align}\label{estLb}
	|(\CL_B g_1,g_2)|\leq& C_k\int_{\R^3}\int_{\R^3}|g_1(t,x,v)g_2(t,x,v)|dxdv\leq C_k\|g_1(t)\|_{L^2_{x,v}}\|g_2(t)\|_{L^2_{x,v}}\notag\\
	\leq& C_k(1+t)^{-\rho}\|(1+t)^\rho g_1(t)\|_{{X_8}}\|g_2(t)\|_{\CE}.
\end{align}
Similarly,
\begin{align}\label{estpaLb}
|(\CL_B \pa^\al_xg_1(t),\pa^\al_xg_2(t))|\leq C_k(1+t)^{-\rho}\|(1+t)^\rho g_1(t)\|_{{X_8}}\|g_2(t)\|_{\CE}.
\end{align}
Combining \eqref{g2g2}, \eqref{estLb} and \eqref{estpaLb}, we obtain
\begin{align*}
	&\frac{1}{2}\frac{d}{dt}\|g_2(t)\|^2_{\CE}+\de\|g_2(t)\|^2_{\CD}\notag\\
	\leq& C\|g_2(t)\|_{\CE}\|g_2(t)\|^2_{\CD}+C_k(1+t)^{-\rho}\|(1+t)^\rho g_1(t)\|_{{X_k}}\|g_2(t)\|_{\CE},
\end{align*}
which gives \eqref{estg2} by taking integration over $[0,t]$.
\end{proof}

With these lemmas in hand, we are ready to prove Theorem \ref{GE} now.
\begin{proof}[Proof of Theorem \ref{GE}]
For $0\leq\ga\leq1$ and $k\geq22$, a suitable linear combination of \eqref{estg1}, \eqref{estg1hard} and \eqref{estg2} shows that 
\begin{align*}
		&\sup_{0\leq s \leq t}\{\|g_1(s)\|^2_{{X_k}}+\|e^{\la s}g_1(s)\|^2_{{X_k}}+\|g_2(t)\|^2_{\CE}\}\notag\\
		&\qquad\qquad+\int_0^t\{\|g_1(s)\|^2_{{X^*_k}}+\|e^{\la s}g_1(s)\|^2_{{X^*_k}}+\|g_2(t)\|^2_{\CD}\}ds\notag\\
	\leq& C_k\|g_0\|^2_{{X_k}}+C_k\sup_{0\leq s \leq t}\{\|g_1(s)\|^2_{{X_k}}+\|e^{\la s}g_1(s)\|^2_{{X_k}}+\|g_2(t)\|^2_{\CE}\}^\frac{1}{2}\notag\\
	&\qquad\qquad\qquad\qquad\times\int_0^t\{\|g_1(s)\|^2_{{X^*_k}}+\|e^{\la s}g_1(s)\|^2_{{X^*_k}}+\|g_2(t)\|^2_{\CD}\}ds.
\end{align*}
Therefore, there exists a constant $\ep_0$ such that if $\|g_0\|^2_{{X_k}}\leq \ep_0$ then 
		\begin{align*}
	&\sup_{0\leq s \leq t}\{\|g_1(s)\|^2_{{X_k}}+\|e^{\la s}g_1(s)\|^2_{{X_k}}+\|g_2(t)\|^2_{\CE}\}\notag\\
	&\qquad+\int_0^t\{\|g_1(s)\|^2_{{X^*_k}}+\|e^{\la s}g_1(s)\|^2_{{X^*_k}}+\|g_2(t)\|^2_{\CD}\}ds\notag\\
	\leq& C\|g_0\|^2_{{X_k}},	
\end{align*}
which yields \eqref{GEH}.
Therefore, we obtain \eqref{Global} for hard potentials by \eqref{GEH} and the fact that
\begin{align}\label{g1+g2}
	\|g(t)\|_{X_k}\leq& \|g_1(t)\|_{X_k}+\|\sqrt{\mu}g_2(t)\|_{X_k}\notag\\
	\leq& \|g_1(t)\|_{X_k}+C_k\|g_2(t)\|_{\CE}.
\end{align}
For soft potentials $-3<\ga<0$, we have from \eqref{estg1}, \eqref{estg1soft} and \eqref{estg2} that for $\rho>1$, it holds
	\begin{align*}
	&\sup_{0\leq s \leq t}\{\|g_1(s)\|^2_{{X_k}}+\|(1+s)^\rho g_1(s)\|^2_{{X_{k+\rho\ga}}}+\|g_2(s)\|^2_{\CE}\}\notag\\&\qquad+\int_0^t\{\|g_1(s)\|^2_{{X^*_k}}+\|(1+s)^\rho g_1(s)\|^2_{{X^*_{k+\rho\ga}}}+\|g_2(s)\|^2_{\CD}\}ds\notag\\
	\leq& C_k\|g_0\|^2_{{X_k}}+C_k\sup_{0\leq s \leq t}\{\|g_1(s)\|^2_{{X_k}}+\|(1+s)^\rho g_1(s)\|^2_{{X_{k+\rho\ga}}}+\|g_2(s)\|^2_{\CE}\}^\frac{1}{2}\notag\\&\qquad\qquad\qquad\qquad\times\int_0^t\{\|g_1(s)\|^2_{{X^*_k}}+\|(1+s)^\rho g_1(s)\|^2_{{X^*_{k+\rho\ga}}}+\|g_2(s)\|^2_{\CD}\}ds.	
\end{align*}
Thus, there exists a constant $\ep_0$ such that if $\|g_0\|^2_{{X_k}}\leq \ep_0$ such that \eqref{GES} holds.
Then \eqref{Global} follows from \eqref{g1+g2} and \eqref{GES}. Also we let $\rho<\frac{k-22}{|\ga|}$ such that the condition $k+\rho\ga\geq22$ is satisfied. Also we require $k>25$ such that $1<\frac{k-22}{|\ga|}$. Thus, the proof of Theorem \ref{GE} is complete.
\end{proof}

\section{Time decay of $g_1$ in $L^p_\xi$}
To obtain the time decay of the solution $g$, we should study $g_1$ and $g_2$ respectively. From our proof of global existence, we expect exponential time decay for $g_1$ when $0\le\ga\le1$, and arbitrarily large polynomial decay when $-3<\ga<0$. Taking the Fourier transform to \eqref{g1} and \eqref{g2} with respect to $x$, it holds
\begin{align}
	\pa_t \widehat{g_1}+iv\cdot\xi \widehat{g_1} =&\CL_D \widehat{g_1}+\widehat{Q}(\widehat{g_1},\widehat{g_1})+\widehat{Q}(\sqrt{\mu}\widehat{g_2},\widehat{g_1})+\widehat{Q}(\widehat{g_1},\sqrt{\mu}\widehat{g_2}), \label{hatg1} \\
	\pa_t \widehat{g_2} +iv\cdot\xi \widehat{g_2} =&L\widehat{g_2}+\CL_B \widehat{g_1}+\hat{\Ga}(\widehat{g_2},\widehat{g_2}),  \label{hatg2}
\end{align}
with the initial data
$$
\widehat{g_1}(0,\xi,v)=\widehat{g_0}(\xi,v)=\hat{F}_0(\xi,v)-\mu(v),\quad \widehat{g_2}(0,\xi,v)=0.
$$
Here we denote
\begin{equation*}
	\widehat{Q}(\hat{f},\hat{g})(\xi,v)=\int_{\mathbb{R}^3}\int_{\mathbb{S}^2} B(v-u,\sigma) [ \hat{f}(u')*_\xi \hat{g}(v')(\xi)- \hat{f}(u)*_\xi \hat{g}(v)(k)]d\sigma du,
\end{equation*}
and
\begin{equation*}
	\hat{\Gamma}(\hat{f},\hat{g})(\xi,v)=\int_{\mathbb{R}^3}\int_{\mathbb{S}^2} B(v-u,\sigma)\mu^{1/2}(u) [ \hat{f}(u')*_\xi \hat{g}(v')(\xi)- \hat{f}(u)*_\xi \hat{g}(v)(k)]d\sigma du,
\end{equation*}
where 
\begin{equation*}
	\hat{f}(u)*_\xi \hat{g}(v)(\xi)=\int_{\mathbb{R}^3} \hat{f}(\xi-\ell,u)\hat{g}(\ell,v)\, d\ell.
\end{equation*}
We mention that the Fourier transform prevents us from directly getting the existence in $L^1_\xi\cap L^p_\xi$. After Fourier transform, it is difficult to obtain inequality like \eqref{Gff}, because now the integral becomes
\begin{align*}
	&\vert (\widehat{Q}(\widehat{g_1},\widehat{g_1})(\xi),\langle v\rangle^{2k}\widehat{g_1}(\xi))_{L^2_v}\vert\\
	=&\vert \int_{\mathbb{R}^3} \int_{\mathbb{R}^3} \int_{\mathbb{S}^{2}} B \int_{\mathbb{R}^3} (\widehat{g_1}(\xi-\ell,v'_*)\hat{g}(\ell,v')-\widehat{g_1}(\xi-\ell,v_*)\widehat{g_1}(\ell,v))d\ell d\sigma dv_* \langle v\rangle^{2k}\bar{\hat{g}}_1(\xi,v)dv\vert.
\end{align*}
We notice the functions $\widehat{g_1}(\ell,v)$ and $\widehat{g_1}(\xi,v)$ are no longer the same function, so \eqref{Gff} is not valid now. There is an additional $\langle v\rangle^{2s}$ in the dissipation norm which can not be absorbed by the left hand side. We need some new approach to control the nonlinear terms by the obtained existence in Theorem \ref{GE}.

\subsection{Estimates for $0\leq\ga\leq1$}
We prove the exponential decay for $g_1$ now.

\begin{lemma}\label{hatg1h}
Let $0\leq\ga\leq1$, $k>14, p \ge 1$ and $g_1$ be a solution to \eqref{hatg1}, then there exists a constant $C_k>0$ such that	
\begin{align}\label{esthatg1h}			
	&\|e^{\la t}\widehat{g_1}\|_{L^p_\xi L^\infty_TL^2_{v,k}}+\Vert e^{\la t}\widehat{g_1}\Vert_{L^p_\xi L^2_TH^s_{v,k+\gamma/2}}\notag\\
	\le&C_k\|\widehat{g_0}\|_{L^p_\xi L^2_{v,k}}+ C_{k} \big(\Vert e^{\la t}\widehat{g_1} \Vert_{L^p_\xi L^\infty_TL^2_{v,k}}+\Vert \widehat{g_2} \Vert_{L^p_\xi L^\infty_TL^2_v} \big)\big(\int^T_0 \Vert e^{\la t}g_1 \Vert^2_{X^*_{k+8+2s}}dt\big)^\frac{1}{2} \notag\\
	&\qquad+C_k\Vert e^{\la t}\widehat{g_1} \Vert_{L^p_\xi L^\infty_TL^2_{v,k}} \Vert \widehat{g_2} \Vert_{L^1_\xi L^2_TH^{s*}_v}.
\end{align}
\end{lemma}

\begin{proof}
	A direct calculation shows that
	\begin{align*}
		\pa_t e^{\la t}\widehat{g_1}+iv\cdot\xi e^{\la t}\widehat{g_1}-\la e^{\la t}\widehat{g_1} =&\CL_D e^{\la t}\widehat{g_1}+\widehat{Q}(e^{\la t}\widehat{g_1},\widehat{g_1})\notag\\
		&+\widehat{Q}(\sqrt{\mu}\widehat{g_2},e^{\la t}\widehat{g_1})+\widehat{Q}(e^{\la t}\widehat{g_1},\sqrt{\mu}\widehat{g_2}). 
	\end{align*}
	We multiply $\langle v\rangle^{2k}e^{\la t}\bar{\hat{g}}_1$ to \eqref{hatg1} and take real part and integrate over $[0,T]\times\R^3_v$ to get
	\begin{align}\label{I1234}			&\frac{1}{2\sqrt{2}}\|e^{\la t}\widehat{g_1}(t,\xi)\|_{L^2_{v,k}}-\sqrt{\la}\big(\int^T_0\|e^{\la t}\widehat{g_1}(t,\xi)\|^2_{L^2_{v,k}}dt\big)^\frac{1}{2}\notag\\
		\le &\sqrt{2}\|\widehat{g_0}(\xi)\|_{L^2_{v,k}}+\big(\int^T_0\rmre(\CL_De^{\la t}\widehat{g_1},\langle v\rangle^{2k}e^{\la t}\widehat{g_1})(t,\xi)dt\big)^\frac{1}{2}\notag\\
		&+\big(\int^T_0\rmre(\widehat{Q}(e^{\la t}\widehat{g_1},\widehat{g_1}),\langle v\rangle^{2k}e^{\la t}\widehat{g_1})(t,\xi)dt\big)^\frac{1}{2}\notag\\&+\big(\int^T_0\rmre(\widehat{Q}(\sqrt{\mu}\widehat{g_2},e^{\la t}\widehat{g_1}),\langle v\rangle^{2k}e^{\la t}\widehat{g_1})(t,\xi)dt\big)^\frac{1}{2}\notag\\&+\big(\int^T_0\rmre(\widehat{Q}(e^{\la t}\widehat{g_1},\sqrt{\mu}\widehat{g_2}),\langle v\rangle^{2k}e^{\la t}\widehat{g_1})(t,\xi)dt\big)^\frac{1}{2}\notag\\
		:=&\sqrt{2}\|\widehat{g_0}(\xi)\|_{L^2_{v,k}}+I_1+I_2+I_3+I_4.
	\end{align}
For $I_1$, by letting $g=0$ in \eqref{leld}, we have
\begin{align}\label{I1}
	I_1\leq - \sqrt{\de} \Vert e^{\la t}\widehat{g_1}(\xi) \Vert_{L^2_TH^s_{v,k+\gamma/2}}.
\end{align}
Using \eqref{fgh}, it holds that
\begin{align}\label{I21}
&\vert (\widehat{Q}(e^{\la t}\widehat{g_1},\widehat{g_1})(t,\xi),\langle v\rangle^{2k}e^{\la t}\widehat{g_1}(t,\xi))\vert\notag\\
 \leq& C_k \int_{\mathbb{R}^3}\Vert e^{\la t}\widehat{g_1}(t,\xi-\ell) \Vert_{L^2_{v,14}} \Vert \widehat{g_1}(t,\ell) \Vert_{H^s_{v,k+\gamma/2+2s }}  \Vert e^{\la t}\widehat{g_1}(t,\xi) \Vert_{H^s_{v,k+\gamma/2}}dl\notag\\
&\quad+  C_k\int_{\mathbb{R}^3}\Vert \widehat{g_1}(t,\xi-\ell) \Vert_{L^2_{v,14}}  \Vert e^{\la t}\widehat{g_1}(t,\ell) \Vert_{H^s_{v,k+\gamma/2}}  \Vert e^{\la t}\widehat{g_1}(t,\xi) \Vert_{H^s_{v,k+\gamma/2}}dl .
\end{align}
Then it is direct to see
\begin{align}\label{I22}
	I_2&\leq C_k\big(\int^T_0 \int_{\mathbb{R}^3}\Vert e^{\la t}\widehat{g_1}(t,\xi-\ell) \Vert_{L^2_{v,14}} \Vert \widehat{g_1}(t,\ell) \Vert_{H^s_{v,k+\gamma/2+2s }}  \Vert e^{\la t}\widehat{g_1}(t,\xi) \Vert_{H^s_{v,k+\gamma/2}} d\ell dt\big)^\frac{1}{2} \notag\\
	&\quad +C_k\big(\int^T_0 \int_{\mathbb{R}^3}\Vert \widehat{g_1}(t,\xi-\ell) \Vert_{L^2_{v,14}}  \Vert e^{\la t}\widehat{g_1}(t,\ell) \Vert_{H^s_{v,k+\gamma/2}}  \Vert e^{\la t}\widehat{g_1}(t,\xi) \Vert_{H^s_{v,k+\gamma/2}}d\ell dt\big)^\frac{1}{2}.
	\end{align}
By Cauchy-Schwarz inequality, one has that
\begin{align}\label{I23}
	I_2&\le C_k\Big( \int^T_0 \Big( \int_{\mathbb{R}^3} \Vert e^{\la t}\widehat{g_1}(t,\xi-\ell) \Vert_{L^2_{v,14}} \Vert \widehat{g_1}(t,\ell) \Vert_{H^s_{v,k+\gamma/2+2s }} d\ell\Big)^2dt \Big)^{1/4}\notag\\
	&\qquad\times \Big( \int^T_0 \Vert e^{\la t}\widehat{g_1}(t,\xi) \Vert^2_{H^s_{v,k+\gamma/2}} dt\Big)^{1/4}\notag\\
	&\quad+C_k\Big( \int^T_0 \Big( \int_{\mathbb{R}^3}\Vert \widehat{g_1}(t,\xi-\ell) \Vert_{L^2_{v,14}}  \Vert e^{\la t}\widehat{g_1}(t,\ell) \Vert_{H^s_{v,k+\gamma/2}} d\ell\Big)^2dt \Big)^{1/4} \notag\\
	&\qquad\times\Big( \int^T_0 \Vert e^{\la t}\widehat{g_1}(t,\xi) \Vert^2_{H^s_{v,k+\gamma/2}} dt\Big)^{1/4}\notag\\
	&\le \eta \Vert e^{\la t}\widehat{g_1}(\xi) \Vert_{L^2_TH^s_{v,k+\gamma/2}}\notag\\
	&\quad+C_{k,\eta} \Big(\int^T_0 \Big( \int_{\mathbb{R}^3} \Vert e^{\la t}\widehat{g_1}(t,\xi-\ell) \Vert_{L^2_{v,14}} \Vert \widehat{g_1}(t,\ell) \Vert_{H^s_{v,k+\gamma/2+2s }} d\ell\Big)^2dt\Big)^\frac{1}{2} \notag\\
	&\quad+C_{k,\eta}\Big(\int^T_0 \Big( \int_{\mathbb{R}^3} \Vert \widehat{g_1}(t,\xi-\ell) \Vert_{L^2_{v,14}}  \Vert e^{\la t}\widehat{g_1}(t,\ell) \Vert_{H^s_{v,k+\gamma/2}} d\ell\Big)^2dt\Big)^\frac{1}{2}.
\end{align}
We further use Minkowski  inequality to get
\begin{align}
	I_2&\le  \eta \Vert e^{\la t}\widehat{g_1}(\xi) \Vert_{L^2_TH^s_{v,k+\gamma/2}} \notag\\
	&\quad+C_{k,\eta}   \int_{\mathbb{R}^3}\big(\int^T_0 \Vert e^{\la t}\widehat{g_1}(t,\xi-\ell) \Vert^2_{L^2_{v,14}} \Vert \widehat{g_1}(t,\ell) \Vert^2_{H^s_{v,k+\gamma/2+2s }}dt\big)^\frac{1}{2} d\ell \notag\\
	&\quad+C_{k,\eta}\int_{\mathbb{R}^3}\big(\int^T_0  \Vert \widehat{g_1}(t,\xi-\ell) \Vert_{L^2_{v,14}}^2  \Vert e^{\la t}\widehat{g_1}(t,\ell) \Vert_{H^s_{v, k+\gamma/2}}^2   dt \big)^\frac{1}{2} d\ell,\notag
\end{align}
which further yields
\begin{align}\label{I2}
	I_2&\le \eta \Vert e^{\la t}\widehat{g_1}(\xi) \Vert_{L^2_TH^s_{v,k+\gamma/2}} +C_{k,\eta} \int_{\mathbb{R}^3}\Vert e^{\la t}\widehat{g_1}(\xi-\ell) \Vert_{L^\infty_TL^2_{v,k}} \Vert \widehat{g_1}(\ell) \Vert_{L^2_TH^s_{v,k+\gamma/2+2s }}d\ell \notag\\
	&\quad+C_{k,\eta}\int_{\mathbb{R}^3}\Vert \widehat{g_1}(\xi-\ell) \Vert_{L^\infty_TL^2_{v,k}}  \Vert e^{\la t}\widehat{g_1}(\ell) \Vert_{L^2_TH^s_{v,k+\gamma/2}}d\ell.
\end{align}
Similar arguments as in \eqref{I21}, \eqref{I22}, \eqref{I23} and \eqref{I2} show that
\begin{align*}
	I_3+I_4&\le \eta \Vert e^{\la t}\widehat{g_1}(\xi) \Vert^2_{L^2_TH^s_{v,k+\gamma/2}} \notag\\
	&\quad+C_{k,\eta} \int_{\mathbb{R}^3}\Vert \sqrt{\mu}\widehat{g_2}(\xi-\ell) \Vert_{L^\infty_TL^2_{v,k}} \Vert e^{\la t}\widehat{g_1}(\ell) \Vert_{L^2_TH^s_{v,k+\gamma/2+2s }}d\ell \notag\\
	&\quad+C_{k,\eta}\int_{\mathbb{R}^3}\Vert e^{\la t}\widehat{g_1}(\xi-\ell) \Vert_{L^\infty_TL^2_{v,k}}  \Vert \sqrt{\mu}\widehat{g_2}(\ell) \Vert_{L^2_TH^s_{v,k+\gamma/2}}d\ell\notag\\
	&\quad+C_{k,\eta} \int_{\mathbb{R}^3}\Vert e^{\la t}\widehat{g_1}(\xi-\ell) \Vert_{L^\infty_TL^2_{v,k}} \Vert \sqrt{\mu}\widehat{g_2}(\ell) \Vert_{L^2_TH^s_{v,k+\gamma/2+2s }}d\ell\notag\\
	&\quad+C_{k,\eta}\int_{\mathbb{R}^3}\Vert \sqrt{\mu}\widehat{g_2}(\xi-\ell) \Vert_{L^\infty_TL^2_{v,k}}  \Vert e^{\la t}\widehat{g_1}(\ell) \Vert_{L^2_TH^s_{v,k+\gamma/2}}d\ell.
\end{align*}
From the facts that
\begin{align*}
	\Vert \sqrt{\mu}\widehat{g_2}(\ell) \Vert_{L^2_TH^s_{v,k+\gamma/2}}+\Vert \sqrt{\mu}\widehat{g_2}(\ell) \Vert_{L^2_TH^s_{v,k+\gamma/2+2s }}\leq C_k\Vert \widehat{g_2}(\ell) \Vert_{L^2_TH^{s*}_v},
\end{align*}
and
\begin{align*}
	\Vert \sqrt{\mu}\widehat{g_2}(\xi-\ell) \Vert_{L^\infty_TL^2_{v,k}}\leq C_k\Vert \widehat{g_2}(\xi-\ell) \Vert_{L^\infty_TL^2_v},
\end{align*}
we have
\begin{align}\label{I3I4}
	I_3+I_4&\le \eta \Vert e^{\la t}\widehat{g_1}(\xi) \Vert_{L^2_TH^s_{v,k+\gamma/2}} +C_{k,\eta} \int_{\mathbb{R}^3}\Vert \widehat{g_2}(\xi-\ell) \Vert_{L^\infty_TL^2_v} \Vert e^{\la t}\widehat{g_1}(\ell) \Vert_{L^2_TH^s_{v,k+\gamma/2+2s }}d\ell \notag\\
	&\quad+C_{k,\eta}\int_{\mathbb{R}^3}\Vert e^{\la t}\widehat{g_1}(\xi-\ell) \Vert_{L^\infty_TL^2_{v,k}} \Vert \widehat{g_2}(\ell) \Vert_{L^2_TH^{s*}_v}d\ell.
\end{align}
Combining \eqref{I1234}, \eqref{I1}, \eqref{I2} and \eqref{I3I4}, it holds
\begin{align}\label{I1I2I3I4}			
	&\frac{1}{2\sqrt{2}}\|e^{\la t}\widehat{g_1}(t,\xi)\|_{L^2_{v,k}}+\sqrt{\de} \Vert e^{\la t}\widehat{g_1}(\xi) \Vert_{L^2_TH^s_{v,k+\gamma/2}}-\sqrt{\la}\|e^{\la t}\widehat{g_1}(\xi)\|_{L^2_TL^2_{v,k}}\notag\\
	\le&\sqrt{2}\|\widehat{g_0}(\xi)\|_{L^2_{v,k}}+ \eta \Vert e^{\la t}\widehat{g_1}(\xi) \Vert_{L^2_TH^s_{v,k+\gamma/2}}\notag\\
	&\quad+C_{k,\eta} \int_{\mathbb{R}^3}\Vert e^{\la t}\widehat{g_1}(\xi-\ell) \Vert_{L^\infty_TL^2_{v,k}} \Vert \widehat{g_1}(\ell) \Vert_{L^2_TH^s_{v,k+\gamma/2+2s }}d\ell \notag\\
	&\quad+C_{k,\eta}\int_{\mathbb{R}^3}\Vert \widehat{g_1}(\xi-\ell) \Vert_{L^\infty_TL^2_{v,k}}  \Vert e^{\la t}\widehat{g_1}(\ell) \Vert_{L^2_TH^s_{v,k+\gamma/2}}d\ell\notag\\
	&\quad+C_{k,\eta} \int_{\mathbb{R}^3}\Vert \widehat{g_2}(\xi-\ell) \Vert_{L^\infty_TL^2_v} \Vert e^{\la t}\widehat{g_1}(\ell) \Vert_{L^2_TH^s_{v,k+\gamma/2+2s }}d\ell \notag\\
	&\quad+C_{k,\eta}\int_{\mathbb{R}^3}\Vert e^{\la t}\widehat{g_1}(\xi-\ell) \Vert_{L^\infty_TL^2_{v,k}} \Vert \widehat{g_2}(\ell) \Vert_{L^2_TH^{s*}_v}d\ell.
\end{align}
Noticing $\Vert \widehat{g_1}(\xi-\ell) \Vert_{L^\infty_TL^2_{v,k}}\le\Vert e^{\la t}\widehat{g_1}(\xi-\ell) \Vert_{L^\infty_TL^2_{v,k}}$ and $$\max\{\Vert \widehat{g_1}(\ell) \Vert_{L^2_TH^s_{v,k+\gamma/2+2s }},\Vert e^{\la t}\widehat{g_1}(\ell) \Vert_{L^2_TH^s_{v,k+\gamma/2}}\}\leq \Vert e^{\la t}\widehat{g_1}(\ell) \Vert_{L^2_TH^s_{v,k+\gamma/2+2s }},$$ we first choose $\la$ and $\eta$ small enough to get
\begin{align*}			
	&\|e^{\la t}\widehat{g_1}(\xi)\|_{L^\infty_TL^2_{v,k}}+\de\Vert e^{\la t}\widehat{g_1}(\xi) \Vert_{L^2_TH^s_{v,k+\gamma/2}}\notag\\
	\le&\|\widehat{g_0}(\xi)\|_{L^2_{v,k}}+ C_k \int_{\mathbb{R}^3}\Vert e^{\la t}\widehat{g_1}(\xi-\ell) \Vert_{L^\infty_TL^2_{v,k}} \Vert e^{\la t}\widehat{g_1}(\ell) \Vert_{L^2_TH^s_{v,k+\gamma/2+2s }}d\ell \notag\\
	&+ C_k\int_{\mathbb{R}^3}\Vert \widehat{g_2}(\xi-\ell) \Vert_{L^\infty_TL^2_v} \Vert e^{\la t}\widehat{g_1}(\ell) \Vert_{L^2_TH^s_{v,k+\gamma/2+2s }}d\ell\notag\\
	&+C_k\int_{\mathbb{R}^3}\Vert e^{\la t}\widehat{g_1}(\xi-\ell) \Vert_{L^\infty_TL^2_{v,k}} \Vert \widehat{g_2}(\ell) \Vert_{L^2_TH^{s*}_v}d\ell.
\end{align*}
Then taking $L^p_\xi$ norm and using Minkowski inequality again, we have
\begin{align}\label{g1lpxih}			
	&\|e^{\la t}\widehat{g_1}\|_{L^p_\xi L^\infty_TL^2_{v,k}}+\de\Vert e^{\la t}\widehat{g_1}\Vert_{L^p_\xi L^2_TH^s_{v,k+\gamma/2}}\notag\\
	\le&\|\widehat{g_0}\|_{L^p_\xi L^2_{v,k}}+ C_k(\Vert e^{\la t}\widehat{g_1} \Vert_{L^p_\xi L^\infty_TL^2_{v,k}}+\Vert \widehat{g_2} \Vert_{L^p_\xi L^\infty_TL^2_v} ) \Vert e^{\la t}\widehat{g_1} \Vert_{L^1_\xi L^2_TH^s_{v,k+\gamma/2+2s }} \notag\\
	&+ C_k\Vert e^{\la t}\widehat{g_1} \Vert_{L^p_\xi L^\infty_TL^2_{v,k}} \Vert \widehat{g_2} \Vert_{L^1_\xi L^2_TH^{s*}_v}.
\end{align}
By definition of $L^1_\xi L^2_TH^s_{v,k+\gamma/2+2s }$ norm, one has
\begin{align*}
	\Vert e^{\la t}\widehat{g_1} \Vert_{L^1_\xi L^2_TH^s_{v,k+\gamma/2+2s }}&=\int_{\R^3}\big(\int^T_0   \Vert e^{\la t}\widehat{g_1}(t,\xi) \Vert^2_{H^s_{v,k+\gamma/2+2s}} dt\big)^\frac{1}{2}d\xi\notag\\
	&= \int_{\R^3}\langle\xi\rangle^{-\frac{3}{2}-}\big(\int^T_0 \langle\xi\rangle^{3+}  \Vert e^{\la t}\widehat{g_1}(t,\xi) \Vert^2_{H^s_{v,k+\gamma/2+2s}} dt\big)^\frac{1}{2}d\xi,
\end{align*}
where $-\frac{3}{2}-$ and $3+$ are defined to be $-\frac{3}{2}-\ka_0$ and $3+\ka_0$ for some sufficiently small $\ka_0>0$. Using Cauchy-Schwarz inequality, we obtain
\begin{align}\label{txiinterpo}
	\Vert e^{\la t}\widehat{g_1} \Vert_{L^1_\xi L^2_TH^s_{v,k+\gamma/2+2s }}&\leq C \big(\int_{\R^3}\langle\xi\rangle^{-3-}d\xi\big)^\frac{1}{2}\notag\\
	&\qquad\times\big(\int^T_0  \int_{\R^3} \langle\xi\rangle^{3+}\Vert e^{\la t}\widehat{g_1}(t,\xi) \Vert^2_{H^s_{v,k+\gamma/2+2s}} d\xi dt\big)^\frac{1}{2}\notag\\
	&=C\big(\int^T_0 \Vert e^{\la t}\langle\xi\rangle^{\frac{3}{2}+}\widehat{g_1}(t,\xi) \Vert^2_{L^2_\xi H^s_{v,k+\gamma/2+2s}}dt\big)^\frac{1}{2}.
\end{align}
By Plancherel Theorem, Sobolev embedding and the definition of $X^*_k$ in \eqref{X*k}, one gets
\begin{align}\label{g1l1xih}
	\Vert e^{\la t}\widehat{g_1} \Vert_{L^1_\xi L^2_TH^s_{v,k+\gamma/2+2s }}&\leq C\big(\int^T_0 \Vert e^{\la t}g_1 \Vert^2_{H^2_x H^s_{v,k+\gamma/2+2s}}dt\big)^\frac{1}{2}\notag\\
	&\leq C\big(\int^T_0 \Vert e^{\la t}g_1 \Vert^2_{X^*_{k+8+2s}}dt\big)^\frac{1}{2}.
\end{align}
Hence, \eqref{esthatg1h} follows from \eqref{g1lpxih} and \eqref{g1l1xih}. The proof of Lemma \ref{hatg1h} is complete.
\end{proof}

\subsection{Estimates for $-3<\ga<0$}
\begin{lemma}\label{hatg1s}
	Let $-3<\ga<0$, $k>18$ and $g_1$ be a solution to \eqref{hatg1}. For any $1<\rho<\frac{k-14}{|\ga|}$, there exists a constant $C_k>0$ such that 	
	\begin{align}\label{esthatg1s}			
		&\|(1+t)^\rho \widehat{g_1}\|_{L^p_\xi L^\infty_TL^2_{v,k}}+\Vert (1+t)^\rho \widehat{g_1}\Vert_{L^p_\xi L^2_TH^s_{v,k+\gamma/2}}\notag\\
		\le&C_k\|\widehat{g_0}\|_{L^p_\xi L^2_{v,k}}+C_k\|\widehat{g_1}\|_{L^p_\xi L^2_TL^2_{v,k+\ga/2-\rho\ga}}\notag\\
		&+ C_{k} \big(\Vert (1+t)^\rho \widehat{g_1} \Vert_{L^p_\xi L^\infty_TL^2_{v,k}}+\Vert \widehat{g_2} \Vert_{L^p_\xi L^\infty_TL^2_v} \big)\big(\int^T_0 \Vert (1+t)^\rho g_1 \Vert^2_{X^*_{k+8+2s}}dt\big)^\frac{1}{2} \notag\\
		&+C_k\Vert (1+t)^\rho \widehat{g_1} \Vert_{L^p_\xi L^\infty_TL^2_{v,k}} \Vert \widehat{g_2} \Vert_{L^1_\xi L^2_TH^{s*}_v}.
	\end{align}
\end{lemma}
\begin{proof}
	It is straightforward to see from \eqref{hatg1} that
	\begin{align}
		&\pa_t (1+t)^\rho \widehat{g_1}+iv\cdot\xi (1+t)^\rho \widehat{g_1}-\rho (1+t)^{\rho-1} \widehat{g_1}\notag\\
		 =&\CL_D (1+t)^\rho \widehat{g_1}+\widehat{Q}((1+t)^\rho \widehat{g_1},\widehat{g_1})+\widehat{Q}(\sqrt{\mu}\widehat{g_2},(1+t)^\rho \widehat{g_1})+\widehat{Q}((1+t)^\rho \widehat{g_1},\sqrt{\mu}\widehat{g_2}).
	\end{align}
Similar arguments as in \eqref{I1234}, \eqref{I1}, \eqref{I21}, \eqref{I22}, \eqref{I23}, \eqref{I2}, \eqref{I3I4} and \eqref{I1I2I3I4} show that for $k>14$,
\begin{align*}		
	&\|(1+t)^\rho \widehat{g_1}(t,\xi)\|_{L^2_{v,k}}+\de \Vert (1+t)^\rho \widehat{g_1}(\xi) \Vert_{L^2_TH^s_{v,k+\gamma/2}}-\sqrt{\rho}\|(1+t)^{\rho-1/2} \widehat{g_1}(\xi)\|_{L^2_TL^2_{v,k}}\notag\\
	\le&\|\widehat{g_0}(\xi)\|_{L^2_{v,k}}+ \eta \Vert (1+t)^\rho \widehat{g_1}(\xi) \Vert_{L^2_TH^s_{v,k+\gamma/2}}\notag\\
	&\quad+C_{k,\eta} \int_{\mathbb{R}^3}\Vert (1+t)^\rho \widehat{g_1}(\xi-\ell) \Vert_{L^\infty_TL^2_{v,k}} \Vert \widehat{g_1}(\ell) \Vert_{L^2_TH^s_{v,k+\gamma/2+2s }}d\ell \notag\\
	&\quad+C_{k,\eta}\int_{\mathbb{R}^3}\Vert \widehat{g_1}(\xi-\ell) \Vert_{L^\infty_TL^2_{v,k}}  \Vert (1+t)^\rho \widehat{g_1}(\ell) \Vert_{L^2_TH^s_{v,k+\gamma/2}}d\ell\notag\\
	&\quad+C_{k,\eta} \int_{\mathbb{R}^3}\Vert \widehat{g_2}(\xi-\ell) \Vert_{L^\infty_TL^2_v} \Vert (1+t)^\rho \widehat{g_1}(\ell) \Vert_{L^2_TH^s_{v,k+\gamma/2+2s }}d\ell \notag\\
	&\quad+C_{k,\eta}\int_{\mathbb{R}^3}\Vert (1+t)^\rho \widehat{g_1}(\xi-\ell) \Vert_{L^\infty_TL^2_{v,k}} \Vert \widehat{g_2}(\ell) \Vert_{L^2_TH^{s*}_v}d\ell.
\end{align*}
Then we take $L^p_\xi$ norm and use Minkowski inequality, Plancherel Theorem and Sobolev embedding as in \eqref{g1lpxih}, \eqref{txiinterpo} and \eqref{g1l1xih} to get
\begin{align}\label{tg1s}
	&\|(1+t)^\rho \widehat{g_1}\|_{L^p_\xi L^\infty_TL^2_{v,k}}+\de\Vert (1+t)^\rho \widehat{g_1}\Vert_{L^p_\xi L^2_TH^s_{v,k+\gamma/2}}-\sqrt{\rho}\|(1+t)^{\rho-1/2} \widehat{g_1}\|_{L^p_\xi L^2_TL^2_{v,k}}\notag\\
	\le&\|\widehat{g_0}\|_{L^p_\xi L^2_{v,k}}+ C_{k} \big(\Vert (1+t)^\rho \widehat{g_1} \Vert_{L^p_\xi L^\infty_TL^2_{v,k}}+\Vert \widehat{g_2} \Vert_{L^p_\xi L^\infty_TL^2_v} \big)\big(\int^T_0 \Vert (1+t)^\rho g_1 \Vert^2_{X^*_{k+8+2s}}dt\big)^\frac{1}{2} \notag\\
	&\qquad+C_k\Vert (1+t)^\rho \widehat{g_1} \Vert_{L^p_\xi L^\infty_TL^2_{v,k}} \Vert \widehat{g_2} \Vert_{L^1_\xi L^2_TH^{s*}_v}.
\end{align}
For the term $\de\Vert (1+t)^\rho \widehat{g_1}\Vert_{L^p_\xi L^2_TH^s_{v,k+\gamma/2}}-\sqrt{\rho}\|(1+t)^{\rho-1/2} \widehat{g_1}\|_{L^p_\xi L^2_TL^2_{v,k}}$, it follows from similar arguments as in \eqref{larget} and \eqref{tlarge} that if  $1+t\geq \frac{1}{\ka(1+|v|)^{\ga}}$ for some $\ka>0$, one has
\begin{align}\label{largets}
	&\sqrt{\rho}\|(1+t)^{\rho-1/2} \widehat{g_1}\|_{L^p_\xi L^2_TL^2_{v,k}}\leq \sqrt{\ka\rho}\Vert (1+t)^\rho \widehat{g_1}\Vert_{L^p_\xi L^2_TL^2_{v,k+\gamma/2}}.
\end{align}
If $1+t\leq \frac{1}{\ka(1+|v|)^{\ga}}$, it holds
\begin{align}\label{smallts}
	&\sqrt{\rho}\|(1+t)^{\rho-1/2} \widehat{g_1}\|_{L^p_\xi L^2_TL^2_{v,k}}\leq \sqrt{\rho}\Vert \widehat{g_1}\Vert_{L^p_\xi L^2_TL^2_{v,k+\gamma/2-\rho\ga}}.	
\end{align}
Then choosing $\ka$ to be small, \eqref{esthatg1soft} holds from \eqref{tg1s}, \eqref{largets} and \eqref{smallts}. Note that as we will prove later, the estimate on $g_1$ requires the order of velocity weight to be larger that $14$, so here we also let $k>18$ and $\rho<\frac{k-14}{|\ga|}$ to guarantee  $k+\gamma/2-\rho\ga>14$ and $1<\frac{k-14}{|\ga|}$.
\end{proof}
From \eqref{esthatg1s}, we see that for $\ga<0$, in order to obtain the time decay, we also need to bound $\Vert \widehat{g_1}\Vert_{L^p_\xi L^2_TL^2_{k}}$ for some large $k$. The following lemma provides the estimate of such term. The details are almost the same as Lemma \ref{hatg1s} and thus omitted.

\begin{lemma}\label{hatg1soft}
	Let $-3<\ga<0$, $k>14$ and $g_1$ be a solution to \eqref{hatg1}. There exist constants $\de,C_k>0$ such that 	
	\begin{align}\label{esthatg1soft}			
		&\| \widehat{g_1}\|_{L^p_\xi L^\infty_TL^2_{v,k}}+\de\Vert \widehat{g_1}\Vert_{L^p_\xi L^2_TH^s_{v,k+\gamma/2}}\notag\\
		\le&\|\widehat{g_0}\|_{L^p_\xi L^2_{v,k}}+ C_{k} \big(\Vert  \widehat{g_1} \Vert_{L^p_\xi L^\infty_TL^2_{v,k}}+\Vert \widehat{g_2} \Vert_{L^p_\xi L^\infty_TL^2_v} \big)\big(\int^T_0 \Vert  g_1 \Vert^2_{X^*_{k+8+2s}}dt\big)^\frac{1}{2} \notag\\
		&\qquad+C_k\Vert \widehat{g_1} \Vert_{L^p_\xi L^\infty_TL^2_{v,k}} \Vert \widehat{g_2} \Vert_{L^1_\xi L^2_TH^{s*}_v}.
	\end{align}
\end{lemma}

\section{Time decay of $g_2$ in $L^1_\xi\cap L^p_\xi$}
We can prove the stability of $g_2$ for both hard and soft potentials. However, the time decay should be proved in two cases, $\ga+2s\ge0$ and $\ga+2s<0$ respectively. The important inequality $\|f\|_{L^2_v}\leq \|f\|_{L^2_{v,\ga+2s}}\leq C\|f\|_{H^{s*}_v}$ will be frequently used in case of $\ga+2s\geq0$. On the other hand, such inequality no longer holds when $\ga+2s<0$. Therefore, we need extra velocity weight to  compensate the dissipation, as we will see in the third subsection.
\subsection{$L^1_\xi\cap L^p_\xi$ estimates on $g_2$}
We first bound the nonlinear term in \eqref{hatg2}. The following lemma is given in \cite[Lemma 3.1]{DSY}.
\begin{lemma}
	For $0<s<1$ and $\gamma>\max\{-3,-3/2-2s\}$, it holds that
	\begin{equation}\label{lem.ibn1}
		\left| \left(\hat{\Gamma}(\hat{f},\hat{g})(\xi),\widehat{H}(\xi)\right)\right|\le C\int_{\mathbb{R}^3_\ell}\Vert \hat{f}(\xi-\ell)\Vert_{L^2_v}\Vert \hat{g}(\ell)\Vert_{H^{s*}_v}  \Vert  (\mathbf{I}-\mathbf{P})  \widehat{H}(\xi)\Vert_{H^{s*}_v} d\ell.
	\end{equation}
\end{lemma}
Then we have the following inequality for $g_2$.
\begin{lemma}\label{lem: micro a priori}
	Let $\max\{-3,-3/2-2s\}<\ga\leq1$, $0<s<1$, $p>1$, $\rho>1$ and $g_2$ be a solution to \eqref{hatg2}. Recalling $A$ and $M$ are two constants in the definition of $L_B$ in \eqref{defLB}, and 
	$$
	\sigma=3\Big(1-\frac{1}{p}\Big) -2\ep_1, 
	$$
	where $\ep_1>0$ is arbitrarily small. There are constants $C>0$ and $C_{A,M}$ such that 
		\begin{align}
		&\Vert \widehat{g_2}\Vert_{L^1_\xi L^\infty_T L^2_v}+\Vert (\mathbf{I}-\mathbf{P})\widehat{g_2} \Vert_{L^1_\xi L^2_T H^{s*}_v}\notag\\
		\le &C_{A,M}\Vert \widehat{g_2}\Vert^\frac{1}{2}_{L^1_\xi L^\infty_T L^2_v}\Big(\int^T_0\|(1+t)^\rho g_1(t)\|^2_{X^*_{8-\ga/2}}dt\Big)^\frac{1}{4}\notag\\
		&+C\Vert \widehat{g_2}\Vert_{L^1_\xi L^\infty_T L^2_v}\Vert (\mathbf{I}-\mathbf{P})\widehat{g_2} \Vert_{L^1_\xi L^2_T H^{s*}_v}+C\Vert \widehat{g_2}\Vert_{L^1_\xi L^\infty_T L^2_v}\Vert  (1+t)^{\sigma/2} \widehat{g_2}\Vert_{L^1_\xi L^\infty_T L^2_v},\label{ineq: L^1_k micro a priori}
	\end{align}
	and 
	\begin{align}
		&\Vert \widehat{g_2}\Vert_{L^p_\xi L^\infty_T L^2_v}+\Vert (\mathbf{I}-\mathbf{P})\widehat{g_2} \Vert_{L^p_\xi L^2_T H^{s*}_v}\notag\\
		\le &C_{A,M}\Vert \widehat{g_2}\Vert^\frac{1}{2}_{L^p_\xi L^\infty_T L^2_v}\Vert (1+t)^\rho\widehat{g_1} \Vert^\frac{1}{2}_{L^p_\xi L^2_TL^2_v}+C\Vert \widehat{g_2}\Vert_{L^p_\xi L^\infty_T L^2_v} \Vert (\mathbf{I}-\mathbf{P})\widehat{g_2}\Vert_{L^1_\xi L^2_T H^{s*}_v}\notag\\
		&+C\Vert \widehat{g_2}\Vert_{L^p_\xi L^\infty_T L^2_v}\Vert  (1+t)^{\sigma/2} \widehat{g_2}\Vert_{L^1_\xi L^\infty_T L^2_v} , \label{ineq: L^infty_k micro a priori}
	\end{align}
	for any $T>0$ and $\rho>1$, where $C_{A,M}$ depends only on $A$ and $M$.
\end{lemma}
\begin{proof}
Multiplying $\bar{\hat{g}}_2$ to \eqref{hatg2}, taking real part and integrating over $[0,T]\times\R^3_v$ , one gets
\begin{align*}
	\frac{1}{2}\|\widehat{g_2}(t,\xi)\|^2_{L^2_v}=&\int^T_0\rmre(L\widehat{g_2},\widehat{g_2})(t,\xi)dt+\int^T_0\rmre(\CL_B\widehat{g_1},\widehat{g_2})(t,\xi)dt\notag\\&\qquad\qquad\qquad\qquad\qquad\qquad+\int^T_0\rmre(\hat{\Ga}(\widehat{g_2},\widehat{g_2}),\widehat{g_2})(t,\xi)dt.
\end{align*}
We have the coercivity estimate from~\cite[Proposition 2.1]{AMUXY-2012-JFA} that for some small constant $\de$,
\begin{align*}
	(Lg,g)_{L^2_v}\leq - \delta \Vert (\mathbf{I}-\mathbf{P})g \Vert_{H^{s*}_v}^2,
\end{align*}
which, together with Cauchy-Schwarz inequality, yields
\begin{align}\label{hardg20}
	&\frac{1}{2}\|\widehat{g_2}(\xi)\|_{L^\infty_TL^2_v}+\de\Big(\int^T_0\|\{\mathbf{I}-\mathbf{P}\}\widehat{g_2}(t,\xi)\|_{H^{s*}_v}dt\Big)^\frac{1}{2}\notag\\
	\leq& \Big(\int^T_0\rmre(\CL_B\widehat{g_1},\widehat{g_2})(t,\xi)dt\Big)^\frac{1}{2}+\Big(\int^T_0\rmre(\hat{\Ga}(\widehat{g_2},\widehat{g_2}),\widehat{g_2})(t,\xi)dt\Big)^\frac{1}{2}.
\end{align}
By the definition of $\CL_B$ in \eqref{defLB}, one has
\begin{align}\label{LBg1g2}
	\Big(\int^T_0\rmre(\CL_B\widehat{g_1},\widehat{g_2})(t,\xi)dt\Big)^\frac{1}{2}\leq&\Big(\int^T_0\int_{\R^3}\mu^{-1/2}(v)A\chi_M|\widehat{g_1}(t,\xi,v)\widehat{g_2}(t,\xi,v)|dvdt\Big)^\frac{1}{2}\notag\\
	\leq& C_{A,M}\Big(\int^T_0\|\widehat{g_1}(t,\xi)\|_{L^2_v}\|\widehat{g_2}(t,\xi)\|_{L^2_v}dt\Big)^\frac{1}{2}\notag\\
	\leq& C_{A,M}\|\widehat{g_2}(\xi)\|^\frac{1}{2}_{L^\infty_TL^2_v}\Big(\int^T_0\|\widehat{g_1}(t,\xi)\|_{L^2_v}dt\Big)^\frac{1}{2}.
\end{align}
Using \eqref{lem.ibn1} and Cauchy-Schwarz inequality, we have
\begin{align}\label{3g20}
	&\Big(\int^T_0 \vert (\hat{\Gamma}(\widehat{g_2},\widehat{g_2}), \widehat{g_2})(t,\xi)\vert dt\Big)^{1/2}\notag \\
	\le& C\Big( \int^T_0 \big( \int_{\mathbb{R}^3} \Vert \widehat{g_2}(t,\xi-\ell) \Vert_{L^2_v} \Vert \widehat{g_2}(t,\ell)\Vert_{H^{s*}_v} d\ell\big)^2dt \Big)^{1/4} \big( \int^T_0 \Vert (\mathbf{I}-\mathbf{P}) \widehat{g_2}(t,\xi)\Vert_{H^{s*}_v}^2 dt\big)^{1/4}\notag\\
	\le& \eta \Big(\int^T_0 \Vert (\mathbf{I}-\mathbf{P}) \widehat{g_2}(t,\xi)\Vert_{H^{s*}_v}^2 dt \Big)^{1/2}\notag\\
	&\quad+C_{\eta} \Big( \int^T_0 \Big( \int_{\mathbb{R}^3} \Vert \widehat{g_2}(t,\xi-\ell) \Vert_{L^2_v} \Vert \widehat{g_2}(t,\ell)\Vert_{H^{s*}_v} d\ell\Big)^2dt \Big)^{1/2}.
\end{align}
We further use the Minkowski  inequality and H\"older's inequality to get
\begin{align}\label{ga3g2}
	&\Big(\int^T_0 \vert (\hat{\Gamma}(\widehat{g_2},\widehat{g_2}), \widehat{g_2})(t,\xi)\vert dt\Big)^{1/2}\notag \\
	&\le \eta \Big(\int^T_0 \Vert (\mathbf{I}-\mathbf{P}) \widehat{g_2}(t,\xi)\Vert_{H^{s*}_v}^2 dt \Big)^{1/2}+C_{\eta} \int_{\R^3} \|\widehat{g_2}(t,\xi-\ell)\|_{L^\infty_TL^2_v}\|\widehat{g_2}(t,\ell)\|_{L^2_TH^{s*}_v}\,d\ell.
\end{align}
Then by substituting \eqref{LBg1g2} and \eqref{ga3g2} into \eqref{hardg20} and choosing $\eta$ to be small enough, one gets
\begin{align}\label{g23}
	&\|\widehat{g_2}(\xi)\|_{L^\infty_TL^2_v}+\|\{\mathbf{I}-\mathbf{P}\}\widehat{g_2}(\xi)\|_{L^2_TH^{s*}_v}\notag\\
	\leq& C_{A,M}\|\widehat{g_2}(\xi)\|^\frac{1}{2}_{L^\infty_TL^2_v}\Big(\int^T_0\|\widehat{g_1}(t,\xi)\|_{L^2_v}dt\Big)^\frac{1}{2}\notag\\
	&\quad+C\int_{\R^3} \|\widehat{g_2}(t,\xi-\ell)\|_{L^\infty_TL^2_v}\|\widehat{g_2}(t,\ell)\|_{L^2_TH^{s*}_v}\,d\ell.
\end{align}
It holds by taking integral over $\R^3_\xi$ and using Cauchy-Schwarz inequality that
	\begin{align}\label{g21}
	&\Vert \widehat{g_2}\Vert_{L^1_\xi L^\infty_T L^2_v}+\Vert (\mathbf{I}-\mathbf{P})\widehat{g_2} \Vert_{L^1_\xi L^2_T H^{s*}_v}\notag\\
	\le& C_{A,M}\int_{\R^3}\|\widehat{g_2}(\xi)\|^\frac{1}{2}_{L^\infty_TL^2_v}\Big(\int^T_0\|\widehat{g_1}(t,\xi)\|_{L^2_v}dt\Big)^\frac{1}{2}d\xi+C\Vert \widehat{g_2}\Vert_{L^1_\xi L^\infty_T L^2_v} \Vert \widehat{g_2}\Vert_{L^1_\xi L^2_T H^{s*}_v}\notag\\
	\le &C_{A,M}\Vert \widehat{g_2}\Vert_{L^1_\xi L^\infty_T L^2_v}\Big(\int_{\R^3}\int^T_0\|\widehat{g_1}(t,\xi)\|_{L^2_v}dtd\xi\Big)^\frac{1}{2}+C\Vert \widehat{g_2}\Vert_{L^1_\xi L^\infty_T L^2_v} \Vert \widehat{g_2}\Vert_{L^1_\xi L^2_T H^{s*}_v}.
\end{align}
Applying Cauchy-Schwarz inequality as in \eqref{txiinterpo}, one has
\begin{align}\label{g22}
	\Big(\int_{\R^3}\int^T_0\|\widehat{g_1}(t,\xi)\|^2_{L^2_v}dtd\xi\Big)^\frac{1}{2}\leq& C\Big(\int^T_0\int_{\R^3}(1+t)^\rho \langle\xi\rangle^{3+}\|\widehat{g_1}(t,\xi)\|^2_{L^2_v}d\xi dt\Big)^\frac{1}{4}\notag\\
	\leq& C\Big(\int^T_0\|(1+t)^\rho g_1(t)\|^2_{H^2_xL^2_v}dt\Big)^\frac{1}{4}.
\end{align}
Then we obtain from \eqref{g21}, \eqref{g22} and the definition of $X_k^*$ in \eqref{X*k} that
\begin{align}\label{g2l1}
	&\Vert \widehat{g_2}\Vert_{L^1_\xi L^\infty_T L^2_v}+\Vert (\mathbf{I}-\mathbf{P})\widehat{g_2} \Vert_{L^1_\xi L^2_T H^{s*}_v}\notag\\
	&\le C_{A,M}\Vert \widehat{g_2}\Vert^\frac{1}{2}_{L^1_\xi L^\infty_T L^2_v}\Big(\int^T_0\|(1+t)^\rho g_1(t)\|^2_{X^*_{8-\ga/2}}dt\Big)^\frac{1}{4}+C\Vert \widehat{g_2}\Vert_{L^1_\xi L^\infty_T L^2_v} \Vert \widehat{g_2}\Vert_{L^1_\xi L^2_T H^{s*}_v}.
\end{align}
Since $\si>1$, then 
\begin{align}\label{g2L1xi}
\Vert \widehat{g_2}\Vert_{L^1_\xi L^2_T H^{s*}_v}\leq& \Vert (\mathbf{I}-\mathbf{P})\widehat{g_2}\Vert_{L^1_\xi L^2_T H^{s*}_v}+C \|(\hat{a},\hat{b},\hat{c})\|_{L^1_\xi L^2_T}\notag\\
\leq& \|(\mathbf{I}-\mathbf{P})\widehat{g_2}\Vert_{L^1_\xi L^2_T H^{s*}_v}+C\int_{\R^3} \sup_{0<t<T}(1+t)^{\frac{\si}{2}}|{(\hat{a},\hat{b},\hat{c})}(t,\xi)|\,d\xi\notag\\
\leq&   \|(\mathbf{I}-\mathbf{P})\widehat{g_2}\Vert_{L^1_\xi L^2_T H^{s*}_v}+C\Vert(1+t)^{\sigma/2} \widehat{g_2}\Vert_{L^1_\xi L^\infty_T L^2_v}.
\end{align}
We deduce \eqref{ineq: L^1_k micro a priori} from \eqref{g2l1} and \eqref{g2L1xi}.

By the fact that 
	\begin{equation}\notag
		\left\|\|\widehat{g_2}\|_{L^\infty_TL^2_v}\ast_\xi\|\widehat{g_2}\|_{L^2_TH^{s*}_v}\right\|_{L^p(\R^3_\xi)}\leq \Vert \widehat{g_2}\Vert_{L^p_\xi L^\infty_T L^2_v} \Vert \widehat{g_2}\Vert_{L^1_\xi L^2_T H^{s*}_v},
	\end{equation}
similar arguments as in \eqref{LBg1g2}, \eqref{3g20}, \eqref{ga3g2} and \eqref{g23} show that
\begin{align}\label{g2Lp1}
	&\Vert \widehat{g_2}\Vert_{L^p_\xi L^\infty_T L^2_v}+\Vert (\mathbf{I}-\mathbf{P})\widehat{g_2} \Vert_{L^p_\xi L^2_T H^{s*}_v}\notag\\
	&\le C_{A,M}\Big(\int_{\R^3}\big(\int^T_0\|\widehat{g_1}(t,\xi)\|_{L^2_v}\|\widehat{g_2}(t,\xi)\|_{L^2_v}dt\big)^\frac{p}{2}d\xi\Big)^\frac{1}{p}+C\Vert \widehat{g_2}\Vert_{L^p_\xi L^\infty_T L^2_v} \Vert \widehat{g_2}\Vert_{L^1_\xi L^2_T H^{s*}_v}.
\end{align}
A direct application of Cauchy-Schwarz inequality gives
\begin{align}\label{g2Lp2}
	&\Big(\int_{\R^3}\big(\int^T_0\|\widehat{g_1}(t,\xi)\|_{L^2_v}\|\widehat{g_2}(t,\xi)\|_{L^2_v}dt\big)^\frac{p}{2}d\xi\Big)^\frac{1}{p}\notag\\
	\leq& C\Big(\int_{\R^3}\|\widehat{g_2}(\xi)\|^\frac{p}{2}_{L^\infty_TL^2_v}\big(\int^T_0\|(1+t)^\rho \widehat{g_1}(t,\xi)\|^2_{L^2_v}dt\big)^\frac{p}{4}d\xi\Big)^\frac{1}{p}\notag\\
	\leq& C\Vert \widehat{g_2}\Vert^\frac{1}{2}_{L^p_\xi L^\infty_T L^2_v}\Vert (1+t)^\rho\widehat{g_1} \Vert^\frac{1}{2}_{L^p_\xi L^2_TL^2_v},
\end{align}
for any $\rho>1$.
We have from \eqref{g2Lp1} and \eqref{g2Lp2} that
\begin{align}\label{g2Lp}
	\Vert \widehat{g_2}\Vert_{L^p_\xi L^\infty_T L^2_v}+\Vert (\mathbf{I}-\mathbf{P})\widehat{g_2} \Vert_{L^p_\xi L^2_T H^{s*}_v}
	&\le C_{A,M}\Vert \widehat{g_2}\Vert^\frac{1}{2}_{L^p_\xi L^\infty_T L^2_v}\Vert (1+t)^\rho g_1 \Vert^\frac{1}{2}_{L^p_\xi L^2_TL^2_v}\notag\\
	&\qquad+C\Vert \widehat{g_2}\Vert_{L^p_\xi L^\infty_T L^2_v} \Vert \widehat{g_2}\Vert_{L^1_\xi L^2_T H^{s*}_v}. 
\end{align}
Then \eqref{ineq: L^infty_k micro a priori} follows from \eqref{g2L1xi} and \eqref{g2Lp}. Hence, the proof of Lemma \ref{lem: micro a priori} is complete.
\end{proof}
From \eqref{ineq: L^1_k micro a priori} and \eqref{ineq: L^infty_k micro a priori}, one notices in order to close the apriori estimate, we should control the term $\Vert  (1+t)^{\sigma/2} \widehat{g_2}\Vert_{L^1_\xi L^\infty_T L^2_v}$. As we will see later, when we estimate $\Vert  (1+t)^{\sigma/2} \widehat{g_2}\Vert_{L^1_\xi L^\infty_T L^2_v}$, we will further need to bound $\Big\Vert  \frac{\vert \nabla_x \vert}{\langle \nabla_x \rangle } (\hat{a},\hat{b},\hat{c}) \Big\Vert_{L^p_\xi L^2_T}$. Therefore, we should estimate the macroscopic term before we turn to the time-weighted term. Then it is necessary for us to study the problem
\begin{align}\label{eq: linearized BE}
	\begin{cases}
		\partial_t f +v\cdot \nabla_x f -Lf=H,\\
		f(0,x,v)=f_0(x,v).
	\end{cases}
\end{align}
In the symmetric case, it is easily seen that we should take $H=\Gamma(f,f)$ when we turn to the nonlinear problem. Thus, $H$ should be microscopic. However, when we use \eqref{eq: linearized BE} to study $g_2$ which satisfies \eqref{g2}, we should let $H=\CL_B g_1+\Ga(g_2,g_2)$, which is no longer purely microscopic. Then the way we study the equation \eqref{eq: linearized BE} is slightly different from \cite{D,DSY} that we remove the condition $\mathbf{P}H\equiv 0$.

\begin{lemma}\label{lem: macro}
	Let $\max\{-3,-3/2-2s\}<\ga\leq1$, $0<s<1$, $1\le p\leq \infty$ and $f$ be a solution to \eqref{eq: linearized BE} with an inhomogeneous term $H=H(t,x,v)$. Then it holds that
	\begin{align}\label{macro}
		\Big\Vert \frac{\vert \nabla_x \vert}{\langle \nabla_x \rangle } (\hat{a}^f,\hat{b}^f,\hat{c}^f) \Big\Vert_{L^p_\xi L^2_T} 
		&\le C \Vert \hat{f}_0 \Vert_{L^p_\xi L^2_v} + \Vert \hat{f}\Vert_{L^p_\xi L^\infty_T L^2_v} + \Vert (\mathbf{I}-\mathbf{P})\hat{f}\Vert_{L^p_\xi L^2_T H^{s*}_v} \notag\\
		&\quad+\Big(\int_{\mathbb{R}^3} \Big(\int^T_0 \frac{1}{1+\vert \xi\vert^2}\vert (\widehat{H}, \mu^{1/4})_{L^2_v}\vert^2 dt \Big)^\frac{p}{2} d\xi\Big)^\frac{1}{p}, 
	\end{align}
for any $T>0$.
\end{lemma}
\begin{proof}
	Taking inner product of the first equation of \eqref{eq: linearized BE} with the $5$ velocity moments
	\begin{align*}
		\mu^{\frac{1}{2}}, 
		v_j\mu^{\frac{1}{2}}, 
		\frac{1}{6}(|v|^2-3)\mu^{\frac{1}{2}},
		(v_j{v_m}-1)\mu^{\frac{1}{2}}, 
		\frac{1}{10}(|v|^2-5)v_j \mu^{\frac{1}{2}}
	\end{align*}
	with {$1\leq j,m\leq 3$} for the first equation of \eqref{eq: linearized BE}, we obtain the fluid-type system 
	\begin{align}\notag
		\begin{cases}
			\partial_t a^f +\nabla_x b^f=(\mu^{1/2},H),\\
			\partial_t b^f +\nabla_x (a^f+2c^f)+\nabla_x\cdot \Theta ((\mathbf{I}-\mathbf{P}) f)=(v\mu^{1/2},H),\\
			\partial_t c^f +\frac{1}{3}\nabla_x\cdot b^f +\frac{1}{6}\nabla_x\cdot \Lambda ((\mathbf{I}-\mathbf{P}) f)=(\frac{1}{6}(|v|^2-3)\mu^{1/2},H),\\
			\partial_t[\Theta_{jm}((\mathbf{I}-\mathbf{P}) f)+2c^f\delta_{ jm}]+\partial_j b^f_m+\partial_m b^f_j=\Theta_{jm}(\mathbbm{r}+\mathbbm{h}),\\
			\partial_t \Lambda_j((\mathbf{I}-\mathbf{P}) f)+\partial_j c^f = \Lambda_j(\mathbbm{r}+\mathbbm{h}),
		\end{cases}
	\end{align}
	where $\Theta_{jm} (f)=((v_jv_m-1)\mu^{1/2},f)$, $\Theta(f)=(\Theta_{jm}(f))_{1\le j,m\le 3}$, $\Lambda_i(f)=\frac{1}{10}((\vert v\vert^2-5)v_j \mu^{1/2},f)$, $\Lambda(f)=(\Lambda_j (f))_{1\le j\le 3}$ and
	\begin{align*}
		\mathbbm{r}= -v\cdot \nabla_x (\mathbf{I}-\mathbf{P})f,\ \mathbbm{h}=-L (\mathbf{I}-\mathbf{P})f+H.
	\end{align*}
	It is direct to see that
	\begin{align*}
		\vert \Theta(f)\vert, \vert \Lambda (f)\vert\le C\min\{\Vert f\Vert_{L^2_v}, \Vert f\Vert_{H^{s*}_v}\}.
	\end{align*}
	
Let $\hat{a}^f$, $\hat{b}^f:=(\hat{b}_1^f,\hat{b}_2^f,\hat{b}_3^f)$, $\hat{c}^f$ and $\hat{f}$ denote the Fourier transformation of $a^f,b^f,c^f$ and $f$ respectively, then we rewrite the above system in terms of $\hat{a}^f$, $\hat{b}^f$, $\hat{c}^f$ and $\hat{f}$:	
	\begin{align}\label{mac.law.fourier}
		\begin{cases}
			\partial_t \hat{a}^f +i\xi\cdot \hat{b}^f=(\mu^{1/2},\widehat{H}),\\
			\partial_t \hat{b}^f +i\xi  (\hat{a}^f+2\hat{c}^f)+i\xi\cdot \Theta ((\mathbf{I}-\mathbf{P}) \hat{f})=(v\mu^{1/2},\widehat{H}),\\
			\partial_t \hat{c}^f +\frac{1}{3}i\xi\cdot \hat{b}^f +\frac{1}{6}i\xi \cdot \Lambda ((\mathbf{I}-\mathbf{P}) \hat{f})=(\frac{1}{6}(|v|^2-3)\mu^{1/2},\widehat{H}),\\
			\partial_t[\Theta_{jm}((\mathbf{I}-\mathbf{P}) \hat{f})+2\hat{c}^f\delta_{ jm}]+i\xi_j \hat{b}^f_m+i\xi_m \hat{b}^f_j=\Theta_{jm}(\hat{\mathbbm{r}}+\hat{\mathbbm{h}}),\\
			\partial_t \Lambda_j((\mathbf{I}-\mathbf{P}) \hat{f})+i\xi_j \hat{c}^f = \Lambda_j(\hat{\mathbbm{r}}+\hat{\mathbbm{h}}).
		\end{cases}
	\end{align}
From the fifth equation of \eqref{mac.law.fourier}, we integrate by parts to get
\begin{align}
	\Big(\partial_t \Lambda_j ((\mathbf{I}-\mathbf{P})\hat{f}), \frac{i\xi_j \hat{c}^f}{1+\vert \xi\vert^2} \Big)+\frac{\xi_j^2}{1+\vert \xi\vert^2} \vert \hat{c}^f\vert^2 =\Big(\Lambda_j (\hat{\mathbbm{r}}+\hat{\mathbbm{h}}), \frac{i\xi_j \hat{c}^f}{1+\vert \xi\vert^2} \Big).\label{ap.adp1}
\end{align}
Combining \eqref{ap.adp1} and the third equation of \eqref{mac.law.fourier} gives
\begin{align*}
	&\partial_t \Big(\Lambda_j ((\mathbf{I}-\mathbf{P})\hat{f}), \frac{i\xi_j \hat{c}^f}{1+\vert \xi\vert^2}\Big)
	 +\frac{\xi_j^2}{1+\vert \xi\vert^2} \vert \hat{c}^f\vert^2\\
	=&\Big(\Lambda_j ((\mathbf{I}-\mathbf{P})\hat{f}), \frac{\xi_j}{1+\vert \xi\vert^2}\Big[ \frac{1}{3}\xi\cdot \hat{b}^f+\frac{1}{6}\xi\cdot \Lambda((\mathbf{I}-\mathbf{P})\hat{f}-(\frac{1}{6}(|v|^2-3)\mu^{1/2},\widehat{H}))\Big]\Big)\notag\\
	&+\Big(\Lambda_j (\hat{\mathbbm{r}}+\hat{\mathbbm{h}}), \frac{i\xi_j \hat{c}^f}{1+\vert \xi\vert^2}\Big).
\end{align*}
Integrating the above equation over $[0,T]$, one has
\begin{align*}
	&\int^T_0 \frac{\xi_j^2}{1+\vert \xi\vert^2}\vert \hat{c}^f\vert^2 dt \notag\\
	 \le& C\Vert \hat{f}_0\Vert^2_{L^2_v}+ C\Vert  \hat{f} \Vert^2_{L^\infty_T L^2_v}  +\ka_1^2\int^T_0 \frac{\vert \xi\vert^2}{1+\vert \xi\vert^2}(\vert \hat{b}^f\vert^2+\vert \hat{c}^f\vert^2)dt\\
	& +C_{\ka_1}\int^T_0 \frac{\vert \Lambda_j((\mathbf{I}-\mathbf{P})\hat{f})\vert^2}{1+\vert \xi\vert^2}dt+C\int^T_0 \frac{\vert \xi\vert^2}{1+\vert \xi\vert^2}\vert \Lambda ((\mathbf{I}-\mathbf{P})\hat{f})\vert^2dt\\
	&+\int^T_0 \frac{1}{1+\vert \xi\vert^2}\vert (\frac{1}{6}(|v|^2-3)\mu^{1/2},\widehat{H})\vert^2dt+ C_{\ka_1}\int^T_0 \frac{\vert \Lambda_j(\hat{\mathbbm{r}}+\hat{\mathbbm{h}})\vert^2}{1+\vert \xi\vert^2} dt,
\end{align*}
where $\ka_1>0$ is a small constant which will be chosen later.
We take the summation of $j=1$, $2$, $3$  to get
\begin{align*}
	&\frac{\vert \xi\vert }{\sqrt{1+\vert \xi\vert^2}} \Big(\int^T_0 \vert \hat{c}^f\vert^2 dt\Big)^{1/2} \notag\\
	\le& C\Vert \hat{f}_0\Vert_{L^2_v}+ C\Vert  \hat{f} \Vert_{L^\infty_T L^2_v}\notag\\
	& +C\ka_1 \frac{\vert \xi\vert}{\sqrt{1+\vert \xi\vert^2}}\Big[ \Big( \int^T_0 \vert \hat{b}^f\vert^2dt\Big)^{1/2}+ \Big(\int^T_0 \vert \hat{c}^f\vert^2 dt\Big)^{1/2}\Big]\notag\\
	& +C_{\ka_1} \Big(\int^T_0 \frac{\vert \Lambda((\mathbf{I}-\mathbf{P})\hat{f})\vert^2}{1+\vert \xi\vert^2}dt \Big)^{1/2}+C\frac{\vert \xi\vert}{\sqrt{1+\vert \xi\vert^2}} \Vert (\mathbf{I}-\mathbf{P})\hat{f}\Vert_{L^2_T H^{s*}_v} \notag \\
	&+\Big(\int^T_0 \frac{1}{1+\vert \xi\vert^2}\vert (\widehat{H}, \mu^{1/4})\vert^2 dt \Big)^{1/2} +C_{\ka_1}  \Big( \int^T_0 \frac{\vert \Lambda(\hat{\mathbbm{r}}+\hat{\mathbbm{h}})\vert^2}{1+\vert \xi\vert^2} dt\Big)^{1/2}.
\end{align*}
Then by the fact that
\begin{align*}
	&\Big(\int^T_0 \frac{\vert \Lambda(\hat{\mathbbm{r}}+\hat{\mathbbm{h}})\vert^2}{1+\vert \xi\vert^2}dt +\int^T_0 \frac{\vert \Theta(\hat{\mathbbm{r}}+\hat{\mathbbm{h}})\vert^2}{1+\vert \xi\vert^2}dt \Big)^{1/2}\notag\\\le& C \Vert (\mathbf{I}-\mathbf{P})f\Vert_{ L^2_T H^{s*}_v}+C \Big(\int^T_0 \frac{\vert (\widehat{H}, \mu^{1/4})\vert^2}{1+\vert \xi\vert^2}dt \Big)^{1/2},
\end{align*}
it holds
\begin{align}\label{macro.c}
	&\frac{\vert \xi\vert }{\sqrt{1+\vert \xi\vert^2}} \Big(\int^T_0 \vert \hat{c}^f\vert^2 dt\Big)^{1/2} \notag\\
	\le& C\Vert \hat{f}_0\Vert_{L^2_v}+ C\Vert  \hat{f} \Vert_{L^\infty_T L^2_v} +C\ka_1 \frac{\vert \xi\vert}{\sqrt{1+\vert \xi\vert^2}} \Big[\Big( \int^T_0 \vert \hat{b}^f\vert^2dt\Big)^{1/2}+\Big(\int^T_0 \vert \hat{c}^f\vert^2 dt\Big)^{1/2}\Big]\notag\\
	& +C_{\ka_1} \Vert (\mathbf{I}-\mathbf{P})\hat{f}\Vert_{L^2_T H^{s*}_v}+C_{\ka_1}\Big(\int^T_0 \frac{1}{1+\vert \xi\vert^2}\vert (\widehat{H}, \mu^{1/4})\vert^2 dt \Big)^{1/2}.
\end{align}
Similar arguments show that
\begin{align}\label{macro.b}
	&\frac{\vert \xi\vert}{\sqrt{1+\vert \xi\vert^2}}\Big(\int^T_0 \vert \hat{b}^f\vert^2dt \Big)^{1/2} 
	 \notag   \\\le& C\Vert \hat{f}_0\Vert_{L^2_v}+C\Vert  \hat{f} \Vert_{L^\infty_T L^2_v}+C\ka_2 \frac{\vert \xi\vert}{\sqrt{1+\vert \xi\vert^2}} \Big[\Big( \int^T_0 \vert \hat{a}^f\vert^2dt\Big)^{1/2}+\Big(\int^T_0 \vert \hat{b}^f\vert^2 dt\Big)^{1/2}\Big]   \notag   \\
	 & +C_{\ka_2} \Vert (\mathbf{I}-\mathbf{P})\hat{f}\Vert_{L^2_T H^{s*}_v}+C_{\ka_2}\Big(\int^T_0 \frac{1}{1+\vert \xi\vert^2}\vert (\widehat{H}, \mu^{1/4})\vert^2 dt \Big)^{1/2}\notag\\
	 &+C_{\ka_2}\frac{\vert \xi\vert}{\sqrt{1+\vert \xi\vert^2}} \Big(\int^T_0 \vert \hat{c}^f\vert^2dt \Big)^{1/2},
\end{align}
and 
\begin{align}\label{macro.a}
	&\frac{\vert \xi\vert}{\sqrt{1+\vert \xi\vert^2}}\Big( \int^T_0 \vert \hat{a}^f\vert^2dt\Big)^{1/2}\notag\\ 
	\le& C\Vert \hat{f}_0\Vert_{L^2_v} +C\Vert  \hat{f} \Vert_{L^\infty_T L^2_v}+C \frac{\vert \xi\vert}{\sqrt{1+\vert \xi\vert^2}} \Big[\Big( \int^T_0 \vert \hat{b}^f\vert^2dt\Big)^{1/2}+\Big(\int^T_0 \vert \hat{c}^f\vert^2 dt\Big)^{1/2}\Big] \notag \\
	& +C \Vert (\mathbf{I}-\mathbf{P})\hat{f}\Vert_{L^2_T H^{s*}_v}.
\end{align}
Hence, by a linear combination of \eqref{macro.c}, \eqref{macro.b} and \eqref{macro.a}, then choosing $\ka_1$ and $\ka_2$ to be small enough and taking $L^p_\xi$ norm, we obtain \eqref{macro}.
\end{proof}

 Notice that our study of macroscopic quantities is valid for $\max\{-3,-3/2-2s\}<\ga\leq1$ and will be used later when we estimate the soft potentials. Substituting $f=g_2$ and $H=\CL_B g_1+\Ga(g_2,g_2)$ into \eqref{macro}, we obtain the following result.
\begin{lemma}\label{lem: macro a priori}
	Let $\max\{-3,-3/2-2s\}<\ga\leq1$ and $1\le p\leq \infty$.
	Let $g_2$ be a solution to \eqref{hatg2} with an inhomogeneous term $H=H(t,x,v)$. Recalling $A$ and $M$ are two constants in the definition of $L_B$ in \eqref{defLB}. Then it holds that
	\begin{align}\label{lowmacro}
		\Big\Vert \frac{\vert \nabla_x \vert}{\langle \nabla_x \rangle } (\hat{a},\hat{b},\hat{c}) \Big\Vert_{L^p_\xi L^2_T} 
		&\le \Vert \widehat{g_2}\Vert_{L^p_\xi L^\infty_T L^2_v} + \Vert (\mathbf{I}-\mathbf{P})\widehat{g_2}\Vert_{L^p_\xi L^2_T H^{s*}_v} \notag\\
		&\quad+C_{A,M}\|\widehat{g_1}\|_{L^p_\xi L^2_T L^2_v}+C \Vert \widehat{g_2}\Vert_{L^p_\xi L^\infty_T L^2_v}\Vert \widehat{g_2}\Vert_{L^1_\xi L^2_T H^{s*}_v}, 
	\end{align}
	for any $T>0$, where $C_{A,M}$ depends only on $A$ and $M$.
\end{lemma}
\begin{proof}
	We have the inequality which is proved in \cite[Lemma 3.4]{DSY} that for any $\phi \in \mathcal{S}(\mathbb{R}^3_v)$ and $1\leq p\leq\infty$, there exists $C_\phi>0$ such that it holds
	\begin{align*}
		\Big\Vert  \Big( \int^T_0 \vert (\hat{\Gamma}(\hat{f},\hat{g}),\phi )\vert^2 dt \Big)^{1/2} \Big\Vert_{L^p_\xi}  \le C_\phi\Vert f \Vert_{L^p_\xi L^\infty_T L^2_v}  \Vert g\Vert_{L^1_\xi L^2_T H^{s*}_v},
	\end{align*}
	for any $T>0$.
	Then it is straightforward to get 
	\begin{align}\label{macroga}
		\Big(\int_{\mathbb{R}^3} \Big(\int^T_0 \frac{1}{1+\vert \xi\vert^2}\vert (\hat{\Gamma}(\widehat{g_2},\widehat{g_2}), \mu^{1/4})\vert^2 dt \Big)^{p/2} d\xi\Big)^{1/p}\leq C \Vert \widehat{g_2}\Vert_{L^p_\xi L^\infty_T L^2_v}\Vert \widehat{g_2}\Vert_{L^1_\xi L^2_T H^{s*}_v}.
	\end{align}
	By the definition of $\CL_B$ in \eqref{defLB}, one has
	\begin{align}\label{macrog1}
			&\Big(\int_{\mathbb{R}^3} \Big(\int^T_0 \frac{1}{1+\vert \xi\vert^2}\vert (\CL_B \widehat{g_1}, \mu^{1/4})\vert^2 dt \Big)^{p/2} d\xi\Big)^{1/p}\notag\\
			&\leq C_{A,M} \Big(\int_{\mathbb{R}^3} \Big(\int^T_0 \frac{1}{1+\vert \xi\vert^2} \|\widehat{g_1}\|_{L^2_v}^2 dt \Big)^{p/2} d\xi\Big)^{1/p}\notag\\
			&\leq C_{A,M}\|\widehat{g_1}\|_{L^p_\xi L^2_T L^2_v}.
	\end{align}
	Therefore, \eqref{lowmacro} follows from \eqref{macroga} and \eqref{macrog1}.
\end{proof}
Then the following lemma is a direct result from Lemma \ref{lem: micro a priori} and Lemma \ref{lem: macro a priori}.
\begin{lemma}\label{g2hard}
	Let $\max\{-3,-3/2-2s\}<\ga\leq1$, $p>1$ and $g_2$ be a solution to \eqref{hatg2}.
	There exists $C>0$ such that 
	\begin{align}
		&\Vert \widehat{g_2}\Vert_{L^1_\xi L^\infty_T L^2_v}+\Vert (\mathbf{I}-\mathbf{P})\widehat{g_2} \Vert_{L^1_\xi L^2_T H^{s*}_v}+\Big\Vert \frac{\vert \nabla_x \vert}{\langle \nabla_x \rangle } (\hat{a},\hat{b},\hat{c}) \Big\Vert_{L^1_\xi L^2_T} \notag\\
		\le &C_{A,M}\Vert \widehat{g_2}\Vert^\frac{1}{2}_{L^1_\xi L^\infty_T L^2_v}\big(\int^T_0\|(1+t)^\rho g_1(t)\|^2_{X^*_{8-\ga/2}}dt\big)^\frac{1}{4}+C\Vert \widehat{g_2}\Vert_{L^1_\xi L^\infty_T L^2_v}\Vert (\mathbf{I}-\mathbf{P})\widehat{g_2} \Vert_{L^1_\xi L^2_T H^{s*}_v}\notag\\
		&+C_{A,M}\|\widehat{g_1}\|_{L^1_\xi L^2_T L^2_v}+C\Vert \widehat{g_2}\Vert_{L^1_\xi L^\infty_T L^2_v}\Vert  (1+t)^{\sigma/2} \widehat{g_2}\Vert_{L^1_\xi L^\infty_T L^2_v},\label{g2hardl1}
	\end{align}
	and 
	\begin{align}
		&\Vert \widehat{g_2}\Vert_{L^p_\xi L^\infty_T L^2_v}+\Vert (\mathbf{I}-\mathbf{P})\widehat{g_2} \Vert_{L^p_\xi L^2_T H^{s*}_v}+\Big\Vert \frac{\vert \nabla_x \vert}{\langle \nabla_x \rangle } (\hat{a},\hat{b},\hat{c}) \Big\Vert_{L^p_\xi L^2_T} \notag\\
		\le &C_{A,M}\Vert \widehat{g_2}\Vert^\frac{1}{2}_{L^p_\xi L^\infty_T L^2_v}\Vert (1+t)^\rho\widehat{g_1} \Vert^\frac{1}{2}_{L^p_\xi L^2_TL^2_v}+C\Vert \widehat{g_2}\Vert_{L^p_\xi L^\infty_T L^2_v} \Vert (\mathbf{I}-\mathbf{P})\widehat{g_2}\Vert_{L^1_\xi L^2_T H^{s*}_v}\notag\\
		&+C_{A,M}\|\widehat{g_1}\|_{L^p_\xi L^2_T L^2_v}+C\Vert \widehat{g_2}\Vert_{L^p_\xi L^\infty_T L^2_v}\Vert  (1+t)^{\sigma/2} \widehat{g_2}\Vert_{L^1_\xi L^\infty_T L^2_v} , \label{g2hardlp}
	\end{align}
	for any $T>0$ and $\rho>1$, where $C_{A,M}$ depends only on $A$ and $M$ which are two constants in the definition of $L_B$.
	\end{lemma}

	Later we will need the time-weighted macroscopic estimate. We can prove the following lemma for $\max\{-3,-3/2-2s\}<\ga\leq1$, which will be later used in both hard and soft cases.
	\begin{lemma}
			Let $\max\{-3,-3/2-2s\}<\ga\leq1$, $0<s<1$, $3/2<p\le\infty$, $\sigma=3(1-1/p)-2\ep$ where $\ep>0$ is arbitrarily small and $f$ be a solution to \eqref{eq: linearized BE} with an inhomogeneous term $H=H(t,x,v)$. Then it holds that
		\begin{align}\label{tmacro}
		 &\Big\Vert (1+t)^{\si/2}\frac{\vert \nabla_x \vert}{\langle \nabla_x \rangle } (\hat{a}^f,\hat{b}^f,\hat{c}^f) \Big\Vert_{L^1_\xi L^2_T} \notag\\
			\le& C \Vert \hat{f}_0 \Vert_{L^1_\xi L^2_v} + \Vert(1+t)^{\si/2} \hat{f}\Vert_{L^1_\xi L^\infty_T L^2_v} + \Vert(1+t)^{\si/2} (\mathbf{I}-\mathbf{P})\hat{f}\Vert_{L^1_\xi L^2_T H^{s*}_v} \notag\\
			&+\big\Vert \frac{\vert \nabla_x \vert}{\langle \nabla_x \rangle } (\hat{a}^f,\hat{b}^f,\hat{c}^f)\Big \Vert_{L^p_\xi L^2_T}+\int_{\mathbb{R}^3} \Big(\int^T_0(1+t)^{\si/2} \frac{1}{1+\vert \xi\vert^2}\vert (\widehat{H}, \mu^{1/4})\vert^2 dt \big)^{1/2} d\xi, 
		\end{align}
		for any $T>0$.
	\end{lemma}
	\begin{proof}
		Integrating by parts, it is direct to get
		\begin{align}\label{ap.adp2}
			&(1+t)^\sigma \partial_t \big(\Lambda_j ((\mathbf{I}-\mathbf{P})\hat{f}), \frac{i\xi_j \hat{c}^f}{1+\vert \xi\vert^2}\big)\notag\\
			=&\partial_t \big[  (1+t)^\sigma \big(\Lambda_j ((\mathbf{I}-\mathbf{P})\hat{f}), \frac{i\xi_j \hat{c}^f}{1+\vert \xi\vert^2}\big)\Big]-\sigma (1+t)^{\sigma-1} \Big(\Lambda_j ((\mathbf{I}-\mathbf{P})\hat{f}), \frac{i\xi_j \hat{c}^f}{1+\vert \xi\vert^2}\Big).
		\end{align}
		From \eqref{ap.adp2} and the third equation of \eqref{mac.law.fourier}, one has
		\begin{align*}
			&\partial_t \Big[  (1+t)^\sigma \Big(\Lambda_j ((\mathbf{I}-\mathbf{P})\hat{f}), \frac{i\xi_j \hat{c}^f}{1+\vert \xi\vert^2}\Big)\Big]
			+(1+t)^\sigma\frac{\xi_j^2}{1+\vert \xi\vert^2} \vert \hat{c}^f\vert^2\\
			=&\big(\Lambda_j ((\mathbf{I}-\mathbf{P})\hat{f}),(1+t)^\sigma \frac{\xi_j}{1+\vert \xi\vert^2}\Big[ \frac{1}{3}\xi\cdot \hat{b}^f+\frac{1}{6}\xi\cdot \Lambda((\mathbf{I}-\mathbf{P})\hat{f}-(\frac{1}{6}(|v|^2-3)\mu^{1/2},\widehat{H}))\Big]\big)\notag\\
			&+\Big(\Lambda_j (\hat{\mathbbm{r}}+\hat{\mathbbm{h}}),(1+t)^\sigma \frac{i\xi_j \hat{c}^f}{1+\vert \xi\vert^2}\Big)+\sigma (1+t)^{\sigma-1} \Big(\Lambda_j ((\mathbf{I}-\mathbf{P})\hat{f}), \frac{i\xi_j \hat{c}}{1+\vert \xi\vert^2}\Big),
		\end{align*}
		which further yields by similar calculations as in the proof of Lemma \ref{lem: macro} that
		\begin{align}\label{tmacro.c}
			&\frac{\vert \xi\vert }{\sqrt{1+\vert \xi\vert^2}} \Big(\int^T_0 (1+t)^\sigma\vert \hat{c}^f\vert^2 dt\Big)^{1/2} \notag\\
			\le& C\Vert \hat{f}_0(\xi)\Vert_{L^2_v}+ C\Vert(1+t)^\sigma  \hat{f}(\xi) \Vert_{L^\infty_T L^2_v}\notag\\
			& +C\ka_1 \frac{\vert \xi\vert}{\sqrt{1+\vert \xi\vert^2}} \Big[\Big( \int^T_0(1+t)^\sigma \vert \hat{b}^f\vert^2dt\Big)^{1/2}+\Big(\int^T_0 (1+t)^\sigma\vert \hat{c}^f\vert^2 dt\Big)^{1/2}\Big]\notag\\
			& +C_{\ka_1} \Vert(1+t)^\sigma (\mathbf{I}-\mathbf{P})\hat{f}\Vert_{L^2_T H^{s*}_v}+C_{\ka_1}\Big(\int^T_0 (1+t)^\sigma\frac{1}{1+\vert \xi\vert^2}\vert (\widehat{H}, \mu^{1/4})\vert^2 dt \Big)^{1/2}\notag\\
			&+\Big(\int^T_0 \sigma (1+t)^{\sigma-1} \Big(\Lambda_j ((\mathbf{I}-\mathbf{P})\hat{f}), \frac{i\xi_j \hat{c}^f}{1+\vert \xi\vert^2}\Big)dt \Big)^{1/2}.
		\end{align}
		We can see that except for the last term on the right hand side above, other terms can be estimated in the same way as in Lemma \ref{lem: macro}. Hence, we now estimate the last term above as follows:
		\begin{align*}
			&\int^T_0 (1+t)^{\sigma-1} \Big(\Lambda_j ((\mathbf{I}-\mathbf{P})\hat{f}), \frac{i\xi_j \hat{c}^f}{1+\vert \xi\vert^2}\Big) dt \\
			&\le \int^T_0 (1+t)^\sigma \Big( \ka_1\frac{\vert \xi\vert^2}{1+\vert \xi\vert^2} \vert \hat{c}^f\vert^2+ C_{\ka_1} \frac{\Vert (\mathbf{I}-\mathbf{P})\hat{f} \Vert_{H^{s*}_v}^2}{1+\vert \xi\vert^2} \Big)dt.
		\end{align*}
		We substitute the above inequality into \eqref{tmacro.c} and integrate over $\R^3_\xi$ to get
		\begin{align}\label{macrot.c}
			&\int_{\R^3}\frac{\vert \xi\vert }{\sqrt{1+\vert \xi\vert^2}} \Big(\int^T_0 (1+t)^\sigma\vert \hat{c}^f\vert^2 dt\Big)^{1/2}d\xi \notag\\
			\le& C\Vert \hat{f}_0\Vert_{L^1_\xi L^2_v}+ C\Vert(1+t)^\sigma  \hat{f} \Vert_{L^1_\xi L^\infty_T L^2_v}+C_{\ka_1} \Vert(1+t)^\sigma (\mathbf{I}-\mathbf{P})\hat{f}\Vert_{L^1_\xi L^2_T H^{s*}_v}\notag\\
			& +C\ka_1 \int_{\R^3}\frac{\vert \xi\vert}{\sqrt{1+\vert \xi\vert^2}} \Big[\Big( \int^T_0(1+t)^\sigma \vert \hat{b}^f\vert^2dt\Big)^{1/2}+\Big(\int^T_0 (1+t)^\sigma\vert \hat{c}^f\vert^2 dt\Big)^{1/2}\Big]d\xi\notag\\
			& +C_{\ka_1}\int_{\R^3}\Big(\int^T_0 (1+t)^\sigma\frac{1}{1+\vert \xi\vert^2}\vert (\widehat{H}, \mu^{1/4})\vert^2 dt \Big)^{1/2}d\xi.
		\end{align}
		Similarly,
		\begin{align}\label{macrot.b}
			&\frac{\vert \xi\vert}{\sqrt{1+\vert \xi\vert^2}}\Big(\int^T_0(1+t)^\sigma \vert \hat{b}^f\vert^2dt \Big)^{1/2} 
			\notag   \\\le& C\Vert \hat{f}_0\Vert_{L^2_v}+C\Vert (1+t)^\sigma \hat{f} \Vert_{L^\infty_T L^2_v}\notag\\
			&+C\ka_2 \frac{\vert \xi\vert}{\sqrt{1+\vert \xi\vert^2}} \Big[\Big( \int^T_0 (1+t)^\sigma\vert \hat{a}^f\vert^2dt\Big)^{1/2}+\Big(\int^T_0(1+t)^\sigma \vert \hat{b}^f\vert^2 dt\Big)^{1/2}\Big]  \notag   \\
			& +C_{\ka_2} \Vert(1+t)^\sigma (\mathbf{I}-\mathbf{P})\hat{f}\Vert_{L^2_T H^{s*}_v}+C_{\ka_2}\Big(\int^T_0(1+t)^\sigma \frac{1}{1+\vert \xi\vert^2}\vert (\widehat{H}, \mu^{1/4})\vert^2 dt \Big)^{1/2}\notag\\
			&+C_{\ka_2}\frac{\vert \xi\vert}{\sqrt{1+\vert \xi\vert^2}} \Big(\int^T_0(1+t)^\sigma \vert \hat{c}^f\vert^2dt \Big)^{1/2}.
		\end{align}
		However, the estimate of $\hat{a}^f$ is slightly different from above. We first repeat similar procedure as above to get
		\begin{align}\label{tmacro.a}
			&\frac{\vert \xi\vert}{\sqrt{1+\vert \xi\vert^2}}\Big( \int^T_0(1+t)^\sigma \vert \hat{a}^f\vert^2dt\Big)^{1/2}\notag\\ 
			\le& C\Vert \hat{f}_0\Vert_{L^2_v} +C\Vert(1+t)^\sigma  \hat{f} \Vert_{L^\infty_T L^2_v}\notag\\
			&+C \frac{\vert \xi\vert}{\sqrt{1+\vert \xi\vert^2}} \Big[\Big( \int^T_0(1+t)^\sigma \vert \hat{b}^f\vert^2dt\Big)^{1/2}+\Big(\int^T_0(1+t)^\sigma \vert \hat{c}^f\vert^2 dt\Big)^{1/2}\Big] \notag \\
			& +C \Vert(1+t)^\sigma (\mathbf{I}-\mathbf{P})\hat{f}\Vert_{L^2_T H^{s*}_v}+\Big(\int^T_0 \sigma  (1+t)^{\sigma-1} \big(\hat{b}^f, \frac{i\xi\hat{a}^f}{1+\vert \xi\vert^2}\big) dt\Big)^\frac{1}{2}.
		\end{align}
		We should take care of the last term above for high and low frequency parts. If $\vert \xi\vert \ge 1$, by Cauchy-Schwarz inequality and the fact that $|\xi|\leq|\xi|^2$, it is straightforward to see
		\begin{align}\label{highpart}
			&\Big[ \int^T_0 (1+t)^{\sigma-1} \Big(\hat{b}^f, \frac{i\xi\hat{a}^f}{1+\vert \xi\vert^2}\Big) dt \Big]^{1/2} \notag\\
			\le& \frac{\vert \xi \vert}{\sqrt{1+\vert \xi\vert^2}} \Big[ \ka_3 \Big(\int^T_0 (1+t)^\sigma \vert \hat{a}^f\vert^2 dt\Big)^{1/2}
			+C_{\ka_3} \Big( \int^T_0(1+t)^\sigma \vert \hat{b}^f\vert^2 dt\Big)^{1/2} \Big].
		\end{align}
		If $\vert \xi\vert \le 1$ and $\frac{1}{1+t}\leq |\xi|$, it holds
				\begin{align}\label{lowparttlarge}
			&\Big[ \int^T_0 (1+t)^{\sigma-1} \Big(\hat{b}^f, \frac{i\xi\hat{a}^f}{1+\vert \xi\vert^2}\Big) dt \Big]^{1/2} \notag\\
			\le& \frac{\vert \xi \vert}{\sqrt{1+\vert \xi\vert^2}} \Big[ \ka_3 \Big(\int^T_0 (1+t)^\sigma \vert \hat{a}^f\vert^2 dt\Big)^{1/2}
			+C_{\ka_3} \Big( \int^T_0(1+t)^\sigma \vert \hat{b}^f\vert^2 dt\Big)^{1/2} \Big].
		\end{align}
		If $\vert \xi\vert \le 1$ and $\frac{1}{1+t}\geq |\xi|$, one has
		\begin{align*}
			&\Big[ \int^T_0 (1+t)^{\sigma-1} \Big(\hat{b}^f, \frac{i\xi\hat{a}^f}{1+\vert \xi\vert^2}\Big) dt \Big]^{1/2}\notag\\
			\le&\Big[ \int^T_0 |\xi|^{2-\sigma} |\hat{b}^f| |\hat{a}^f|dt \Big]^{1/2} \notag\\
			\le& |\xi|^{-\si/2}\Big[|\xi|\Big( \int^T_0 \vert \hat{a}^f\vert^2 dt\Big)^{1/2}+|\xi|\Big( \int^T_0\vert \hat{b}^f\vert^2 dt\Big)^{1/2}\Big],
		\end{align*}
		which further implies
	\begin{align}	\label{lowparttsmall}
			&\int_{\R^3}\Big[ \int^T_0 (1+t)^{\sigma-1} \Big(\hat{b}^f, \frac{i\xi\hat{a}^f}{1+\vert \xi\vert^2}\Big) dt \Big]^{1/2}d\xi\notag\\
		\le& (\int_{\R^3}|\xi|^{-\si p'/2 }d\xi)^{1/p'}\Big(\int_{\R^3}\Big[|\xi|^p\Big( \int^T_0 \vert \hat{a}^f\vert^2 dt\Big)^{p/2}+|\xi|^p\Big( \int^T_0\vert \hat{b}^f\vert^2 dt\Big)^{p/2}\Big]d\xi\Big)^{1/p}\notag\\
		\leq& C\Big\Vert \frac{\vert \nabla_x \vert}{\langle \nabla_x \rangle } (\hat{a}^f,\hat{b}^f)\Big \Vert_{L^p_\xi L^2_T},
		\end{align}
		by H\"older's inequality and the fact that $-\si p'/2>-3$.
		We combine \eqref{tmacro.a}, \eqref{highpart}, \eqref{lowparttlarge} and \eqref{lowparttsmall}, then choose $\ka_3$ to be small to get
		\begin{align}\label{macrot.a}
			&\int_{\R^3}\frac{\vert \xi\vert}{\sqrt{1+\vert \xi\vert^2}}\Big( \int^T_0(1+t)^\sigma \vert \hat{a}^f\vert^2dt\Big)^{1/2}d\xi\notag\\ 
			\le& C\Vert \hat{f}_0\Vert_{L^1_\xi L^2_v} +C\Vert(1+t)^\sigma  \hat{f} \Vert^2_{L^1_\xi L^\infty_T L^2_v}\notag\\
			&+C \int_{\R^3}\frac{\vert \xi\vert}{\sqrt{1+\vert \xi\vert^2}} \Big[\Big( \int^T_0(1+t)^\sigma \vert \hat{b}^f\vert^2dt\Big)^{1/2}+\Big(\int^T_0(1+t)^\sigma \vert \hat{c}^f\vert^2 dt\Big)^{1/2}\Big]d\xi \notag \\
			& +C \Vert(1+t)^\sigma (\mathbf{I}-\mathbf{P})\hat{f}\Vert_{L^1_\xi L^2_T H^{s*}_v}+\Big\Vert \frac{\vert \nabla_x \vert}{\langle \nabla_x \rangle } (\hat{a}^f,\hat{b}^f)\Big \Vert_{L^p_\xi L^2_T}.
		\end{align}
		Hence, we deduce \eqref{tmacro} by collecting \eqref{macrot.c}, \eqref{macrot.b} and \eqref{macrot.a}, and choosing $\ka_1$ and $\ka_2$ to be small. 
	\end{proof}

	Substituting $f=g_2$ and $H=\CL_B g_1+\Ga(g_2,g_2)$ into \eqref{tmacro} and using similar arguments as in the proof of Lemma \ref{lem: macro a priori}, we obtain the following result.
	\begin{lemma}\label{lem.twemp}
		Let $\max\{-3,-3/2-2s\}<\ga \leq1$, $0<s<1$, $3/2<p\le \infty$, $\sigma=3(1-1/p)-2\ep$ where $\ep>0$ is arbitrarily small and $g_2$ be a solution to \eqref{hatg2}. Then, it holds
		\begin{align}
			&\Big \Vert (1+t)^{\si/2} \frac{\vert \nabla_x \vert}{\langle \nabla_x \rangle} (\hat{a},\hat{b},\hat{c}) \Big\Vert_{L^1_\xi L^2_T}\notag \\
			\le& C\Vert (1+t)^{\sigma/2} \widehat{g_2}\Vert_{L^1_\xi L^\infty_T L^2_v}
			+C\Vert (1+t)^{\sigma/2} (\mathbf{I}-\mathbf{P})\widehat{g_2}\Vert_{L^1_\xi L^2_T H^{s*}_v} \notag \\
			&+C\Big\Vert \frac{\vert \nabla_x \vert}{\langle \nabla_x \rangle}(\hat{a},\hat{b},\hat{c})\Big\Vert_{L^p_\xi L^2_T}+C_{A,M}\|(1+t)^{\si/2}\widehat{g_1}\|_{L^1_\xi L^2_T L^2_v}\notag\\
			&+C\Vert (1+t)^{\sigma/2} \widehat{g_2}\Vert_{L^1_\xi L^\infty_T L^2_v} \Vert \widehat{g_2}\Vert_{L^1_\xi L^2_T H^{s*}_v},
			\label{ineq: macro time-weighted}
		\end{align}
		for any $T>0$, where $C_{A,M}$ depends only on $A$ and $M$ which are two constants in the definition of $L_B$.
	\end{lemma}

	\subsection{Decay estimates for $\ga+2s\geq 0$}
	
With Lemma \ref{g2hard} and Lemma \ref{lem.twemp}, we can prove the time-weighted estimates now. Note that now we require $\ga+2s\geq0$. Our main purpose of this subsection is to deduce Lemma \ref{timehard}.
\begin{lemma}\label{lem: micro time-weighted}
	Let $-2s\leq\ga\leq1$, $0<s<1$, $3/2<p\leq \infty$ and $\sigma=3(1-1/p)-2\ep$ where $\ep>0$ is arbitrarily small. For any $T>0$ and $\rho>1$,
	it holds that
	\begin{align}
		&\Vert  (1+t)^{\sigma/2} \widehat{g_2}\Vert_{L^1_\xi L^\infty_T L^2_v}  
		+\Vert (1+t)^{\sigma/2} (\mathbf{I}-\mathbf{P})\widehat{g_2} \Vert_{L^1_\xi L^2_T H^{s*}_v} \notag \\
		\le &C_{A,M}\Vert  (1+t)^{\sigma/2}\widehat{g_2}\Vert^\frac{1}{2}_{L^1_\xi L^\infty_T L^2_v}\Big(\int^T_0\|(1+t)^\rho g_1(t)\|^2_{X^*_{8-\ga/2}}dt\Big)^\frac{1}{4}\notag\\&+C\Vert(1+t)^{\sigma/2} \widehat{g_2}\Vert_{L^1_\xi L^\infty_T L^2_v}\Vert (\mathbf{I}-\mathbf{P})\widehat{g_2} \Vert_{L^1_\xi L^2_T H^{s*}_v}+C\Vert  (1+t)^{\sigma/2} \widehat{g_2}\Vert^2_{L^1_\xi L^\infty_T L^2_v}\notag\\&+C\eta \Big\Vert (1+t)^{\sigma/2} \frac{\vert \nabla_x \vert}{\langle \nabla_x \rangle } (\hat{a},\hat{b},\hat{c}) \Big\Vert_{L^1_\xi L^2_T}\notag\\
		&
		+ C_\eta  \Big(\|(\mathbf{I}-\mathbf{P})\widehat{g_2}\|_{L^1_\xi L^2_TH^{s*}_v}+\Big\|\frac{\vert \nabla_x \vert}{\langle \nabla_x \rangle}(\hat{a},\hat{b},\hat{c})\Big\|_{L^1_\xi L^2_T}+\Big\|\frac{\vert \nabla_x \vert}{\langle \nabla_x \rangle}(\hat{a},\hat{b},\hat{c})\Big\|_{L^p_\xi L^2_T}\Big),
		\label{ineq: micro time-weighted}
	\end{align}
	where $\eta>0$ is an arbitrarily small constant and $C_{A,M}$ depends only on $A$ and $M$ which are two constants in the definition of $L_B$.
\end{lemma}
\begin{proof}
	Similar arguments as in the proof of  Lemma \ref{lem: micro a priori} show that
	\begin{align}\label{tg2}
		&\Vert  (1+t)^{\sigma/2} \widehat{g_2}\Vert_{L^1_\xi L^\infty_T L^2_v}  
		+\Vert (1+t)^{\sigma/2} (\mathbf{I}-\mathbf{P})\widehat{g_2} \Vert_{L^1_\xi L^2_T H^{s*}_v} \notag \\
		\le& C_{A,M}\Vert  (1+t)^{\sigma/2}\widehat{g_2}\Vert^\frac{1}{2}_{L^1_\xi L^\infty_T L^2_v}\Big(\int^T_0\|(1+t)^\rho g_1(t)\|^2_{X^*_{8-\ga/2}}dt\Big)^\frac{1}{4}\notag\\
		&+C\Vert(1+t)^{\sigma/2} \widehat{g_2}\Vert_{L^1_\xi L^\infty_T L^2_v}\Vert (\mathbf{I}-\mathbf{P})\widehat{g_2} \Vert_{L^1_\xi L^2_T H^{s*}_v}\notag\\
		&+C\Vert  (1+t)^{\sigma/2} \widehat{g_2}\Vert^2_{L^1_\xi L^\infty_T L^2_v}+C\sqrt{\sigma} \int_{\mathbb{R}^3}\Big( \int^T_0 (1+t)^{\sigma -1} \Vert \widehat{g_2}\Vert_{L^2_v}^2 dt\Big)^{1/2} d\xi.
	\end{align}
It remains to estimate the last term above. It is direct to get
\begin{align}\label{J1J2}
	&\sqrt{\sigma} \int_{\mathbb{R}^3}\Big( \int^T_0 (1+t)^{\sigma -1} \Vert \widehat{g_2}\Vert_{L^2_v}^2 dt\Big)^{1/2} d\xi\notag\\
	\leq& \sqrt{\sigma} \int_{\R^3}\Big( \int^T_0 (1+t)^{\sigma -1} \Vert(\mathbf{I}-\mathbf{P}) \widehat{g_2}\Vert_{L^2_v}^2 dt\Big)^{1/2} d\xi\notag\\
	&+\sqrt{\sigma} \int_{\R^3}\Big( \int^T_0 (1+t)^{\sigma -1} \Vert\mathbf{P} \widehat{g_2}\Vert_{L^2_v}^2 dt\Big)^{1/2} d\xi\notag\\
	=&J_1+J_2.
\end{align}
For $J_1$, by \begin{align*}(1+t)^{\si-1}&\leq (1+t)^{\si-1}\mathbbm{1}_{(1+t)^{-1}\leq \eta/\sqrt\si}+(1+t)^{\si-1}\mathbbm{1}_{(1+t)^{-1}\geq \eta/\sqrt\si}\notag\\
	&\leq \frac{\eta}{\sqrt{\sigma}} (1+t)^\si +C_\eta,\end{align*}
it holds that
\begin{align}\label{J1}
	J_1\leq& \sqrt{\sigma} \int_{\R^3}\Big( \int^T_0 \{\frac{\eta}{\sqrt{\sigma}} (1+t)^\si +C_\eta\} \Vert(\mathbf{I}-\mathbf{P}) \widehat{g_2}\Vert_{L^2_v}^2 dt\Big)^{1/2} d\xi\notag\\
	\leq& \eta  \Vert (1+t)^{\sigma/2}(\mathbf{I}-\mathbf{P})\widehat{g_2} \Vert_{L^1_\xi L^2_T H^{s*}_v}+C_\eta  \|(\mathbf{I}-\mathbf{P})\widehat{g_2}\|_{L^1_\xi L^2_TH^{s*}_v}.
\end{align}
For $J_2$, we decompose
\begin{align}\label{J2decompose}
	J_2=&\sqrt{\sigma} \int_{|\xi|> 1}\Big( \int^T_0 (1+t)^{\sigma -1} \Vert\mathbf{P} \widehat{g_2}\Vert_{L^2_v}^2 dt\Big)^{1/2} d\xi\notag\\
	&+\sqrt{\sigma} \int_{|\xi|\leq 1}\Big( \int^T_0 (1+t)^{\sigma -1} \Vert\mathbf{P} \widehat{g_2}\Vert_{L^2_v}^2 dt\Big)^{1/2} d\xi\notag\\=&J_{21}+J_{22}.
\end{align}
Similar arguments as in \eqref{J1} show that
\begin{align}\label{J21}
	J_{21}\leq C\eta \Big\Vert (1+t)^{\sigma/2} \frac{\vert \nabla_x \vert}{\langle \nabla_x \rangle } (\hat{a},\hat{b},\hat{c}) \Big\Vert_{L^1_\xi L^2_T}+C_\eta \Big\Vert  \frac{\vert \nabla_x \vert}{\langle \nabla_x \rangle } (\hat{a},\hat{b},\hat{c}) \Big\Vert_{L^1_\xi L^2_T}.
\end{align}
By the fact that \begin{align*}(1+t)^{\si-1}&\leq (1+t)^{\si-1}\mathbbm{1}_{(1+t)^{-1}\leq \eta|\xi|^2/\sqrt\si}+(1+t)^{\si-1}\mathbbm{1}_{(1+t)^{-1}\geq \eta|\xi|^2/\sqrt\si}\notag\\
	&\leq \frac{\eta|\xi|^2}{\sqrt{\sigma}} (1+t)^\si +C_\eta|\xi|^{-2(\si-1)},\end{align*}
we get
\begin{align}\label{J22}
	J_{22}\leq& \sqrt{\sigma} \int_{|\xi|\leq 1}\Big( \int^T_0 \{\frac{\eta}{\sqrt{\sigma}} (1+t)^\si |\xi|^2 +C_\eta |\xi|^{-2(\si-1)}\} \Vert \mathbf{P}\widehat{g_2}\Vert_{L^2_v}^2 dt\Big)^{1/2} d\xi\notag\\
	\le&C\eta \Big\Vert (1+t)^{\sigma/2} \frac{\vert \nabla_x \vert}{\langle \nabla_x \rangle } (\hat{a},\hat{b},\hat{c}) \Big\Vert_{L^1_\xi L^2_T}+C_\eta
	\int_{|\xi|\leq 1}|\xi|^{-\si}\Big( \int^T_0 |\xi|^2 \Vert\mathbf{P} \widehat{g_2}\Vert_{L^2_v}^2 dt\Big)^{1/2} d\xi.
	\end{align}
For the second term on the right hand side in \eqref{J22}, an application of H\"older's inequality leads to
\begin{align}\label{J221}
	&C_\eta
	\int_{|\xi|\leq 1}|\xi|^{-\si}\Big( \int^T_0 |\xi|^2 \Vert\mathbf{P} \widehat{g_2}\Vert_{L^2_v}^2 dt\Big)^{1/2} d\xi\notag\\\leq& C_\eta
	\Big(\int_{|\xi|\leq 1}|\xi|^{-\si p'}d\xi\Big)^{1/p'}\Big(\int_{|\xi|\leq 1}\Big( \int^T_0 |\xi|^2 \Vert\mathbf{P} \widehat{g_2}\Vert_{L^2_v}^2 dt\Big)^{p/2}d\xi\Big)^{1/p} \notag\\
	\leq& C_\eta\Big\|\frac{\vert \nabla_x \vert}{\langle \nabla_x \rangle}(\hat{a},\hat{b},\hat{c})\Big\|_{L^p_\xi L^2_T},
\end{align}
where $1/p'+1/p=1$. The last inequality above holds since $\si p'>-3$ by our choice of $\si$ and $p$.
It follows by \eqref{J2decompose}, \eqref{J21}, \eqref{J22} and \eqref{J221} that
\begin{align}\label{J2}	
	J_2
	\leq& C\eta \Big\Vert (1+t)^{\sigma/2} \frac{\vert \nabla_x \vert}{\langle \nabla_x \rangle } (\hat{a},\hat{b},\hat{c}) \Big\Vert_{L^1_\xi L^2_T}\notag\\
	&
	+ C_\eta  \Big(\Big\|\frac{\vert \nabla_x \vert}{\langle \nabla_x \rangle}(\hat{a},\hat{b},\hat{c})\Big\|_{L^1_\xi L^2_T}+\Big\|\frac{\vert \nabla_x \vert}{\langle \nabla_x \rangle}(\hat{a},\hat{b},\hat{c})\Big\|_{L^p_\xi L^2_T}\Big).
\end{align}

Thus, \eqref{ineq: micro time-weighted} holds by collecting \eqref{tg2}, \eqref{J1J2}, \eqref{J1} and \eqref{J2}, and choosing $\eta$ to be small enough.
\end{proof}

Combining Lemma \ref{lem.twemp} and Lemma \ref{lem: micro time-weighted}, the following result can be directly deduced by a linear combination of \eqref{ineq: macro time-weighted} and \eqref{ineq: micro time-weighted}, and choosing $\eta$ to be small.
\begin{lemma}\label{timehard}
	Let $-2s\le\ga\le1$, $0<s<1$, $3/2<p\leq \infty$ and $\sigma=3(1-1/p)-2\ep$ where $\ep>0$ is arbitrarily small. For any $T>0$ and $\rho>1$,
	it holds that
	\begin{align}\label{esthatg2}
		&\Vert  (1+t)^{\sigma/2} \widehat{g_2}\Vert_{L^1_\xi L^\infty_T L^2_v}  
		+\Vert (1+t)^{\sigma/2} (\mathbf{I}-\mathbf{P})\widehat{g_2} \Vert_{L^1_\xi L^2_T H^{s*}_v}+\Big \Vert (1+t)^{\sigma/2} \frac{\vert \nabla_x \vert}{\langle \nabla_x \rangle} (\hat{a},\hat{b},\hat{c}) \Big\Vert_{L^1_\xi L^2_T} \notag \\
		\le &C_{A,M}\Vert  (1+t)^{\sigma/2}\widehat{g_2}\Vert^\frac{1}{2}_{L^1_\xi L^\infty_T L^2_v}\Big(\int^T_0\|(1+t)^\rho g_1(t)\|^2_{X^*_{8-\ga/2}}dt\Big)^\frac{1}{4}\notag\\&+C\Vert(1+t)^{\sigma/2} \widehat{g_2}\Vert_{L^1_\xi L^\infty_T L^2_v}\Vert (\mathbf{I}-\mathbf{P})\widehat{g_2} \Vert_{L^1_\xi L^2_T H^{s*}_v}+C\Vert  (1+t)^{\sigma/2} \widehat{g_2}\Vert^2_{L^1_\xi L^\infty_T L^2_v}\notag\\&+ C \Big(\|(\mathbf{I}-\mathbf{P})\widehat{g_2}\|_{L^1_\xi L^2_TH^{s*}_v}+\Big\|\frac{\vert \nabla_x \vert}{\langle \nabla_x \rangle}(\hat{a},\hat{b},\hat{c})\Big\|_{L^1_\xi L^2_T}+\Big\Vert \frac{\vert \nabla_x \vert}{\langle \nabla_x \rangle}(\hat{a},\hat{b},\hat{c})\Big\Vert_{L^p_\xi L^2_T}\Big)\notag\\
		&+C_{A,M}\|(1+t)^{\si/2}\widehat{g_1}\|_{L^1_\xi L^2_T L^2_v}.
	\end{align}
\end{lemma}

\subsection{Decay estimates for $\ga+2s< 0$}
In this case, we should estimate $g_2$ with an additional velocity weight. Then we need the following lemma on $L$, where the proof is given in \cite{GS, AMUXY-2012-JFA, DLYZ-VMB}.
\begin{lemma}\label{wgesL}
	It holds
	\begin{equation*}
		\left( Lg,\langle v\rangle^{2k}g\right)_{L^2_v}\geq
		\delta\Vert g \Vert^2_{H^{s*}_{v,k}}
		-C_k\Vert g\Vert_{L^2_v({B_R})}^2,
	\end{equation*}
	where $\de$, $C_k>0$, and  $B_R$ denotes the closed ball in $\R^3_v$ with center at the origin and radius $R>0$.
\end{lemma}

We now bound the velocity weighted $L^1_\xi L^\infty_T L^2_v$ norm of $g_2$.
\begin{lemma}\label{g2l1soft}
	Let $k\geq 0$, $\max\{-3,-3/2-2s\}<\ga<2s$, $0<s<1$ and $g_2$ be a solution to \eqref{hatg2}.
	There exists $C_k>0$ such that 
	\begin{align}\label{g2soft}
		&\Vert \widehat{g_2}\Vert_{L^1_\xi L^\infty_T L^2_{v,k}}+\Vert (\mathbf{I}-\mathbf{P})\widehat{g_2} \Vert_{L^1_\xi L^2_T H^{s*}_{v,k}}+\Big\Vert \frac{\vert \nabla_x \vert}{\langle \nabla_x \rangle } (\hat{a},\hat{b},\hat{c}) \Big\Vert_{L^1_\xi L^2_T}\notag\\
		\le &C_{A,M}\|\widehat{g_2}\|^\frac{1}{2}_{L^1_\xi L^\infty_TL^2_{v,k}}\Big(\int^T_0\|(1+t)^\rho g_1(t)\|^2_{X^*_{k+8-\ga/2}}dt\Big)^\frac{1}{4}\notag\\&+C_k \Vert \widehat{g_2} \Vert_{L^1_\xi L^\infty_T L^2_{v,k}} \Vert(1+t)^{\sigma/2} \widehat{g_2}\Vert_{L^1_\xi L^\infty_T L^2_v}  \notag \\
		&+C_{A,M}\|\widehat{g_1}\|_{L^1_\xi L^2_T L^2_v}+C_k \Vert \widehat{g_2} \Vert_{L^1_\xi L^\infty_T L^2_{v,k}} \Vert(\mathbf{I}-\mathbf{P}) \widehat{g_2}\Vert_{L^1_\xi L^2_T H^{s*}_{v,k}},
	\end{align}
 where $C_{A,M}$ depends only on $A$ and $M$ which are two constants in the definition of $L_B$.
\end{lemma}
\begin{proof}
	Our proof mainly contains three parts which are the high frequency part of $\widehat{g_2}$ and the low frequency part of $(\mathbf{I}-\mathbf{P}) \widehat{g_2}$. Recalling the low frequency part of $\mathbf{P}\widehat{g_2}$ is given in \eqref{lowmacro}, then we take suitable linear combination to get our result. We start with the low frequency part of $(\mathbf{I}-\mathbf{P}) \widehat{g_2}$.
	Recall that $g_2(0) = 0$, applying $(\mathbf{I}-\mathbf{P})$ to \eqref{hatg2}, multiplying with $\langle v\rangle^{2k} (\mathbf{I}-\mathbf{P}) \bar{\hat{g}}_2$, integrating over $[0,T]\times \mathbb{R}^3_v $ and using Lemma \ref{wgesL}, one gets
	\begin{align*}
		&\Vert (\mathbf{I}-\mathbf{P}) \widehat{g_2} \Vert_{L^\infty_T L^2_{v,k}} + \Vert (\mathbf{I}-\mathbf{P}) \widehat{g_2} \Vert_{L^2_T H^{s*}_{v,k}} 
		\\\le&  C\Vert( (\mathbf{I}-\mathbf{P}) \CL_B \widehat{g_1},\langle v\rangle^{2k}(\mathbf{I}-\mathbf{P})\widehat{g_2}) \Vert_{L^2_T}
		+ C_k\Vert (\mathbf{I}-\mathbf{P}) \widehat{g_2} \Vert_{L^2_T L^2_v(B_R)} \\
		+& C\Big(\int^T_0 \Big\vert \int_{\mathbb{R}^3}  \mathrm{Re}(\hat{\Gamma}(f,f)-(\mathbf{I}-\mathbf{P})[iv\cdot \xi \mathbf{P}\widehat{g_2}] \notag\\
		&\qquad\qquad\qquad\qquad+\mathbf{P}[iv\cdot \xi (\mathbf{I}-\mathbf{P}) \widehat{g_2}],\langle v\rangle^{2k}(\mathbf{I}-\mathbf{P})\widehat{g_2}))  dv\Big\vert dt\Big)^{1/2}.
	\end{align*}
Integrating over $\R^3_\xi$, the last two term on the right hand side above can be bounded by similar arguments as in \cite[Lemma 4.3]{DSY}, then we have
	\begin{align}\label{highmacrog20}
	&\Vert  (\mathbf{I}-\mathbf{P}) \widehat{g_2} \Vert_{L^1_{\vert \xi\vert\le 1} L^\infty_T L^2_{v,k}} 
	+ \Vert  (\mathbf{I}-\mathbf{P}) \widehat{g_2} \Vert_{L^1_{\vert \xi\vert\le 1} L^2_T H^{s*}_{v,k}} \notag \\
	\le&  C\Vert( (\mathbf{I}-\mathbf{P}) \CL_B \widehat{g_1},\langle v\rangle^{2k}(\mathbf{I}-\mathbf{P})\widehat{g_2}) \Vert_{L^1_{\vert \xi\vert\le 1}L^2_T}
	+ C_k\Vert (\mathbf{I}-\mathbf{P}) \widehat{g_2} \Vert_{L^1_\xi L^2_T H^{s*}_v} \notag \\
	&
	+ C_k \Big\Vert \frac{\vert \nabla_x \vert}{\langle \nabla_x \rangle } (\hat{a},\hat{b},\hat{c}) \Big\Vert_{L^1_\xi L^2_T}  +C_k \Vert \widehat{g_2} \Vert_{L^1_\xi L^\infty_T L^2_{v,k}} \Vert(1+t)^{\sigma/2} \widehat{g_2}\Vert_{L^1_\xi L^\infty_T L^2_v}\notag \\
	&+C_k \Vert \widehat{g_2} \Vert_{L^1_\xi L^\infty_T L^2_{v,k}} \Vert(\mathbf{I}-\mathbf{P}) \widehat{g_2}\Vert_{L^1_\xi L^2_T H^{s*}_v}.
\end{align}
Now we concentrate on $\Vert( (\mathbf{I}-\mathbf{P}) \CL_B \widehat{g_1},\langle v\rangle^{2k}(\mathbf{I}-\mathbf{P})\widehat{g_2}) \Vert_{L^1_{\vert \xi\vert\le 1}L^2_T}$.
By the definition of $\CL_B$ in \eqref{defLB} and the fact that $\Vert  \mathbf{P} g \Vert_{L^2_{v,k}}\le C \Vert g\Vert_{L^2_v}$, one has
\begin{align*}
	&\Big(\int^T_0\rmre((\mathbf{I}-\mathbf{P})\CL_B\widehat{g_1},\langle v\rangle^{2k}(\mathbf{I}-\mathbf{P})\widehat{g_2})(t,\xi)dt\Big)^\frac{1}{2}\notag\\
	\leq& C\Big(\int^T_0\|(\mathbf{I}-\mathbf{P})\CL_B\widehat{g_1}(t,\xi)\|_{L^2_{v,k}}\|(\mathbf{I}-\mathbf{P})\widehat{g_2}(t,\xi)\|_{L^2_{v,k}}dt\Big)^\frac{1}{2}\notag\\
	\leq& C\|(\mathbf{I}-\mathbf{P})\CL_B\widehat{g_1}(\xi)\|^\frac{1}{2}_{L^\infty_TL^2_{v,k}}\Big(\int^T_0\|(\mathbf{I}-\mathbf{P})\widehat{g_2}(t,\xi)\|_{L^2_{v,k}}dt\Big)^\frac{1}{2}\notag\\
	\leq& C_{A,M}\|\widehat{g_2}(\xi)\|^\frac{1}{2}_{L^\infty_TL^2_{v,k}}\Big(\int^T_0\|\widehat{g_1}(t,\xi)\|_{L^2_{v,k}}dt\Big)^\frac{1}{2},
\end{align*}
which yields by Cauchy-Schawarz inequality and Sobolev imbedding that
\begin{align}\label{highmacrog1}
	&\Vert( (\mathbf{I}-\mathbf{P}) \CL_B \widehat{g_1},\langle v\rangle^{2k}(\mathbf{I}-\mathbf{P})\widehat{g_2}) \Vert_{L^1_{\vert \xi\vert\le 1}L^2_T}\notag\\
	\leq& C_{A,M}\|\widehat{g_2}(\xi)\|^\frac{1}{2}_{L^1_\xi L^\infty_TL^2_{v,k}}\Big(\int_{\R^3}\int^T_0\|\widehat{g_1}(t,\xi)\|_{L^2_{v,k}}dtd\xi\Big)^\frac{1}{2}\notag\\
	\leq& C_{A,M}\|\widehat{g_2}(\xi)\|^\frac{1}{2}_{L^1_\xi L^\infty_TL^2_{v,k}}\Big(\int^T_0\int_{R^3_\xi}(1+t)^\rho \langle\xi\rangle^{3+}\|\widehat{g_1}(t,\xi)\|^2_{L^2_{v,k}}d\xi dt\Big)^\frac{1}{4}\notag\\
	\leq& C_{A,M}\|\widehat{g_2}(\xi)\|^\frac{1}{2}_{L^1_\xi L^\infty_TL^2_{v,k}}\Big(\int^T_0\|(1+t)^\rho g_1(t)\|^2_{H^2_xL^2_{v,k}}dt\Big)^\frac{1}{4}.
\end{align}
Recalling the definition of $X^*_k$ in \eqref{X*k}, then it follows from \eqref{highmacrog20} and \eqref{highmacrog1} that
\begin{align}
	&\Vert  (\mathbf{I}-\mathbf{P}) \widehat{g_2} \Vert_{L^1_{\vert \xi\vert\le 1} L^\infty_T L^2_v} 
	+ \Vert  (\mathbf{I}-\mathbf{P}) \widehat{g_2} \Vert_{L^1_{\vert \xi\vert\le 1} L^2_T H^{s*}_{v,k}} \notag \\
	\le&  C_{A,M}\|\widehat{g_2}\|^\frac{1}{2}_{L^1_\xi L^\infty_TL^2_{v,k}}\Big(\int^T_0\|(1+t)^\rho g_1(t)\|^2_{X^*_{k+8-\ga/2}}dt\Big)^\frac{1}{4}
	+ C_k\Vert (\mathbf{I}-\mathbf{P}) \widehat{g_2} \Vert_{L^1_\xi L^2_T H^{s*}_v} 
	  \notag \\
	&+ C_k \Big\Vert \frac{\vert \nabla_x \vert}{\langle \nabla_x \rangle } (\hat{a},\hat{b},\hat{c}) \Big\Vert_{L^1_\xi L^2_T}+C_k \Vert \widehat{g_2} \Vert_{L^1_\xi L^\infty_T L^2_{v,k}} \Vert(1+t)^{\sigma/2} \widehat{g_2}\Vert_{L^1_\xi L^\infty_T L^2_v}  \notag \\
	&+C_k \Vert \widehat{g_2} \Vert_{L^1_\xi L^\infty_T L^2_{v,k}} \Vert(\mathbf{I}-\mathbf{P}) \widehat{g_2}\Vert_{L^1_\xi L^2_T H^{s*}_{v,k}}.\notag
\end{align}
We turn to the high frequency part of $\widehat{g_2}$. By similar calculations as above, we obtain
\begin{align}\label{lowg2}
	&\Vert \widehat{g_2} \Vert_{L^1_{\vert \xi\vert \ge 1} L^\infty_T L^2_{v,k}} + \Vert \widehat{g_2} \Vert_{L^1_{\vert \xi\vert \ge 1} L^2_T H^{s*}_{v,k}} \notag \\
	\le& C_{A,M}\|\widehat{g_2}\|^\frac{1}{2}_{L^1_\xi L^\infty_TL^2_{v,k}}\Big(\int^T_0\|(1+t)^\rho g_1(t)\|^2_{X^*_{k+8-\ga/2}}dt\Big)^\frac{1}{4} + C_k\Vert (\mathbf{I}-\mathbf{P}) \widehat{g_2} \Vert_{L^1_\xi L^2_T H^{s*}_v} 
	\notag \\
	&+ C_k \Big\Vert \frac{\vert \nabla_x \vert}{\langle \nabla_x \rangle } (\hat{a},\hat{b},\hat{c}) \Big\Vert_{L^1_\xi L^2_T}+C_k \Vert \widehat{g_2} \Vert_{L^1_\xi L^\infty_T L^2_{v,k}} \Vert(1+t)^{\sigma/2} \widehat{g_2}\Vert_{L^1_\xi L^\infty_T L^2_v}  \notag \\
	&+C_k \Vert \widehat{g_2} \Vert_{L^1_\xi L^\infty_T L^2_{v,k}} \Vert(\mathbf{I}-\mathbf{P}) \widehat{g_2}\Vert_{L^1_\xi L^2_T H^{s*}_{v,k}}.
\end{align}
We note that \eqref{g2hardl1} is also valid for soft potentials. Then taking suitable linear combination with \eqref{ineq: L^1_k micro a priori}, \eqref{lowmacro}, \eqref{g2hardl1} and \eqref{lowg2}, we get \eqref{g2soft}.
\end{proof}
Then we have the following result for $L^p_\xi$ norm of $\widehat{g_2}$. The proof is very similar to Lemma \ref{g2l1soft} and thus omitted.
\begin{lemma}
	Let $k\geq 0$, $\max\{-3,-3/2-2s\}<\ga<2s$, $0<s<1$, $3/2<p\leq\infty$ and $g_2$ be a solution to \eqref{hatg2}.
	There exist $C_k,C_{A,M}>0$ such that 
	\begin{align}\label{g2lpsoft}
		&\Vert \widehat{g_2}\Vert_{L^p_\xi L^\infty_T L^2_{v,k}}+\Vert (\mathbf{I}-\mathbf{P})\widehat{g_2} \Vert_{L^p_\xi L^2_T H^{s*}_{v,k}}+\Big\Vert \frac{\vert \nabla_x \vert}{\langle \nabla_x \rangle } (\hat{a},\hat{b},\hat{c}) \Big\Vert_{L^p_\xi L^2_T}\notag\\
		\le &C_{A,M}\Vert \widehat{g_2}\Vert^\frac{1}{2}_{L^p_\xi L^\infty_T L^2_{v,k}}\Vert (1+t)^\rho\widehat{g_1} \Vert^\frac{1}{2}_{L^p_\xi L^2_TL^2_{v,k}}+C_k \Vert \widehat{g_2} \Vert_{L^p_\xi L^\infty_T L^2_{v,k}} \Vert(1+t)^{\sigma/2} \widehat{g_2}\Vert_{L^1_\xi L^\infty_T L^2_v}  \notag \\
		&+C_{A,M}\|\widehat{g_1}\|_{L^p_\xi L^2_T L^2_v}+C_k \Vert \widehat{g_2} \Vert_{L^p_\xi L^\infty_T L^2_{v,k}} \Vert(\mathbf{I}-\mathbf{P}) \widehat{g_2}\Vert_{L^1_\xi L^2_T H^{s*}_{v,k}},
	\end{align}
	where $C_{A,M}$ depends only on $A$ and $M$ which are two constants in the definition of $L_B$.
\end{lemma}
With the $L^1_\xi\cap L^p_\xi$ estimate of $\widehat{g_2}$, we now turn to the time weighted estimate. Note that in this case, we should bound 
$$\Vert (1+t)^{\si/2}\widehat{g_2}\Vert_{L^1_\xi L^\infty_T L^2_{v,k}}+\Vert(1+t)^{\si/2} (\mathbf{I}-\mathbf{P})\widehat{g_2} \Vert_{L^1_\xi L^2_T H^{s*}_{v,k}}+\Big\Vert (1+t)^{\si/2}\frac{\vert \nabla_x \vert}{\langle \nabla_x \rangle } (\hat{a},\hat{b},\hat{c}) \Big\Vert_{L^1_\xi L^2_T}.
$$
Recall that in Lemma \ref{lem: micro time-weighted} for $0\leq\ga\leq1$, except for the additional term 
$$\sqrt{\sigma} \int_{|\xi|\leq 1}\Big( \int^T_0 (1+t)^{\sigma -1} \Vert \widehat{g_2}\Vert_{L^2_v}^2 dt\Big)^{1/2} d\xi,
$$
all other terms can be calculated in the same way as the case without the time weight. Now for $\ga<0$, the approach is similar as in  Lemma \ref{lem: micro time-weighted}, and we should focus on the extra terms induced by the derivative of time weight.
\begin{lemma}\label{lemmasoft}
	Let $k\geq 0$, $\max\{-3,-3/2-2s\}<\ga<2s$, $0<s<1$, $3/2<p\le \infty$, $g_2$ be a solution to \eqref{hatg2} and $\sigma=3\Big(1-\frac{1}{p}\Big) -2\ep$ where $\ep>0$ is arbitrarily small. There exist constants $C_k,C_{A,M}>0$ such that
\begin{align}\label{g2timesoft}
	&\Vert (1+t)^{\si/2}\widehat{g_2}\Vert_{L^1_\xi L^\infty_T L^2_{v,k}}+\Vert (1+t)^{\si/2}(\mathbf{I}-\mathbf{P})\widehat{g_2} \Vert_{L^1_\xi L^2_T H^{s*}_{v,k}}+\Big\Vert(1+t)^{\si/2} \frac{\vert \nabla_x \vert}{\langle \nabla_x \rangle } (\hat{a},\hat{b},\hat{c}) \Big\Vert_{L^1_\xi L^2_T}\notag\\
	\le &C_{A,M}\|(1+t)^{\si/2}\widehat{g_2}\|^\frac{1}{2}_{L^1_\xi L^\infty_TL^2_{v,k}}\Big(\int^T_0\|(1+t)^\rho g_1(t)\|^2_{X^*_{k+8-\ga/2}}dt\Big)^\frac{1}{4}+C_k  \Vert(1+t)^{\sigma/2} \widehat{g_2}\Vert^2_{L^1_\xi L^\infty_T L^2_{v,k}}  \notag \\
	&+C_k\|(1+t)^{\si/2}\widehat{g_1}\|_{L^1_\xi L^2_T L^2_v}+C_k \Vert(1+t)^{\si/2}\widehat{g_2} \Vert_{L^1_\xi L^\infty_T L^2_{v,k}} \Vert(\mathbf{I}-\mathbf{P}) \widehat{g_2}\Vert_{L^1_\xi L^2_T H^{s*}_{v,k}}\notag\\
	&+ C_k(\Vert  (\mathbf{I}-\mathbf{P}) \widehat{g_2} \Vert_{L^1_\xi L^2_T H^{s*}_{v,k+\si|\ga+2s|/2}}+\Big\Vert \frac{\vert \nabla_x \vert}{\langle \nabla_x \rangle } (\hat{a},\hat{b},\hat{c}) \Big\Vert_{L^1_\xi L^2_T}+\Big\Vert \frac{\vert \nabla_x \vert}{\langle \nabla_x \rangle } (\hat{a},\hat{b},\hat{c}) \Big\Vert_{L^p_\xi L^2_T}),
\end{align}	
where $C_{A,M}$ depends only on $A$ and $M$ which are two constants in the definition of $L_B$.
\end{lemma}

\begin{proof}
Similar arguments as in the proof of Lemma \ref{g2l1soft}, by dividing $(1+t)^{\si/2}\widehat{g_2}$ into high frequency part of $(1+t)^{\si/2}\widehat{g_2}$, the low frequency part of $(1+t)^{\si/2}(\mathbf{I}-\mathbf{P}) \widehat{g_2}$ and the low frequency part of $(1+t)^{\si/2}\mathbf{P}\widehat{g_2}$, taking linear combination, one gets
\begin{align}\label{g2ts}
&\Vert (1+t)^{\si/2}\widehat{g_2}\Vert_{L^1_\xi L^\infty_T L^2_{v,k}}+\Vert (1+t)^{\si/2}(\mathbf{I}-\mathbf{P})\widehat{g_2} \Vert_{L^1_\xi L^2_T H^{s*}_{v,k}}+\Big\Vert(1+t)^{\si/2} \frac{\vert \nabla_x \vert}{\langle \nabla_x \rangle } (\hat{a},\hat{b},\hat{c}) \Big\Vert_{L^1_\xi L^2_T}\notag\\
\le &C_{A,M}\|(1+t)^{\si/2}\widehat{g_2}\|^\frac{1}{2}_{L^1_\xi L^\infty_TL^2_{v,k}}\Big(\int^T_0\|(1+t)^\rho g_1(t)\|^2_{X^*_{k+8-\ga/2}}dt\Big)^\frac{1}{4}\notag \\
&+C_k  \Vert(1+t)^{\sigma/2} \widehat{g_2}\Vert^2_{L^1_\xi L^\infty_T L^2_{v,k}} +C_k \Vert(1+t)^{\si/2}\widehat{g_2} \Vert_{L^1_\xi L^\infty_T L^2_{v,k}} \Vert(\mathbf{I}-\mathbf{P}) \widehat{g_2}\Vert_{L^1_\xi L^2_T H^{s*}_{v,k}}\notag\\
& +C_k\|(1+t)^{\si/2}\widehat{g_1}\|_{L^1_\xi L^2_T L^2_v}+C\int_{\vert \xi\vert\le 1} \Big(\int^T_0 (1+t)^{\sigma-1}\Vert  (\mathbf{I}-\mathbf{P}) \widehat{g_2} \Vert_{L^2_{v,k}}^2 dt \Big)^{1/2} d\xi\notag\\
&+C\int_{\vert \xi\vert \ge 1} \big(\int^T_0 (1+t)^{\sigma-1} \Vert  \widehat{g_2} \Vert_{L^2_{v,k}}^2 dt \big)^\frac{1}{2} d\xi+C\int_{\mathbb{R}^3} \big(\int^T_0 (1+t)^{\sigma-1} \Vert \widehat{g_2}\Vert_{L^2_v}^2 dt \big)^\frac{1}{2}d\xi.
\end{align}
We only need to estimate the last three terms on the right hand side above. Denote
\begin{align*}
	J_3&=\int_{\vert \xi\vert\le 1} \Big(\int^T_0 (1+t)^{\sigma-1}\Vert  (\mathbf{I}-\mathbf{P}) \widehat{g_2} \Vert_{L^2_{v,k}}^2 dt \Big)^{1/2} d\xi,\notag\\
	J_4&=\int_{\vert \xi\vert \ge 1} \Big(\int^T_0 (1+t)^{\sigma-1} \Vert  \widehat{g_2} \Vert_{L^2_{v,k}}^2 dt \Big)^{1/2} d\xi,\notag\\
	J_5&=\int_{\mathbb{R}^3} \Big(\int^T_0 (1+t)^{\sigma-1} \Vert \widehat{g_2}\Vert_{L^2_v}^2 dt \Big)^{1/2} d\xi.
\end{align*}
We first consider $J_3$ and $J_4$ for two cases. When $\frac{1}{1+t}\leq \eta\langle v\rangle^{\ga+2s}$, by $1\leq 2\frac{\vert \xi \vert^2}{1+|\xi|^2}$ for $|\xi|\ge1$ and $\Vert \widehat{g_2}\Vert_{L^2_{v,k}}^2\leq C\Vert (\mathbf{I}-\mathbf{P})\widehat{g_2}\Vert_{L^2_{v,k}}^2+C\Vert \mathbf{P}\widehat{g_2}\Vert_{L^2_{v,k}}^2$, we have
\begin{align*}
	J_3+J_4\leq& C\int_{\R^3} \Big(\int^T_0 (1+t)^{\sigma-1}\Vert  (\mathbf{I}-\mathbf{P}) \widehat{g_2} \Vert_{L^2_{v,k}}^2 dt \Big)^{1/2} d\xi\notag\\&+C\int_{\vert \xi\vert\ge 1} \Big(\int^T_0 (1+t)^{\sigma-1}\frac{\vert \xi \vert^2}{1+|\xi|^2}\Vert  \mathbf{P} \widehat{g_2} \Vert_{L^2_{v,k}}^2dt \Big)^{1/2} d\xi\notag\\
	\leq& C\int_{\R^3} \Big(\int^T_0 (1+t)^{\sigma}\Vert  (\mathbf{I}-\mathbf{P}) \widehat{g_2} \Vert_{L^2_{k+\ga/2+s}}^2 dt \Big)^{1/2} d\xi\notag\\&+C\int_{\vert \xi\vert\ge 1} \Big(\int^T_0 (1+t)^{\sigma}\frac{\vert \xi \vert^2}{1+|\xi|^2}\Vert  \mathbf{P} \widehat{g_2} \Vert_{L^2_{k+\ga/2+s}}^2dt \Big)^{1/2} d\xi.
	\end{align*}
	By the fact that $\Vert  \mathbf{P} \widehat{g_2} \Vert_{L^2_{l}}\leq C_l|(\hat{a},\hat{b},\hat{c})|$, one gets
	\begin{align}\label{J3J4smallt}
	J_3+J_4&\leq C\eta \Vert (1+t)^{\sigma/2} (\mathbf{I}-\mathbf{P}) \widehat{g_2} \Vert_{L^1_\xi L^2_T H^{s*}_{v,k}}+C_k\eta\Big\Vert (1+t)^{\si/2}\frac{\vert \nabla_x \vert}{\langle \nabla_x \rangle } (\hat{a},\hat{b},\hat{c}) \Big\Vert_{L^1_\xi L^2_T}.
\end{align}
When $\frac{1}{1+t}> \eta\langle v\rangle^{\ga+2s}$, since $\si-1>0$, a direct calculation shows that $(1+t)^{\sigma-1}\leq C_\eta \langle v\rangle^{(\sigma-1)|\ga+2s|}$, which yields
\begin{align}\label{J3J4larget}
	J_3+J_4&\leq C\int_{\R^3} \Big(\int^T_0 (1+t)^{\sigma-1}\Vert  (\mathbf{I}-\mathbf{P}) \widehat{g_2} \Vert_{L^2_{v,k}}^2 dt \Big)^{1/2} d\xi\notag\\&\quad+C\int_{\vert \xi\vert\ge 1} \Big(\int^T_0 (1+t)^{\sigma-1}\frac{\vert \xi \vert^2}{1+|\xi|^2}\Vert  \mathbf{P} \widehat{g_2} \Vert_{L^2_{v,k}}^2dt \Big)^{1/2} d\xi\notag\\
	&\leq C_\eta \Vert  (\mathbf{I}-\mathbf{P}) \widehat{g_2} \Vert_{L^1_\xi L^2_T H^{s*}_{v,k+\si|\ga+2s|/2}}+C_{k,\eta}\Big\Vert \frac{\vert \nabla_x \vert}{\langle \nabla_x \rangle } (\hat{a},\hat{b},\hat{c}) \Big\Vert_{L^1_\xi L^2_T}.
\end{align}
Then it follows from \eqref{J3J4smallt} and \eqref{J3J4larget} that
\begin{align}\label{J3J4}
	J_3+J_4&\leq C\eta \Vert (1+t)^{\sigma/2} (\mathbf{I}-\mathbf{P}) \widehat{g_2} \Vert_{L^1_\xi L^2_T H^{s*}_{v,k}}+C_k\eta\Big\Vert (1+t)^{\si/2}\frac{\vert \nabla_x \vert}{\langle \nabla_x \rangle } (\hat{a},\hat{b},\hat{c}) \Big\Vert_{L^1_\xi L^2_T}\notag\\&\quad +C_\eta \Vert  (\mathbf{I}-\mathbf{P}) \widehat{g_2} \Vert_{L^1_\xi L^2_T H^{s*}_{v,k+\si|\ga+2s|/2}}+C_{k,\eta}\Big\Vert \frac{\vert \nabla_x \vert}{\langle \nabla_x \rangle } (\hat{a},\hat{b},\hat{c}) \Big\Vert_{L^1_\xi L^2_T}.
\end{align}
For $J_5$ we have
\begin{align}\label{J51}
	J_5&\leq\int_{|\xi|\leq1} \Big(\int^T_0 (1+t)^{\sigma-1} \Vert (\mathbf{I}-\mathbf{P})\widehat{g_2}\Vert_{L^2_v}^2 dt \Big)^{1/2} d\xi+\int_{|\xi|\leq1} \Big(\int^T_0 (1+t)^{\sigma-1} \Vert \mathbf{P}\widehat{g_2}\Vert_{L^2_v}^2 dt \Big)^{1/2} d\xi\notag\\
	&\quad+\int_{|\xi|\geq1} \Big(\int^T_0 (1+t)^{\sigma-1} \Vert \widehat{g_2}\Vert_{L^2_v}^2 dt \Big)^{1/2} d\xi\notag\\
	&\leq J_3+J_4+\int_{|\xi|\leq1} \Big(\int^T_0 (1+t)^{\sigma-1} \Vert \mathbf{P}\widehat{g_2}\Vert_{L^2_v}^2 dt \Big)^{1/2} d\xi.
\end{align}
Since $J_3+J_4$ is bounded by \eqref{J3J4}, we only need to estimate the last term on the right hand side above. We still consider it in two cases. When  $\frac{1}{1+t}\leq \eta|\xi|^2$, it is straightforward to get
\begin{align}\label{J5larget}
	\int_{|\xi|\leq1} \Big(\int^T_0 (1+t)^{\sigma-1} \Vert \mathbf{P}\widehat{g_2}\Vert_{L^2_v}^2 dt \Big)^{1/2} d\xi&\leq \eta\int_{|\xi|\leq1} \Big(\int^T_0 (1+t)^{\sigma}|\xi|^2 \Vert \mathbf{P}\widehat{g_2}\Vert_{L^2_v}^2 dt \Big)^{1/2}d\xi\notag\\
	&\leq C\eta\Big\Vert (1+t)^{\si/2}\frac{\vert \nabla_x \vert}{\langle \nabla_x \rangle } (\hat{a},\hat{b},\hat{c}) \Big\Vert_{L^1_\xi L^2_T}.
\end{align}
On the other hand, when $\frac{1}{1+t}\leq \eta|\xi|^2$, by $(1+t)^{\sigma-1}\leq C_\eta |\xi|^{-2(\si-1)}$ we have
\begin{align*}
	\int_{|\xi|\leq1} \Big(\int^T_0 (1+t)^{\sigma-1} \Vert \mathbf{P}\widehat{g_2}\Vert_{L^2_v}^2 dt \Big)^{1/2} d\xi&\leq C_\eta\int_{|\xi|\leq1} |\xi|^{-\si} |\xi| \Big(\int^T_0\Vert \mathbf{P}\widehat{g_2}\Vert_{L^2_v}^2 dt \Big)^{1/2}d\xi.
\end{align*}
Then an application of H\"older's inequality shows that
\begin{align*}
	&\int_{|\xi|\leq1} \Big(\int^T_0 (1+t)^{\sigma-1} \Vert \mathbf{P}\widehat{g_2}\Vert_{L^2_v}^2 dt \Big)^{1/2} d\xi\notag\\\leq& C_\eta\big(\int_{|\xi|\leq1} |\xi|^{-p^\prime\si}d\xi\big)^\frac{1}{p^\prime} \big(\int_{|\xi|\leq1}|\xi|^p \Big(\int^T_0\Vert \mathbf{P}\widehat{g_2}\Vert_{L^2_v}^2 dt \Big)^{p/2}d\xi\big)^\frac{1}{p}.
\end{align*}
Recalling $\sigma=3\Big(1-\frac{1}{p}\Big) -2\ep$, we have $0<p'\si<3$, which leads to
\begin{align}\label{J5smallt}
	&\int_{|\xi|\leq1} \Big(\int^T_0 (1+t)^{\sigma-1} \Vert \mathbf{P}\widehat{g_2}\Vert_{L^2_v}^2 dt \Big)^{1/2} d\xi\leq C_\eta\Big\Vert \frac{\vert \nabla_x \vert}{\langle \nabla_x \rangle } (\hat{a},\hat{b},\hat{c}) \Big\Vert_{L^p_\xi L^2_T}.
\end{align}
Then we combine \eqref{J5larget} and \eqref{J5smallt} to get
\begin{align}\label{J52}
	&\int_{|\xi|\leq1} \Big(\int^T_0 (1+t)^{\sigma-1} \Vert \mathbf{P}\widehat{g_2}\Vert_{L^2_v}^2 dt \Big)^{1/2} d\xi\notag\\
	\leq& C\eta\Big\Vert (1+t)^{\si/2}\frac{\vert \nabla_x \vert}{\langle \nabla_x \rangle } (\hat{a},\hat{b},\hat{c}) \Big\Vert_{L^1_\xi L^2_T}+C_\eta\Big\Vert \frac{\vert \nabla_x \vert}{\langle \nabla_x \rangle } (\hat{a},\hat{b},\hat{c}) \Big\Vert_{L^p_\xi L^2_T}.
\end{align}
It follows from \eqref{J3J4}, \eqref{J51} and \eqref{J52} that
\begin{align}\label{J5}
	J_5&\leq C\eta \Vert (1+t)^{\sigma/2} (\mathbf{I}-\mathbf{P}) \widehat{g_2} \Vert_{L^1_\xi L^2_T H^{s*}_{v,k}}+C_k\eta\Big\Vert (1+t)^{\si/2}\frac{\vert \nabla_x \vert}{\langle \nabla_x \rangle } (\hat{a},\hat{b},\hat{c}) \Big\Vert_{L^1_\xi L^2_T}\notag\\&\quad +C_\eta \Vert  (\mathbf{I}-\mathbf{P}) \widehat{g_2} \Vert_{L^1_\xi L^2_T H^{s*}_{v,k+\si|\ga+2s|/2}}+C_{k,\eta}\Big\Vert \frac{\vert \nabla_x \vert}{\langle \nabla_x \rangle } (\hat{a},\hat{b},\hat{c}) \Big\Vert_{L^1_\xi L^2_T}\notag\\
	&\quad+C_\eta\Big\Vert \frac{\vert \nabla_x \vert}{\langle \nabla_x \rangle } (\hat{a},\hat{b},\hat{c}) \Big\Vert_{L^p_\xi L^2_T}.
\end{align}
By \eqref{J3J4}, \eqref{J5} and \eqref{g2ts}, we see that \eqref{g2timesoft} follows by selecting a sufficiently small $\eta$. The proof of Lemma \ref{lemmasoft} is complete.
\end{proof}

\section{Proof of Theorem \ref{decay}}
In order to make the proof clearer, we define the norms in terms of $g_1$ and $g_2$:
\begin{align*}
	&\|g_1\|_{\hat{X}^{h}_k}:=\|e^{\la t}\widehat{g_1}\|_{L^p_\xi L^\infty_TL^2_{v,k}},\qquad\|g_1\|_{\hat{X}^{h*}_k}:=\Vert e^{\la t}\widehat{g_1}\Vert_{L^p_\xi L^2_TH^s_{v,k+\gamma/2}},\notag\\
	&\|g_2\|_{\hat{Y}^{h}_k}:=\Vert \widehat{g_2}\Vert_{L^1_\xi L^\infty_T L^2_v}+\Vert \widehat{g_2}\Vert_{L^p_\xi L^\infty_T L^2_v}+\Vert  (1+t)^{\sigma/2} \widehat{g_2}\Vert_{L^1_\xi L^\infty_T L^2_v}\notag\\
	&\|g_2\|_{\hat{Y}^{h*}_k}:=\Vert (\mathbf{I}-\mathbf{P})\widehat{g_2} \Vert_{L^1_\xi L^2_T H^{s*}_v}+\Big\Vert \frac{\vert \nabla_x \vert}{\langle \nabla_x \rangle } (\hat{a},\hat{b},\hat{c}) \Big\Vert_{L^1_\xi L^2_T}\notag\\
	&\qquad\qquad\quad\ +\Vert (\mathbf{I}-\mathbf{P})\widehat{g_2} \Vert_{L^p_\xi L^2_T H^{s*}_v}+\Big\Vert \frac{\vert \nabla_x \vert}{\langle \nabla_x \rangle } (\hat{a},\hat{b},\hat{c}) \Big\Vert_{L^p_\xi L^2_T}\notag\\
	&\qquad\qquad\quad\ +\Vert (1+t)^{\sigma/2} (\mathbf{I}-\mathbf{P})\widehat{g_2} \Vert_{L^1_\xi L^2_T H^{s*}_v}+\Big \Vert (1+t)^{\sigma/2} \frac{\vert \nabla_x \vert}{\langle \nabla_x \rangle} (\hat{a},\hat{b},\hat{c}) \Big\Vert_{L^1_\xi L^2_T}.
\end{align*}
The definitions above are mainly for the case $\ga>0$. For soft potentials, we should define the other group of norms as follows:
\begin{align*}
	&\|g_1\|_{\hat{X}^{s}_k}:=\|\widehat{g_1}\|_{L^p_\xi L^\infty_TL^2_{v,k+\ga/2-\rho\ga}}+\|(1+t)^\rho \widehat{g_1}\|_{L^p_\xi L^\infty_TL^2_{v,k}},\notag\\
	&\|g_1\|_{\hat{X}^{s*}_k}:=\Vert  \widehat{g_1}\Vert_{L^p_\xi L^2_TH^s_{v,k+\gamma-\rho\ga}}+\Vert (1+t)^\rho \widehat{g_1}\Vert_{L^p_\xi L^2_TH^s_{v,k+\gamma/2}},\notag\\
	&\|g_2\|_{\hat{Y}^{s}_{k}}:=\Vert \widehat{g_2}\Vert_{L^1_\xi L^\infty_T L^2_{v,k+\si|\ga+2s|/2}}+\Vert \widehat{g_2}\Vert_{L^p_\xi L^\infty_T L^2_{k}}+\Vert  (1+t)^{\sigma/2} \widehat{g_2}\Vert_{L^1_\xi L^\infty_T L^2_{v,k}}\notag\\
	&\|g_2\|_{\hat{Y}^{s*}_{k}}:=\Vert (\mathbf{I}-\mathbf{P})\widehat{g_2} \Vert_{L^1_\xi L^2_T H^{s*}_{v,k+\si|\ga+2s|/2}}+\Big\Vert \frac{\vert \nabla_x \vert}{\langle \nabla_x \rangle } (\hat{a},\hat{b},\hat{c}) \Big\Vert_{L^1_\xi L^2_T}\notag\\
	&\qquad\qquad\quad\ +\Vert (\mathbf{I}-\mathbf{P})\widehat{g_2} \Vert_{L^p_\xi L^2_T H^{s*}_{v,k}}+\Big\Vert \frac{\vert \nabla_x \vert}{\langle \nabla_x \rangle } (\hat{a},\hat{b},\hat{c}) \Big\Vert_{L^p_\xi L^2_T}\notag\\
	&\qquad\qquad\quad\ +\Vert (1+t)^{\sigma/2} (\mathbf{I}-\mathbf{P})\widehat{g_2} \Vert_{L^1_\xi L^2_T H^{s*}_{v,k}}+\Big \Vert (1+t)^{\sigma/2} \frac{\vert \nabla_x \vert}{\langle \nabla_x \rangle} (\hat{a},\hat{b},\hat{c}) \Big\Vert_{L^1_\xi L^2_T}.
\end{align*}

\begin{proof}[Proof of Theorem \ref{decay}]
We divide our proof to three parts. When $0\leq\ga\leq1$, $g_1$ has exponential time decay in $L^p_\xi$ and $g_2$ has polynomial decay in $L^1_\xi\cap L^p_\xi$. When $\ga<0$ and $\ga+2s\geq0$, both $g_1$ and $g_2$ decay in polynomial time weight. When $\ga+2s<0$ and $\ga+2s>-1$,  $g_1$ and $g_2$ have polynomial decay in time with additional velocity weight. Now we go to the details of the proof.
For $k>22$, we first choose $A$ and $M$ in the definition of $L_B$ \eqref{defLB} as in Lemma \ref{Q x nonlinear estimate 1} such that Theorem \ref{GE} holds. Then we see the constants $A$ and $M$ now depend only on $k$.

\medskip
\noindent{\it Case 1.} $0\leq\ga\leq1$.

 Combining \eqref{esthatg1h} and \eqref{esthatg2}, it holds that
\begin{align}\label{decayhard}
	&\|g_1\|_{\hat{X}^{h}_k}+\|g_1\|_{\hat{X}^{h*}_k}+\|g_2\|_{\hat{Y}^{h}_0}+\|g_2\|_{\hat{Y}^{h*}_0}\notag\\
	\leq& C_k\big(\|\widehat{g_0}\|_{L^p_\xi L^2_{v,k}}+ \int^T_0 \Vert e^{\la t}g_1 \Vert^2_{X^*_{k+8+2s}}dt+(\int^T_0\|e^{\la t} g_1(t)\|^2_{X^*_{8-\ga/2}}dt^\frac{1}{2}\notag\\
	&\qquad+\|e^{\la t}\widehat{g_1}\|_{L^1_\xi L^2_T L^2_v}\big)+C_k(\|g_1\|_{\hat{X}^{h}_k}+\|g_1\|_{\hat{X}^{h*}_k}+\|g_2\|_{\hat{Y}^{h}_0}+\|g_2\|_{\hat{Y}^{h*}_0})^2.
\end{align} 
Note that we have
\begin{align}\label{g1L1xi}
	\Vert e^{\la t}\widehat{g_1} \Vert_{L^1_\xi L^2_TL^2_v }&=\int_{\R^3}\big(\int^T_0   \Vert e^{\la t}\widehat{g_1}(t,\xi) \Vert^2_{L^2_v} dt\big)^\frac{1}{2}d\xi\notag\\
	&\leq C\big(\int^T_0  \int_{\R^3} \langle\xi\rangle^{3+}\Vert e^{\la t}\widehat{g_1}(t,\xi) \Vert^2_{L^2_v} d\xi dt\big)^\frac{1}{2}\notag\\
	&\leq C\big(\int^T_0 \Vert e^{\la t}\widehat{g_1}(t,\xi) \Vert^2_{X^*_{8-\ga/2}}dt\big)^\frac{1}{2}.
\end{align}
Then it follows from \eqref{decayhard}, \eqref{g1L1xi}, \eqref{GEH} and the fact $8+2s<10$ that there exists a constant $\ep_0$ such that if
\begin{align*}
	\|\widehat{g_0}\|_{L^p_\xi L^2_{v,k}}+\|g_0\|_{X_{k+10}}\leq \ep_0,
\end{align*}
then
\begin{align*}
	&\|g_1\|_{\hat{X}^{h}_k}+\|g_1\|_{\hat{X}^{h*}_k}+\|g_2\|_{\hat{Y}^{h}_0}+\|g_2\|_{\hat{Y}^{h*}_0}\notag\\
	\leq& C_k\big(\|\widehat{g_0}\|_{L^p_\xi L^2_{v,k}}+\|g_0\|_{X_{k+10}}),
\end{align*}
which implies \eqref{decayh}.

\medskip
\noindent{\it Case 2.} $\ga<0$ and $\ga+2s>0$.

Similar arguments as in \eqref{decayhard} and \eqref{g1L1xi} show that by \eqref{esthatg1s}, \eqref{esthatg1soft}, \eqref{g2hardl1}, \eqref{g2hardlp} and \eqref{esthatg2}, one has
\begin{align}
	&\|g_1\|_{\hat{X}^{s}_k}+\|g_1\|_{\hat{X}^{s*}_k}+\|g_2\|_{\hat{Y}^{h}_0}+\|g_2\|_{\hat{Y}^{h*}_0}\notag\\
	\leq& C_k\big(\|\widehat{g_0}\|_{L^p_\xi L^2_{v,k+\ga/2-\rho\ga}}+ \int^T_0 \Vert (1+t)^\rho g_1 \Vert^2_{X^*_{k+\ga/2-\rho\ga+8+2s}}dt\notag\\
	&\qquad+(\int^T_0\|(1+t)^\rho g_1(t)\|^2_{X^*_{8-\ga/2}}dt)^\frac{1}{2}\big)\notag\\
	&+C_k(\|g_1\|_{\hat{X}^{s}_k}+\|g_1\|_{\hat{X}^{s*}_k}+\|g_2\|_{\hat{Y}^{h}_0}+\|g_2\|_{\hat{Y}^{h*}_0})^2.\notag
\end{align}
Hence, by choosing $\rho=3/2$, there exists a constant $\ep_0$ such that if
\begin{align*}
	\|\widehat{g_0}\|_{L^p_\xi L^2_{k-\ga}}+\|g_0\|_{X_{k+14}}\leq \ep_0,
\end{align*}
then
\begin{align}\label{decays1}
	&\|g_1\|_{\hat{X}^{h}_k}+\|g_1\|_{\hat{X}^{h*}_k}+\|g_2\|_{\hat{Y}^{h}_0}+\|g_2\|_{\hat{Y}^{h*}_0}\notag\\
	\leq& C_k\big(\|\widehat{g_0}\|_{L^p_\xi L^2_{k-\ga}}+\|g_0\|_{X_{k+14}}\big).
	\end{align}

\medskip
\noindent{\it Case 3.} $\ga+2s<0$ and $\ga+2s>-1$.
Similarly as above, by \eqref{esthatg1s}, \eqref{esthatg1soft}, \eqref{g2soft}, \eqref{g2lpsoft} and \eqref{g2timesoft}, we have
\begin{align}
	&\|g_1\|_{\hat{X}^{s}_k}+\|g_1\|_{\hat{X}^{s*}_k}+\|g_2\|_{\hat{Y}^{s}_{k,j}}+\|g_2\|_{\hat{Y}^{s*}_{k,j}}\notag\\
	\leq& C_k\|\widehat{g_0}\|_{L^p_\xi L^2_{v,k+\ga/2-\rho\ga}}+ C_k\int^T_0 \Vert (1+t)^\rho g_1 \Vert^2_{X^*_{k+\ga/2-\rho\ga+8+2s}}dt\notag\\
	&+C_k\big(\int^T_0\|(1+t)^\rho g_1(t)\|^2_{X^*_{k+\si|\ga+2s|/2+8-\ga/2}}dt\big)^\frac{1}{2}\notag\\
	&+C_k(\|g_1\|_{\hat{X}^{s}_k}+\|g_1\|_{\hat{X}^{s*}_k}+\|g_2\|_{\hat{Y}^{s}_{k}}+\|g_2\|_{\hat{Y}^{s*}_{k}})^2.\notag
\end{align}
Hence, choosing $\rho=3/2$, by the fact that
$$
\min\{-\ga+8+2s,\si|\ga+2s|/2+8-\ga/2\}\le 14,
$$
there exists a constant $\ep_0$ such that if
\begin{align*}
	\|\widehat{g_0}\|_{L^p_\xi L^2_{k-\ga}}+\|g_0\|_{X_{k+14}}\leq \ep_0,
\end{align*}
then
\begin{align}\label{decays2}
	&\|g_1\|_{\hat{X}^{s}_k}+\|g_1\|_{\hat{X}^{s*}_k}+\|g_2\|_{\hat{Y}^{s}_{k,j}}+\|g_2\|_{\hat{Y}^{s*}_{k,j}}\notag\\
	\leq& C_k\big(\|\widehat{g_0}\|_{L^p_\xi L^2_{k-\ga}}+\|g_0\|_{X_{k+14}}\big).
\end{align}
Hence, we obtain \eqref{decays} by \eqref{decays1} and \eqref{decays2}. The proof of Theorem \ref{decay} is complete.
\end{proof}

\medskip
\noindent {\bf Acknowledgments:}\, CQC was partially supported by Research Centre for Nonlinear Analysis, Hong Kong Polytechnic University. RJD was partially supported by the General Research Fund (Project No.~14303321) from RGC of Hong Kong and a Direct Grant from CUHK. ZGL was supported by the Hong Kong PhD Fellowship Scheme (HKPFS).


\end{document}